\documentclass[12pt]{article}
\usepackage[utf8]{inputenc}

\usepackage[a4paper,margin=2cm]{geometry}
\usepackage[english]{babel}
\usepackage{latexsym,amsmath,enumerate,graphics,enumerate,amsthm,tikz,hyperref,float}
\usepackage[mathscr]{euscript}
\usepackage[affil-it]{authblk}
\usepackage{enumerate,amsthm,dsfont,pstricks, ulem}
\usepackage{latexsym,amsmath,amssymb,amscd,wrapfig,graphicx}
\usepackage{hyperref,cleveref}

\newtheorem{theorem}{Theorem}[section]
\newtheorem{lemma}[theorem]{Lemma}
\newtheorem{proposition}[theorem]{Proposition}
\newtheorem{corollary}[theorem]{Corollary}

\theoremstyle{definition}
\newtheorem{definition}[theorem]{Definition}
\newtheorem{example}[theorem]{Example}

\theoremstyle{remark}
\newtheorem{remark}[theorem]{Remark}
\newtheorem*{notation*}{{\bf Notation}}

\newtheorem*{theorem*}{{\bf Theorem}}

\newtheorem*{assumption*}{{\bf Assumption}}

\let\phi=\varphi
\def\N{\mathbb{N}}
\def\R{\mathbb{R}}
\def\H{\mathbb{H}}
\def\O{\mathbb{O}}
\def\C{\mathbb{C}}

\def\Ci{C^\circ}

\def\eps{\varepsilon}

\def\ol{\overline}

\newcommand{\comment}[1]{}
\newcommand{\norm}[1]{\left\Vert #1 \right\Vert}

\newcommand{\ip}[2]{\left\langle #1 , #2 \right\rangle}
\newcommand{\mat}[4]{\begin{pmatrix} #1 & #2 \\ #3 & #4 \\ \end{pmatrix} }

\def\sqr#1#2{{\,\vcenter{\vbox{\hrule height.#2pt\hbox{\vrule width.#2pt
height#1pt \kern#1pt\vrule width.#2pt}\hrule height.#2pt}}\,}}
\def\square{\sqr52\,}

\numberwithin{equation}{section}

\let\epsilon=\varepsilon

\usepackage[toc]{appendix}

\makeatletter
\def\@maketitle{%
  \newpage
  \null
  \vskip 2em%
  \begin{center}%
  \let \footnote \thanks
    {\Large\bfseries \@title \par}%
    \vskip 1.5em%
    {\normalsize
      \lineskip .5em%
      \begin{tabular}[t]{c}%
        \@author
      \end{tabular}\par}%
    \vskip 1em%
    {\normalsize \@date}%
  \end{center}%
  \par
  \vskip 1.5em}
\makeatother

\title{\sc \huge Infinite dimensional symmetric cones and gauge-reversing maps}

\author{Bas Lemmens%
\thanks{Email: \texttt{B.Lemmens@kent.ac.uk}, supported by the EPSRC (grant EP/R044228/1)}}
\affil{School of Mathematics, Statistics \& Actuarial Science, University of Kent, Canterbury, Kent
CT2 7NX, UK}

\author{Mark Roelands%
\thanks{Email: \texttt{m.roelands@math.leidenuniv.nl}, partially supported by the EPSRC (grant EP/R044228/1)}}
\affil{Mathematical Institute, Leiden University, 2300 RA Leiden,
The Netherlands}

\author{Marten Wortel%
\thanks{Email: \texttt{marten.wortel@up.ac.za}}}
\affil{Department of Mathematics and Applied Mathematics, University of Pretoria, Private Bag X20
Hatfield, 0028 Pretoria, South Africa}
\begin{document}

\maketitle
\vspace{-4mm}
\date{}
\vspace{-10mm}
\begin{abstract}
  The famous Koecher–Vinberg theorem characterises the finite dimensional formally real Jordan algebras among the finite dimensional order unit spaces as the ones that have a symmetric cone. An alternative characterisation of symmetric cones was obtained by Walsh who showed that the symmetric cones correspond exactly to the finite dimensional order unit spaces with cone $C$ for which there exists a bijective map $\Psi\colon C^\circ\to C^\circ$ with the property that $\Psi$ is an order-antimorphism and homogeneous of degree $-1$. In this paper we prove an infinite dimensional version of this characterisation of symmetric cones.
\end{abstract}


\vspace{-4mm}
{\footnotesize\tableofcontents}

\section{Introduction}

A {\em real Jordan algebra} is a real vector space $A$ equipped with a commutative bilinear product $x \bullet y$ satisfying the Jordan identity $ x^2 \bullet( x\bullet y) = x\bullet (x^2 \bullet y)$. It is said to be {\em formally real} if $x_1^2+\cdots +x_k^2=0$ implies $x_1=\ldots =x_k=0$ for $k\in\mathbb{N}$. The concept of a formally real Jordan algebra has a rich history in mathematics. It was  originally introduced by P. Jordan, J. von Neumann and E. Wigner \cite{JvNW}  as an algebraic model for quantum mechanics, but soon after unexpected connections with Lie algebras, differential geometry and analysis were discovered (cf.\,\cite{AS2,Chu,FK,HO,Loos,Mc,Up}).

A fundamental connection between formally real Jordan algebras and the geometry of cones was discovered independently by Koecher \cite{Koe} and Vinberg \cite{Vin}. They showed that any real finite dimensional Hilbert space $A$ with a {\em symmetric cone} $A^\circ_+$, i.e., the closed cone $A_+$ is self-dual and the group of linear automorphisms of $A_+$ acts transitively on  $A_+^\circ$, is a formally real Jordan algebra where $A_+$ corresponds to the cone of squares.  Conversely, a finite dimensional formally real Jordan algebra $A$ can be endowed with an inner product so that the interior of the cone of squares is a symmetric cone.

The Koecher-Vinberg Theorem provides a striking link between real Jordan algebras and differential geometry. More specifically, symmetric cones are prime examples of Riemannian symmetric spaces of non-compact type. In fact, the (Riemannian) symmetry at $x\in A_+^\circ$ is given by 
 $S_x(y) := Q_x(y^{-1})$ for $y\in A_+^\circ$, where $Q_x$ is the {\em quadratic representation of $x$} given by  $Q_x(z):= 2x\bullet (x\bullet z) - x^2\bullet z$.  The symmetry $S_e$ at the unit $e$ is the inverse map $\iota \colon x\mapsto x^{-1}$ on $A_+^\circ$, which  has  a special order theoretic property.  Indeed, the cone $A_+$ induces a partial ordering $\leq$ on $A$ by $x\leq y$
if $y-x\in A_+$.  With respect to this partial ordering the inverse map $\iota$ is an order-antimorphism, i.e., $x\leq y$  if and only if  $\iota(y)\leq \iota(x)$.   Moreover, $\iota$ is homogeneous of degree $-1$ in the sense that  $\iota (\lambda x) =\lambda^{-1}\iota(x)$ for all $\lambda >0$ and $x\in A^\circ_+$.  

In finite dimensions there is a remarkable result due to Walsh \cite{Wa1}, which says that given a closed cone $C$ (with non-empty interior) in a real vector space, $C^\circ$ is a symmetric cone (with respect to some inner product) if and only if there exists a homogeneous degree $-1$ bijective map $\Psi\colon C^\circ\to C^\circ$ that is an order-antimorphism. This result provides a completely order theoretic characterisation of finite dimensional symmetric cones. Note that in possibly infinite dimensional Hilbert spaces, the definition of a symmetric cone given above makes perfect sense, and so the main goal of this paper is to establish an infinite dimensional version of Walsh' result, see Theorem \ref{T:char of JH-algebras}.

The natural infinite dimensional generalisation of formally real Jordan algebras are so-called JB-algebras, which were introduced by Alfsen, Schultz and St\o rmer \cite{ASS}. A {\em JB-algebra} $A$ is a real Banach space with a Jordan algebra product $x \bullet y$ and unit $e$ (note that we define a JB-algebra to be unital, and some authors do not make this assumption), where the norm satisfies:
\[ \|x \bullet y\|\leq \|x\|\|y\|,\quad \|x^2\|=\|x\|^2,\mbox{\quad and\quad  }\|x^2\|\leq \|x^2+y^2\|.\]
Important examples include the self-adjoint parts of  $C^*$-algebras and the self-adjoint matrices $M_n(\R)_{sa}$, $M_n(\H)_{sa}$ or $M_3(\O)_{sa}$ all equipped with the Jordan product $x \bullet y := (xy + yx)/2$. 
JB-algebras are very special examples of so-called complete order unit spaces, see \cite{AS1}. 
(Definitions will be recalled in the next section.) However, not all infinite dimensional  JB-algebras $A$ can be equipped with an inner-product that turns it into a Hilbert space and with respect to which the interior of the cone of squares $A_+$ becomes a symmetric cone.  It was shown by Chu \cite[Theorem~3.1]{chu1} that Hilbert spaces with a symmetric cone are precisely the so-called unital JH-algebras. A JH-algebra is a real Jordan algebra $A$ that can be endowed with an inner-product $\langle \cdot,\cdot\rangle$, making it a Hilbert space, such that $\langle x\bullet y, z \rangle = \langle y, x\bullet z \rangle$ for all $x,y,z \in A$.  As we shall see in the next section these unital JH-algebras are JB-algebras for the equivalent order unit norm.

Thus, to characterise the symmetric cones among the possibly infinite dimensional complete order unit spaces in a similar way as Walsh's finite dimensional one, additional assumptions need to be made on the order unit space or on its cone. In this paper we establish such an infinite dimensional characterisation if one assumes either reflexivity or a condition on the cone closely related to reflexivity, cf.\ \Cref{T:char of JH-algebras}. We mention that in an earlier work by the first two authors and H. van Imhoff \cite{LRvI} the case where the cone in the complete order unit space is strictly convex was considered. In that case it was shown  that there exists a homogeneous degree $-1$ bijective map that is an order-antimorphism if and only if the order unit space is a spin factor, which form a particular class of JH-algebras.

The proof contains two important steps. Firstly, we show in Section \ref{sec:hom} that the cone is homogeneous, i.e., the group $\mathrm{Aut}(C)$ of linear automorphisms of $C$ acts transitively on $C^\circ$. In fact, to establish that result we only need to assume that the order unit space is spanned by the extreme vectors of the cone. Subsequently we will use ideas from metric geometry, involving horofunctions, in Section \ref{sec:sym} to construct an inner-product on the order unit space, and show that the cone is self-dual. Basically, the horofunctions allow one to study the action of the order-antimorphism on the boundary at infinity, which can be connected with the dual space of the order unit space. The relevant results about horofunctions will be recalled in Sections \ref{sec:hor} and \ref{sec:funk}, many of which were established by Walsh in \cite{Wa1,Wa2}. But we begin by collecting the basic concepts concerning order unit spaces.

\section{Preliminaries}
\subsection{Order unit spaces}
Firstly, we fix the terminology and recall the basic concepts concerning order unit spaces. 

Let $V$ be a real vector space. A subset $C$ of $V$ is called a {\em cone} if $C$ is convex, $\lambda C \subseteq C$ for all $\lambda\geq 0$, and $C\cap -C =\{0\}$. The cone $C$ induces a partial ordering on $V$ by, $v\leq w$ if $w-v\in C$, which turns it into a partially ordered vector space. The cone is said to be {\em Archimedean} if for each $v\in V$ and $w\in C$ we have that $nv\leq w$ for all $n \in \N$ implies that $v\leq 0$.  It is said to be {\em almost Archimedean} if $-x/n\leq y\leq x/n$ for all $n \in \N$ and $x,y\in V$ implies that $y=0$. A straightforward verification shows that every Archimedean cone is almost Archimedean.

An element $u\in C$ is said to be an {\em order unit} if for each $v\in V$ there exists $\lambda\geq 0$ such that $-\lambda u\leq v\leq\lambda u$.  
The triple $(V,C,u)$ is said to be an {\em order unit space} if $C$ is an Archimedean cone in $V$ with order unit $u$.  In that case $V$ can be equipped with the so-called {\em order unit norm}, 
\[
\|v\|_u := \inf\{\lambda\geq 0\colon -\lambda u\leq v\leq\lambda u\}\mbox{\quad for }v\in V.
\]
With respect to the order unit norm the cone $C$ is closed and has a non-empty interior, denoted $C^\circ$, and $u\in C^\circ$. Moreover, each $w\in C^\circ$ is an order unit for $(V,C)$. Furthermore, the boundary of $C$ is denoted by $\partial C$.
Throughout the paper we almost always work with order unit spaces. 

The dual space $V^*$ of an order unit space $V$ is an ordered vector space with {\em dual cone} $C^* := \{\phi\in V^*\colon \phi(v)\geq 0\mbox{ for all }v\in C\}$. The elements of $C^*$ are called {\em positive} linear functionals. It is known \cite[Lemma 1.16]{AS1} that a linear functional $\phi$ on an order unit space $V$ is positive if and only if it is bounded with $\|\phi\| = \phi(u)$.  A positive linear functional $\phi$ on an order unit space $V$ is called a {\em state} if $\phi(u) =1$. The set of all states is called the {\em state space} of $V$ and will be denoted by $S=S(V)$.  The state space $S$ is a $w^*$-compact convex set, and hence it is the $w^*$-closure of the convex hull of its extreme points by the Krein-Milman theorem. The extreme points  of $S$ are called {\em pure} states. A useful fact that will be used several times in the paper is that the states determine the norm and ordering on $V$, see \cite[Lemma~1.18]{Alfsen}.

Given $v\leq w$ in $V$ we write $[v,w] := \{z\in V\colon v\leq z\leq w\}$ to denote the {\em order interval}. A subset $A$ of $V$ is said to be {\em full } if $[v,w]\subseteq A$ for all $v,w\in A$ with $v\leq w$. We call a non-zero vector $r\in C$ an {\em extreme vector of} $C$ if $0 \leq v\leq r$ implies that $v=\lambda r$ for some $\lambda\geq 0$. 

We will also need the following lemma.

\begin{lemma}\label{L:totally ordered extreme}
Let $(V,C)$ be a partially ordered vector space and $x\leq y$ be distinct points in $V$. Consider the following statements.
\begin{itemize}
\item[$(i)$] $[x,y]$ is totally ordered;
\item[$(ii)$] $[x,y]=\{ (1-t)x + ty  \colon 0\le t\le 1\}$;
\item[$(iii)$] $y-x$ is an extreme vector of $C$.
\end{itemize}
Then $(iii)\Leftrightarrow(ii)\Rightarrow (i)$, and $(i)\Rightarrow (iii)$ if $(V,C)$ is almost Archimedean.
\end{lemma}
\begin{proof}
 Since translation by $-x$ is an order-isomorphism, we may assume that $x=0$. If $y$ is an extreme vector of $C$, then $[0,y]=\{ty \colon 0\le t\le 1\}$, which yields $(ii)$. If $[0,y]=\{ty\colon 0\le t\le 1\}$, then $[0,y]$ is totally ordered, as comparison of elements now only depends on the parameter $0\le t\le 1$, and so $y$ is an extreme vector of $C$. This shows  $(iii)\Leftrightarrow(ii)\Rightarrow (i)$.

For the last implication, suppose that $(V,C)$ is almost Archimedean and $[0,y]$ is totally ordered. Let $z\in[0,y]$. For each $0\le \lambda\le 1$ it follows that $z$ and $\lambda y$ are comparable and define $\mu := \sup\{\lambda\ge 0\colon \lambda y \le z\} = \inf\{\lambda\le 1\colon \lambda y \ge z\}$. Then $(\mu - \frac{1}{n})y \leq z \leq (\mu + \frac{1}{n})y$ and so $-\frac{1}{n}y \leq z- \mu y \leq \frac{1}{n}y$. Since $(V,C)$ is almost Archimedean, $z- \mu y = 0$. So $z = \mu y$ showing that $y$ is an extreme vector of $C$.
\end{proof}

\begin{definition}
Given $x\in C^\circ$ and $w\in V$ we define $\ell_x^w := \{x+tw\colon t\in\mathbb{R}\} \cap C^\circ$. A subset $L$ of $C^\circ$ is called an {\em extreme half-line} if there exists an $x \in C^\circ$ and and extreme vector $r$ of $C$ such that $L = \ell_x^r$.    
\end{definition}

Note that, given $x\in C^\circ$ and $w\in V$, we have $\ell_x^w = \ell_x^{-w}$.  The following lemma gives an order-theoretic characterisation of extreme half-lines and is similar to \cite[Proposition 3.1]{LvIvG}. 

\begin{lemma}\label{L:extreme lines}
Let $(V,C,u)$ be an order unit space and $L\subseteq C^\circ$. Then $L$ is an extreme half-line if and only if $L$ contains at least two points and is maximal among the subsets of $C^\circ$ that are totally ordered and full.
\end{lemma}
\begin{proof}
Suppose that $L\subseteq C^\circ$ contains distinct points $x$ and $y$ and is maximal among the sets in $C^\circ$ that are totally ordered and full. We first show that $L\subseteq \ell_x^{y-x}$. If $z\in L\setminus\{x,y\}$, then there are six options describing how $x,y$, and $z$ are ordered. Since it is analogous to show that $z\in \ell_x^{y-x}$ for all these cases, we shall only cover the situation where $z\le y\le x$. As $[y,x]\subseteq L$ is totally ordered, it follows from \Cref{L:totally ordered extreme} that $x-y$ is an extreme vector of $C$, so $\ell_x^{y-x}$ is an extreme half-line. Furthermore, as $[z,x]\subseteq L$ is totally ordered as well, we have that $y=sz+(1-s)x$ for some $0<s\le 1$ by Lemma~\ref{L:totally ordered extreme}. Hence $z=x+\frac{1}{s}(y-x)\in\ell_x^{y-x}$. As any extreme half-line in $C^\circ$ is totally ordered and full by Lemma~\ref{L:totally ordered extreme}, it follows from the maximality of $L$ that $L=\ell_x^{y-x}$. 

For the converse, suppose $L$ is an extreme half-line. Then there exists $x \in C^\circ$ and an extreme vector $r$ of $C$ such that $L=\ell_x^r$, and so $L$ contains two distinct points. Let $S \subseteq C^\circ$ be a full and totally ordered subset that contains $\ell_x^r$. If $y\in S$ is different from $x$, then $y\leq x$, $x\leq y\le x+r$, or $x+r\le y$.  If $y\leq x$, then $x\in [y,x+r]$ and $[y,x+r]$ is a subset of $S$, as $S$ is full, so it is totally ordered. So, by Lemma~\ref{L:totally ordered extreme}, $x = \lambda y +(1-\lambda)(x+r)$ for some $0<\lambda \leq 1$, and hence $y = \lambda^{-1}x +(1-\lambda^{-1})(x+r) = x+(1-\lambda^{-1})r$.   On the other hand, if $x\leq y\le x+r$, then $y=\sigma x+(1-\sigma)(x+r)=x+(1-\sigma)r$ for some $0\le \sigma \le 1$ by Lemma~\ref{L:totally ordered extreme}, as $[x,x+r]\subseteq S$ is totally ordered. Likewise, if $x+r\le y$, then $x+r\in [x,y]$ and we have $x+r=\tau x+(1-\tau)y$ for some $0\le \tau<1$, which implies that $y=x+(1-\tau)^{-1}r$. We conclude that $S=\ell_x^r=L$ and so $L$ is maximal among the sets in $C^\circ$ that are full and totally ordered.
\end{proof}
We conclude this subsection with the following useful fact. 

\begin{lemma}\label{L:M function is continuous}
Let $(V,C,u)$ be an order unit space with state space $S$. Then $T\colon V\times C^\circ\to \mathbb{R}_+$ given by $T(x,y)=M(x/y)$ is continuous.
\end{lemma}
\begin{proof}
If we equip $S$ with the $w$*-topology, the map $\psi\colon (V,\|\cdot\|_u)\to (C(S),\|\cdot\|_\infty)$ defined by $\psi(x)(\varphi) = \varphi(x)$ is a
bi-positive isometric embedding by \cite[Lemma~1.18]{Alfsen}. Note that the composition of continuous functions
\begin{align*}
(x, y) &\mapsto (\psi(x), \psi(y)) \mapsto (\psi(x), \psi(y)^{-1}) \mapsto \psi(x)\psi(y)^{-1}
\\&\mapsto \|\psi(x)\psi(y)^{-1}\|_\infty=\sup_{S}\frac{\varphi(x)}{\varphi(y)}= M(x/y)
\end{align*}
implies that $T$ is continuous as required. 
\end{proof}

\subsection{Gauge-reversing maps on cones in order unit spaces}
Let $(V,C,u)$ and $(W,K,e)$ be order unit spaces. A map $\Psi\colon C^\circ\to K^\circ$ is said to be {\em antitone} if  $v,w\in C^\circ$ with $v\leq w$ implies  $\Psi(w)\leq \Psi(v)$. It is said to be an {\em order-antimorphism} if $\Psi\colon C^\circ\to K^\circ$ is a bijective antitone map whose inverse is antitone.  The map $\Psi\colon C^\circ\to K^\circ$ is said to be {\em homogeneous of degree $-1$} if $\Psi(\lambda v) =\lambda^{-1}\Psi(v)$ for all $v\in C^\circ$ and $\lambda>0$.  Order-antimorphisms $\Psi\colon C^\circ\to K^\circ$ which are homogeneous of degree $-1$ correspond to so-called gauge-reversing maps. 

Given $v\in V$ and $w\in C^\circ$ we define 
\begin{equation}\label{gauge}
M(v/w) := \inf\{\mu \in\mathbb{R}\colon v\leq \mu w\}.
\end{equation}
Note that $M(v/w)<\infty$, as $w$ is an order unit. The $M$-function is called a {\em gauge}. A bijective map $\Psi\colon C^\circ\to K^\circ$ is said to be {\em gauge-reversing} if $M(v/w) = M(\Psi(w)/\Psi(v))$ for all $v,w\in C^\circ$. 
The following result is well-known \cite{NS}. 
\begin{lemma}\label{L:antitone M-functions}
Let $(V,C,u)$ and $(W,K,e)$ be order unit spaces and $\Psi\colon C^\circ\to K^\circ$ a bijection. Then $\Psi$ is an order-antimorphism and homogeneous of degree $-1$ if and only if $\Psi$ is gauge-reversing. 
\end{lemma}
\begin{proof}
Suppose that $ \Psi\colon C^\circ\to K^\circ$  is a bijective order-antimorphism and  homogeneous of degree $-1$. If  $v\le \lambda w$ in $C^\circ$, then $\lambda^{-1}\Psi(w)= \Psi(\lambda w) \leq \Psi(v)$, so that $\Psi(w) \leq \lambda \Psi(v)$. This implies that $M(\Psi(w)/\Psi(v))\le M(v/w)$. On the other hand, $\Psi^{-1}$ is also an order-antimorphism and  homogeneous of degree $-1$, so $M(w/v)\le M(\Psi(v)/\Psi(w))$, and hence $\Psi$ is gauge-reversing. 

Conversely, suppose $\Psi\colon C^\circ\to K^\circ$ is a gauge-reversing bijection. If $v \le w$ in $C^\circ$, then $M(\Psi(w)/\Psi(v)) =
M(v/w) \le 1$, so that $\Psi(w) \le \Psi(v)$. Likewise $\Psi(w) \le \Psi(v)$ implies $M(v/w) = M(\Psi(w)/\Psi(v)) \le 1$,
so that $v\le w$, which shows that $\Psi$ and $\Psi^{-1}$ are antitone.  To see that $\Psi$ is homogeneous of degree $-1$,
let $v \in C^\circ$ and $\lambda>0$ and consider $w = \lambda v$. It follows that $M(\Psi(w)/\Psi(v)) = M(v/w) = \lambda^{-1}$ and 
$M(\Psi(v)/\Psi(w)) = M(w/v) = \lambda$. Hence $\lambda \Psi(w) \le \Psi(v) \le \lambda \Psi(w)$, so $\Psi(\lambda v) = \Psi(w)= \lambda^{-1}\Psi(v)$.
\end{proof}

We point out that  in the same vein it can be shown that a bijective map $f\colon C^\circ\to K^\circ$ is an order-isomorphism and homogeneous of degree $1$ if and only if $f$ is gauge-preserving, see \cite{NS}. In fact, it was shown by Sch\"affer  \cite[Theorem B]{Schaf} that if $V$ and $W$ are order unit spaces, every bijective map $f\colon C^\circ\to K^\circ$ which is an order-isomorphism and is homogeneous of degree $1$ must be linear.  We will use this result later in the paper. 

Lemma~\ref{L:extreme lines} has the following direct corollary for gauge-reversing maps.
\begin{corollary}\label{C:extreme lines -> extreme lines}
If $(V,C,u)$ and $(W,K,e)$ are order unit spaces and $\Psi\colon C^\circ\to K^\circ$ is a gauge-reversing map, then $\Psi$ maps extreme half-lines onto extreme half-lines.  
\end{corollary}

\subsection{JBW-algebras}

A \emph{JBW-algebra} can be defined as a JB-algebra that is a dual space; it is the Jordan analogue of a von Neumann algebra. A finite type $I$ JBW-factor is a JBW-algebra isomorphic to the self-adjoint matrices $M_n(\R)_{sa}$, $M_n(\C)_{sa}$, $M_n(\H)_{sa}$, $M_3(\O)_{sa}$, or a spin factor $\R \oplus H$ where $H$ is a real Hilbert space and $(\lambda, x) \bullet (\mu,y) := (\lambda \mu + \ip{x}{y}_H, \lambda y + \mu x)$; all these algebras are also JH-algebras, where the matrix algebras are equipped with the inner product $\ip{A}{B} := \mathrm{Tr}(A \bullet B)$. For more details about these concepts and unexplained terminology in the next proof we refer to \cite[Chapters~2~and~3]{AS2}. 

Recall that a JH-algebra $A$ is a real Jordan algebra with an inner-product, for which it is a Hilbert space, such that $\ip{x \bullet y}{z} = \ip{y}{x \bullet z}$ for all $x,y,z \in A$. It turns out that unital JH-algebras are JB-algebras for the equivalent order unit norm. This relies on the spectral decomposition for elements in JH-algebras, and on page 41 in \cite{No} a sketch of the proof for this result is given. For sake of clarity, we will work this out in more detail. 

\begin{lemma}\label{L:orthogonal basis in associtive JH-algebras}
    Let $A$ be an associative JH-algebra such that the map $L \colon A \to B(A)$ given by $L(x)y := x \bullet y$ is injective. Then there exists an orthogonal basis of idempotents $(p_k)_k$ such that $A = \R p_k \oplus A_0$ as an orthogonal algebra direct sum for all $k$ and $A = \ell_2 - \bigoplus_k \R p_k$.
\end{lemma}

\begin{proof}
    If $A$ is an associative JH-algebra and $L$ is injective, then $A$ is a flat JH-triple such that $L(A,A) \neq 0$, which meets the requirements of the starred remark in \cite[Theorem~IV.2.4]{Ne}. Here $L(x,y) = x\square y := L(x\bullet y) +[L(x),L(y)]$, where $[T,S] := TS-ST$, and the triple product is given by $\{x,y,z\} := (x\square y)(z)$. Hence there exists a nonzero generalised tripotent $c$ in $A$, that is $c^3 = \pm c$, such that $A = \R c \oplus A'$ as an orthogonal algebra direct sum. Note that $0 < \ip{c}{c} = \pm \ip{c^3}{c} = \pm \ip{c^2}{c^2}$ forces $c^3 = c$, and $p := c^2$ is a nonzero idempotent for which we also have $A = \R p \oplus A'$ as an orthogonal algebra direct sum. Now consider a maximal set of orthogonal idempotents $(p_k)_k$ such that $A = \R p_k \oplus A_{k}$ for all $k$. It follows that $A_0 := \cap_{k}A_k$ is an associative JH-algebra, and if $A_0 \neq \{0\}$, then it would contain a nonzero idempotent by the argument above as $L$ is also injective when restricted to $A_0$, which contradicts the maximality of $(p_k)_k$. We conclude that $A_0 = \{0\}$, and if $x \in A$ is such that $\ip{x}{p_k} = 0$ for all $k$, then $x \in A_0$ so $x = 0$. Hence $(p_k)_k$ is an orthogonal basis for $A$ and we have that $A = \ell_2 - \bigoplus_k\R p_k$.  
\end{proof}

\begin{theorem}\label{T: JH-algebra is JB-algebra}
    Let $A$ be a JH-algebra with unit $e$. Then $e$ is an order unit and the order unit norm is an equivalent norm that turns $A$ into a JB-algebra. 
\end{theorem}

\begin{proof}
    Since $A$ is unital, the map that sends $x \in A$ to the left multiplication operator $L(x)y := x \bullet y$ is injective, since if $L(x) = 0$, then $0 = L(x)e = x$. Hence the map $L \colon A \to B(A)$ defined by $x \mapsto L(x)$ is nonzero and bounded by \cite[Lemma~1.1]{No}. Let $x \in A$ and consider the JH-algebra generated by $x$ and $e$ denoted by JH$(x,e)$. Since $L$ is bounded, it follows that JH$(x,e)$ is an associative JH-algebra. Furthermore, the restriction of $L$ to JH$(x,e)$ is injective, so by \Cref{L:orthogonal basis in associtive JH-algebras} there is an orthogonal basis of idempotents $(p_k)_k$ such that $\mathrm{JH}(x,e) = \ell_2 - \bigoplus_k \R p_k$. For each $p_k$ it follows that 
    \[
    \|p_k\| = \|L(p_k)p_k\| \le \|L(p_k)\|\|p_k\| \le \|L\|\|p_k\|^2,
    \]
    so that $\|p_k\| \ge \|L\|^{-1}$. If $(p_k)_k$ would contain a countably infinite set $p_1, p_2, \dots$, then for any $n \in \N$ we would have for $P_n := p_1 + \dots + p_n$ that 
    \[
    \ip{P_n}{P_n} = \ip{P_n}{P_n} + \ip{P_n}{e-P_n} = \ip{P_n}{e} \le \ip{e}{e}, 
    \]
    so that $ n\|L\|^{-2}\le \sum_{k=1}^n \|p_k\|^2 = \|P_n\|^2 \le \|e\|^2$ for all $n \in \N$, a contradiction. Hence there is an $n \in \N$ such that $(p_k)_k = \{p_1, \dots, p_n\}$, $p_1 + \dots + p_n = e$, and $x = \lambda_1 p_1 + \dots + \lambda_n p_n$ for some $\lambda_k \in \R$. By \cite[Lemma~1.4]{No} we have that $p \bullet q = 0$ for any two orthogonal idempotents and the identity in the proof of this statement also shows that $\ip{p}{q} \ge 0$ for any two idempotents $p$ and $q$ in $A$. This implies that $x \in A$ is a square if and only if $\ip{x}{y^2} \ge 0$ for all $y \in A$. Hence the set of squares is an Archimedean cone in $A$ and $e$ is an order unit, as 
    \[
    -\Bigl(\max_{1 \le k \le N} \lambda_k\Bigr) e \le x \le \Bigl(\max_{1 \le k \le N} \lambda_k\Bigr) e.
    \]
    If we denote the order unit norm by $\|\cdot\|_e$, observe that 
    \[
    \|x\|^2 = \sum_{k=1}^N \lambda_k^2\|p_k\|^2 \le \left(\max_{1 \le k \le N}\lambda_k^2 \right)\sum_{k=1}^N\|p_k\|^2 = \|e\|^2 \max_{1 \le k \le N}\lambda_k^2 = \|e\|^2 \norm{x}_e^2,
    \]
    so that $\|x\| \le \|e\|\|x\|_e$. On the other hand, by the Cauchy-Schwarz inequality we have that
    \[
    \|x\|_e = \max_{1 \le k \le N}|\lambda_k| = \max_{1 \le k \le N}|\ip{x}{p_k}| \le \|x\|\max_{1 \le k \le N}\|p_k\| \le \|e\|\|x\|.
    \]
    Hence the norms are equivalent. Moreover, if $-e \le x \le e$, then $-1 \le \lambda_k \le 1$ for all $k$, and so $\lambda_k^2 \le 1$ for all $k$, thus $0 \le x^2 \le e$. By \cite[Theorem~1.11]{AS2} $A$ is a JB-algebra for the order unit norm.   
\end{proof} 

\Cref{T: JH-algebra is JB-algebra} is used in the following characterisation of unital JH-algebras.

\begin{theorem}\label{T:char JH-algebras as reflexive JB}
Let $A$ be a unital real Jordan algebra with unit $e$. Then the following are equivalent:
\begin{enumerate}
\item $A$ can be equipped with an inner product under which it is a JH-algebra.
\item $(A, \norm{\cdot}_e)$ is a reflexive JB-algebra.
\item $(A, \norm{\cdot}_e)$ is a JBW-algebra not containing an infinite sequence $(p_n)_n$ of orthogonal projections.
\item $A$ is a finite direct sum of finite type $I$ JBW-factors.
\end{enumerate}
\end{theorem}

\begin{proof}
$(i) \Rightarrow (ii)$: By \Cref{T: JH-algebra is JB-algebra}, the order unit norm is a JB-algebra norm equivalent with the JH-norm.

$(ii) \Rightarrow (iii)$: Since $A$ has a predual, it is a JBW-algebra. Suppose that $A$ contains an infinite sequence $(p_n)_n$ of orthogonal projections. Then 
$$(\lambda_n)_n \mapsto \sum_{n=1}^\infty \lambda_n p_n $$
is an isometric embedding of $c_0$ into $A$, contradicting the reflexivity of $A$.

$(iii) \Rightarrow (iv)$: If the central projection onto the nonatomic part of $A$ is nonzero, then we can keep splitting it up as a sum of two orthogonal projections, yielding an infinite sequence of orthogonal projections. Hence $A$ is atomic, so it is a direct sum of type $I$ factors. The restriction on the projections yields that there can only be finitely many direct summands, and that each direct summand has to be a finite type $I$ JBW-factor.

$(iv) \Rightarrow (i)$: Take $\ip{x}{y}$ to be the sum of the inner products on each component.
\end{proof}

\section{Finsler metric and symmetries} 
One can use the gauge function $M(x/y)$ to define a metric on $C^\circ$ in an order unit space $V$ as follows:
For $x,y\in C^\circ$ let 
\[
d_T(x,y) :=\max\{\log M(x/y), \log M(y/x)\}.
\] 
This metric is called the {\em Thompson metric} \cite{Thom}, and has a Finsler structure. More precisely,  if one defines on the tangent space $T_wC^\circ\cong V$ at $w\in C^\circ$ the Finsler metric by $F(w,x) := \|x\|_w$, where $\|\cdot\|_w$ is the order unit norm with respect to $w$, then $d_T(x,y)$ is the infimum of lengths, 
\[
L(\gamma) = \int_0^1 F(\gamma(t),\gamma'(t)) \,\mathrm{d} t,
\] 
 over all piecewise $C^1$-smooth paths $\gamma\colon [0,1]\to C^\circ$ with $\gamma(0)=x$ and $\gamma(1)=y$, see \cite{Nu}. 
It should be pointed out that in this setting $F$ does not satisfy the usual smoothness and strong convexity conditions used in the theory of Riemann-Finsler manifolds \cite{BCS}.

It is well-known that the metric balls of $d_T$ are convex and the topology of $d_T$ coincides with the order unit norm topology on $C^\circ$ if $V$ is an order unit space, see \cite[Chapter 2]{LNBook}. Clearly all gauge-preserving maps and gauge-reversing maps $\Psi\colon C^\circ\to K^\circ$ are $d_T$-isometries.   

One of the main results of this paper is that the existence of a single gauge-reversing map $\Psi\colon C^\circ\to K^\circ$ turns $(C^\circ,d_T)$ (and hence also $(K^\circ,d_T)$) into a symmetric Finsler space, see Theorem \ref{T:symmetry from antitone}.
\begin{definition} Given $(C^\circ, d_T)$ and $x\in C^\circ$ we call a surjective $d_T$-isometry $\sigma_x\colon C^\circ\to C^\circ$ a {\em $d_T$-symmetry at $x$} if $\sigma_x$ has $x$ as a unique fixed point and satisfies $\sigma^2_x(y)=y$ for all $y\in C^\circ$. Moreover, $(C^\circ,d_T)$ is said to be a {\em symmetric Finsler space}, if there exists a $d_T$-symmetry at each $x\in C^\circ$.
\end{definition}
The interior of the cone of squares in a JB-algebra $(A_+^\circ,d_T)$ is a symmetric Finsler space where the symmetry at $x\in A_+^\circ$ is given by $y\mapsto Q_x(y^{-1})$. 

It turns out that $d_T$-symmetries $\sigma_x\colon C^\circ\to C^\circ$  are gauge-reversing maps. To show this it is convenient to first recall some facts about the geodesics in $(C^\circ,d_T)$.  Recall that a map $\gamma\colon I\to C^\circ$, where $I$ is a possibly unbounded interval in $\mathbb{R}$, is called a {\em geodesic path} if 
\[
d_T(\gamma(t),\gamma(s))=|t-s|\qquad \mbox{ for all }t,s\in I.
\]
The image $\gamma(I)$ will be called a {\em geodesic} in $(C^\circ,d_T)$. 

Of particular interest in this work are unique geodesics of $d_T$, which were studied in \cite{LR}. To analyse them, it is useful to note that if $x,y\in C^\circ$ are linearly independent, then the distance $d_T(x,y)$ is  completely determined by their distance in the $2$-dimensional subcone $C(x,y) :=\mathrm{Span}\{x,y\}\cap C$.  Unique geodesics in $2$-dimensional cones can be characterised as follows, see  \cite[Lemma~3.3]{LR}. 
\begin{lemma}\label{L:unique geod in 2d cones}
Let $C$ be a $2$-dimensional cone and let $x,y\in C^\circ$. Then there is a unique geodesic in $(C^\circ,d_T)$ connecting $x$ and $y$ if and only if either 
\begin{itemize}
\item[$(i)$] $M(x/y)=M(y/x)$, or,
\item[$(ii)$] $M(x/y)=M(y/x)^{-1}$, in which case $x=\lambda y$ for some $\lambda>0$.
\end{itemize}
\end{lemma}

Given $x$ and $y$ linearly independent vectors in $C^\circ$ with $M(x/y)=M(y/x)$, we call the unique geodesic passing through $x$ and $y$ in $(C(x,y)^\circ,d_T)$ a \emph{type I} geodesic, and the unique geodesic defining the ray through $x$ is called a \emph{type II} geodesic in $(C(x,y)^\circ,d_T)$. These geodesics also define geodesics in $(C^\circ,d_T)$ and the type II geodesic is of the form $t\mapsto e^{-t}x$, and the type I geodesic is of the form $t\mapsto \alpha(e^tv+e^{-t}w)$ for some linearly independent vectors $v,w\in \partial C$ and $\alpha>0$ by \cite[Lemma~3.7]{LR}. See Figure~\ref{F:types of geodesics} below. 

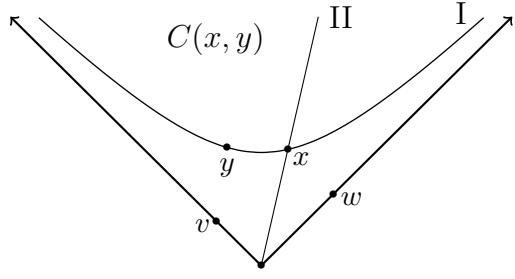
\begin{figure}[H]
\centering
\begin{tikzpicture}[scale=1.5]
\draw (0,0)--(0.5,2.2);
\draw[->,thick] (0,0)--(-2.2,2.2);
\draw[->,thick] (0,0)--(2.2,2.2);      
\draw[domain=-1.95:1.95,smooth,variable=\t,line width=.5pt] plot ({\t},{sqrt(\t*\t+1)});
\node at (0,0) {\tiny{$\bullet$}};
\node at (0.235,1.025) {\tiny{$\bullet$}};
\node at (0.63,0.63) {\tiny{$\bullet$}};
\node at (-0.395,0.395) {\tiny{$\bullet$}};
\node at (-0.3,1.044) {\tiny{$\bullet$}};
\node[below] at (-0.3,1.044) {\small{$y$}};
\node[below] at (0.35,1.1) {\small{$x$}};
\node[below right] at (0.6,0.74) {\small{$w$}};
\node[below left] at (-0.35,0.5) {\small{$v$}};
\node at (-0.4,2) {$C(x,y)$};
\node[right] at (0.5,2.2) {II};
\node[left] at (1.9,2.236) {I};
\end{tikzpicture}
\caption{Type I and type II geodesics}\label{F:types of geodesics}
\end{figure}
\noindent
Every type II geodesic is unique in $(C^\circ,d_T)$ by \cite[Proposition~4.1]{LR}. Determining whether type I geodesics are unique in $(C^\circ,d_T)$ is more subtle and the following theorem, which is \cite[Theorem~4.3]{LR} for order unit spaces, provides a useful characterisation.

\begin{theorem}\label{T:char unique geodesics d_T}
Let $(V,C,u)$ be an order unit space and let $x,y\in C^\circ$ be linearly independent. Suppose that $M(x/y)=M(y/x)$ and define $x',y'\in\partial C$ to be the endpoints of the straight line segment connecting $x$ and $y$ such that $x$ is between $x'$ and $y$, and $y$ is between $x$ and $y'$. Then the type I geodesic passing through $x$ and $y$ is unique in $(C^\circ,d_T)$ if and only there exist no $z\in V\setminus\{0\}$ and $\eps>0$ such that 
\[
x'+tz\in\partial C(x,y,z)\qquad\mbox{and}\qquad y'+tz\in\partial C(x,y,z)
\]
for all $|t|<\varepsilon$. Here $C(x,y,z) :=\mathrm{Span}\{x,y,z\}\cap C$. 
\end{theorem}  

In particular, we have the following consequence, which will be useful in the sequel. 
\begin{corollary}\label{C:unique geod with extreme vector}
Let $(V,C,u)$ be an order unit space and $x\in C^\circ$. If $r,s\in\partial C$ are such that $r+s=x$, then the type I geodesic $\gamma(t)=e^tr+e^{-t}s$ passing through $x$ is unique in $(C^\circ,d_T)$ if either $r$ or $s$ is an extreme vector.
\end{corollary}
\begin{proof}
Suppose without loss of generality that $r$ is an extreme vector and in view of Theorem~\ref{T:char unique geodesics d_T}, assume for sake of a contradiction, that there is a $z\in V\setminus\{0\}$ and an $\varepsilon>0$ such that $r+tz,s+tz\in\partial C(z,r,s)$ whenever $|t|<\varepsilon$. But then 
\[
r=\textstyle{\frac{1}{2}}(r+\textstyle{\frac{1}{2}}\varepsilon z)+\textstyle{\frac{1}{2}}(r-\textstyle{\frac{1}{2}}\varepsilon z),
\]
and so by the extremality of $r$, it follows that $r+\textstyle{\frac{1}{2}}\varepsilon z$ and $r-\textstyle{\frac{1}{2}}\varepsilon z$ are positive multiples of $r$. Hence $z$ is a multiple of $r$ and by our assumption $s-\lambda r\in\partial C(r,s)$ for some $\lambda>0$ which is impossible.
\end{proof}

Using the knowledge of the unique geodesics we now show that each $d_T$-symmetry is a gauge-reversing map.
\begin{theorem}\label{T:d_T-symmetry implies order antimorphism}
Let $(V,C,u)$ be an order unit space. Every $d_T$-symmetry $\sigma_x\colon C^\circ\to C^\circ$  is a gauge-reversing map.
\end{theorem}
\begin{proof}
Let $\gamma_x(t)=e^{-t} x$ be the unique type II geodesic path in $(C^\circ,d_T)$ that passes through the fixed point $x$ and write $R_x=\gamma(\mathbb{R})$. We first  show by contradiction that $\sigma_x(R_x)=R_x$. So, suppose that $\sigma_x$ does not map $R_x$ onto itself. Since $\sigma_x$ is an isometry for $d_T$, the unique geodesic $R_x$ must be mapped to a unique geodesic in $(C^\circ,d_T)$ again. As there is only one type II geodesic passing through $x$, it follows that $\sigma_x$ maps $R_x$ to the type I geodesic $\mu_x$ passing through $x$. 
By \cite[Lemma~3.7]{LR} we know that there exists $v,w\in\partial C$ with $v+w=x$ such that $R_x$ is the image of the geodesic path $t\mapsto e^t v+e^{-t}w$ for $t\in\mathbb{R}$. 

Let $V_x :=\mathrm{Span} \left( R_x\cup \sigma_x(R_x)\right)$ and consider the subcone $C_x := V_x\cap C$. We claim that $\sigma_x$ leaves its relative  interior $C_x^\circ = V_x\cap C^\circ$ invariant. Indeed, for $\lambda>0$ there is a unique type I geodesic path $\gamma_\lambda$ in $C_x^\circ$ that passes through $\lambda x$. In fact, $\gamma_\lambda(t) = \lambda(e^t v+e^{-t}w)=\lambda\gamma_x(t)$ for $t\in\mathbb{R}$ by \cite[Lemma~3.7]{LR}.  See Figure~\ref{F:geodesics and foliation} below.
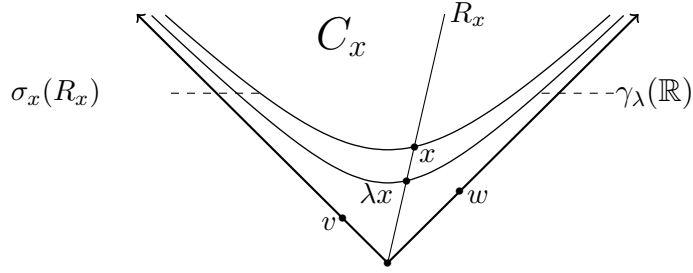
\begin{figure}[H]
\centering
\begin{tikzpicture}[scale=1.5]
\draw (0,0)--(0.5,2.2);
\draw[->,thick] (0,0)--(-2.2,2.2);
\draw[->,thick] (0,0)--(2.2,2.2);      
\draw[domain=-1.95:1.95,smooth,variable=\t,line width=.5pt] plot ({\t},{sqrt(\t*\t+1)});
\draw[domain=-2.07:2.07,smooth,variable=\t,line width=.5pt] plot ({\t},{sqrt(\t*\t+0.5)});
\node at (0,0) {\tiny{$\bullet$}};
\node at (0.235,1.025) {\tiny{$\bullet$}};
\node at (0.63,0.63) {\tiny{$\bullet$}};
\node at (-0.395,0.395) {\tiny{$\bullet$}};
\node at (0.165,0.725) {\tiny{$\bullet$}};
\node[below] at (0.35,1.1) {\small{$x$}};
\node[below] at (-0.1,0.8) {\small{$\lambda x$}};
\node[below right] at (0.6,0.74) {\small{$w$}};
\node[below left] at (-0.35,0.5) {\small{$v$}};
\node[right] at (0.45,2.2) {\small{$R_x$}};
\node[right] at (-3.4,1.5) {\small{$\sigma_x(R_x)$}};
\node[right] at (1.9,1.5) {$\gamma_\lambda(\mathbb{R})$};
\draw[dashed] (-1.9,1.5)--(-1.1,1.5);
\draw[dashed] (1.35,1.5)--(2,1.5);
\node at (-0.4,2) {\Large{$C_x$}};
\end{tikzpicture}
\caption{Geodesics in $C_x$}\label{F:geodesics and foliation}
\end{figure}
\noindent
The geodesics $\sigma_x(R_x)$ and $\gamma_\lambda(\mathbb{R})$ are parallel in the sense that $d_T(\gamma_\lambda(t),\sigma_x(\gamma_x(t)))=|\log\lambda|$ for all $t\in\R$. Hence the images of these geodesics must be parallel as well. Furthermore, we have that $\gamma_x(t)\to 0$ as $t\to\infty$ and the inequality
\[
\log M(\sigma_x(\gamma_\lambda(t))/\gamma_x(t))\le d_T(\sigma_x(\gamma_\lambda(t)),\gamma_x(t))=d_T(\gamma_\lambda(t),\sigma_x(\gamma_x(t)))=|\log\lambda|
\]
implies that $\norm{\sigma_x(\gamma_\lambda(t))}_u\le\exp(|\log\lambda|)\norm{\gamma_x(t)}_u$ for all $t\in\mathbb{R}$. Therefore $\sigma_x(\gamma_\lambda(t))\to 0$ as $t\to\infty$. It follows that $\sigma_x\circ\gamma_\lambda$ is a type II geodesic in $(C^\circ,d_T)$ that passes through $\sigma_x(\lambda x)\in \sigma_x(R_x)\subset C_x^\circ$. Hence $\sigma_x(\gamma_\lambda(t))\in C_x^\circ$ for all $t\in\mathbb{R}$. This implies that $\sigma_x(C_x^\circ) =C^\circ_x$, since $C^\circ_x$ is the disjoint union of type I geodesics passing through $\lambda x$ for $\lambda>0$, which are mapped to type II geodesics passing through $\sigma_x(\lambda x)$ for $\lambda >0$. 

As the $2$-dimensional cone $C_x =\mathrm{Span}_+\{v,w\}$ is isomorphic to the standard $2$-dimensional cone 
$\mathbb{R}^2_+=\{(x_1,x_2)\colon x_1,x_2\geq 0\}$, we know from \cite[Proposition 2.2.1]{LNBook} that there is a surjective isometry $T\colon (C_x^\circ,d_T)\to (\mathbb{R}^2,\|\cdot\|_\infty)$ such that $T(x)=0$.  Moreover, the map $A=T\circ \sigma_x\circ T^{-1}$ is a surjective linear isometry on $\mathbb{R}^2$ for $\|\cdot\|_\infty$ that maps the line $y=x$ onto the line $y=-x$, maps the line $y=-x$ onto the line $y=x$, and satisfies $A^2=\mathrm{Id}$. Hence 
\[
A=\pm\begin{bmatrix}
1 & 0 \\ 0 & -1
\end{bmatrix}
\]
are the only two choices for $A$. However, in this case, either the $x$-axis or the $y$-axis is left invariant by $A$, which contradicts the fact that $\sigma_x$ has a unique fixed point. We conclude that $\sigma_x(R_x)=R_x$. 

To show that $\sigma_x$ is homogeneous of degree $-1$ we let $\lambda>0$ and note that $\sigma_x(\lambda x)=\mu x$ for some $\mu>0$. As 
\[
|\log\lambda|=d_T(x,\lambda x)=d_T(\sigma_x(x),\sigma_x(\lambda x))=d_T(x,\mu x)=|\log\mu|,
\]
$\mu=\lambda$ or $\mu=\lambda^{-1}$. Again, since $\sigma_x$ has a unique fixed point, we see that $\sigma_x(\lambda x)=\lambda^{-1}x$. 

Now let $y\in C^\circ$ and $\lambda >0$. Consider the unique type II geodesic $\gamma_y(t)=e^{-t}y$ in $(C^\circ,d_T)$ passing through $y$. Note that $d_T(\sigma_x(\gamma_y(t)), \sigma_x(\gamma_x(t))  =d_T(\gamma_x(t),\gamma_y(t))=d_T(x,y)$ for all $t\in\R$. Suppose that  $\sigma_x\circ\gamma_y$ is a unique type I geodesic path in $(C^\circ,d_T)$. Then it can be given by $t\mapsto e^tv'+e^{-t}w'$ for some $v',w'\in\partial C$ with $v'+w'=\sigma_x(y)$. From the inequality
\[
e^{-t}w'\le e^tv'+e^{-t}w'\le M(e^tv'+e^{-t}w'/e^t x)e^tx=M(\sigma_x(\gamma_y(t))/\sigma_x(\gamma_x(t)))e^t x,
\]
we infer, after taking norms and logarithms, that 
\[
-2t+\log\frac{\|w'\|_u}{\|x\|_u}\le \log M(\sigma_x(\gamma_y(t))/\sigma_x(\gamma_x(t)))\le d_T(\sigma_x(\gamma_y(t)),\sigma_x(\gamma_x(t)))=d_T(x,y)
\]
for all $t\in\R$,  which is absurd. Thus, $\sigma_x\circ\gamma_y$ must be a type II geodesic path  in $(C^\circ,d_T)$ passing through $\sigma_x(y)$. 

Let $\alpha>0$ be such that $\sigma_x(\lambda y)=\alpha \sigma_x(y)$. As before, it follows from
\[
|\log\alpha|=d_T(\sigma_x(y),\alpha \sigma_x(y))=d_T(\sigma_x(y),\sigma_x(\lambda y))=d_T(y,\lambda y)=|\log\lambda|,
\]
that $\alpha=\lambda$ or $\alpha=\lambda^{-1}$.  As 
\[
\log M(\alpha \sigma_x(y)/\lambda^{-1}x) =\log M(\sigma_x(\lambda y)/\sigma_x(\lambda x))\le d_T(\sigma_x(\lambda y),\sigma_x(\lambda x))=d_T(\lambda y,\lambda x)= d_T(x,y),
\]
we find that $\alpha \norm{\sigma_x(y)}_u =\norm{\sigma_x(\lambda y)}_u\le \lambda^{-1}\exp(d_T(x,y))\norm{x}_u\to 0$ as $\lambda\to \infty$. Hence $\alpha=\lambda^{-1}$, showing that $\sigma_x$ is homogeneous of degree $-1$.

Using the previous observation we now show that $\sigma_x$ is gauge-reversing. Let $y,z\in C^\circ$ and $\lambda>0$ be large enough so that $\sigma_x(y)\le\lambda \sigma_x(z)$ and $z\le \lambda y$. Then $M(\sigma_x(\lambda y)/\sigma_x(z))= M(\sigma_x(y)/\lambda \sigma_x(z))\leq 1$ and $M(z/\lambda y)\le 1$, so that 
\begin{align*}
\log \lambda +\log M(y/z)&=\log M(\lambda y/z)=d_T(\lambda y,z)=d_T(\sigma_x(\lambda y),\sigma_x(z))=\log M(\sigma_x(z)/\sigma_x(\lambda y))\\&\qquad =\log M(\sigma_x(z)/\lambda^{-1}\sigma_x(y))=\log\lambda +\log M(\sigma_x(z)/\sigma_x(y)),
\end{align*}
which shows that $M(y/z)=M(\sigma_x(z)/\sigma_x(y))$.  Thus, $\sigma_x$ is a gauge-reversing map. 
\end{proof}

\begin{remark}
It should be noted that not every gauge-reversing map is a $d_T$-symmetry as is illustrated in the example below. On the $2\times 2$  symmetric real matrices, consider the gauge-reversing map given by 
\[
A \mapsto \mat{0}{-1}{2}{0} A^{-1} \mat{0}{2}{-1}{0}.
\]
Expressing this in terms of the coefficients yields
$$
\mat{a}{b}{b}{c} \mapsto \frac{1}{ac-b^2} \mat{a}{2b}{2b}{4c}.
$$
This map has no positive-definite fixed points, since its square is the map
$$
\mat{a}{b}{b}{c} \mapsto \mat{\frac{1}{4}a}{b}{b}{4c}
$$
which does not have positive-definite fixed points as $a$ and $c$ must be nonzero for the matrix to be positive-definite. 
\end{remark}

\section{Horofunctions and the detour distance} \label{sec:hor}
In our analysis horofunctions will play an important role. In particular, we will use  the so-called Funk and reverse-Funk horofunctions, see \cite{Wa1}.  Recall that on $C^\circ$ in an order unit space $V$ the \emph{Funk metric} is defined by 
\[
F(x,y) :=\log M(x/y),
\]
and the \emph{reverse-Funk metric} is given by
\[
RF(x,y) :=\log M(y/x) \mbox{\qquad (so $RF(x,y)=F(y,x)$)}.
\]
It should be noted that $d_T(x,y) =\max\{F(x,y),RF(x,y)\}$ for $x,y\in C^\circ$. 

The functions $F$ and $RF$ are not metrics on $C^\circ$, but are  quasi-metrics. Recall that a {\em quasi-metric} on a space $X$ is a function $\rho\colon X\times X\to \R$ satisfying: 
\begin{enumerate}[(1)]
\item $\rho(x,x)=0$ for all $x\in X$. 
\item $\rho(x,z)\leq \rho(x,y)+\rho(y,z)$ for all $x,y,z\in X$ (triangle inequality). 
\item $\rho(x,y) = 0 = \rho(y,x)$ for $x,y\in X$ implies $x=y$. 
\end{enumerate} 
So, a quasi-metric $\rho$ need not be symmetric nor nonnegative. It is straightforward to check that if $\rho$ is a quasi-metric on $X$, then the reverse function, $\rho_R(x,y) = \rho(y,x)$, is also a quasi-metric on $X$. Moreover, 
$\rho_\infty(x,y) := \max\{\rho(x,y),\rho_R(x,y)\}$ is a metric on $X$, and for each $y\in X$ the functions $x\mapsto \rho(x,y)$ and $x\mapsto \rho_R(x,y)$  are $1$-Lipschitz on $(X,\rho_\infty)$. 

This allows one to construct the so-called horofunction boundary of $(X,\rho)$ as follows. Fix a basepoint $b\in X$ and consider the space $C(X)$ of continuous functions on $(X,\rho_\infty)$ equipped with the pointwise convergence topology. Let $\mathrm{Lip}_b(X)$ be the set of $1$-Lipschitz functions $f$ on $(X,\rho_\infty)$ with $f(b)=0$. Then it follows from Tychonov's compactness theorem that $\mathrm{Lip}_b(X)$ is a compact subset of $C(X)$, as $f(x)\in  [-\rho_\infty(x,b), \rho_\infty(x,b)]$ for all $x\in X$, and $\mathrm{Lip}_b(X)$ is a closed subset. 

We like to point out that on $\mathrm{Lip}_b(X)$ the topology of compact convergence coincides with the topology of pointwise convergence. Indeed, the sets $B_C(h,\epsilon) := \{g\in  \mathrm{Lip}_b(X)\colon |g(x)-h(x)|<\epsilon\mbox{ for all }x\in C\}$, where $C\subseteq X$ is compact  and $\epsilon>0$, generate the compact convergence topology. As $C$ is compact, we can find $x_1,\ldots,x_n\in C$ such that $C\subseteq B(x_1)\cup\ldots\cup B(x_n)$, where $B(x_i) =\{y\in S\colon \rho_\infty(x_i,y) <\epsilon/3\}$.  Then $U :=\{g\in  \mathrm{Lip}_b(X)\colon |g(x_i)-h(x_i)|<\epsilon/3 \mbox{ for }i=1,\ldots,n\}$ is a neighbourhood of $h$ in the pointwise convergence topology. But for $g\in U$ and $x\in C$, there exists a $k$ with $x\in B(x_k)$, and  
\[
|g(x)-h(x)|\leq |g(x)-g(x_k)|+|g(x_k)-h(x_k)|+|h(x)-h(x_k)|<\epsilon. 
\]
It follows that $U\subseteq B_C(g,\epsilon)$. As the compact convergence topology is finer than the pointwise convergence topology, we conclude the topologies are the same on $\mathrm{Lip}_b(X)$. 

Define $i_\rho\colon X\to C(X)$ by    
\begin{equation}\label{eq:intpoint}
i_\rho(y)(x) := \rho(x,y)-\rho(b,y)\mbox{\qquad for $x\in X$}. 
\end{equation}
Then $i_\rho(y) \in \mathrm{Lip}_b(X)$, and hence $i_\rho(X)$ has a compact closure in $C(X)$, which is denoted $\ol{i_\rho(X)}$ and called the {\em horofunction compactification} of $(X,\rho)$.  The boundary  $\mathcal{H}_\rho(X) :=  \ol{i_\rho(X)}\setminus i_\rho(X)$ is called the {\em horofunction boundary} of $(X,\rho)$, and its elements are called {\em horofunctions}.  It should be noted that in this generality the horofunction compactification need not be a compactification in the usual topological sense, which will not be required here.  

Furthermore, if $f\in \ol{i_\rho(X)}$, then there exists a net $(y_\alpha)_\alpha$ in $X$ such that $(i_\rho(y_\alpha))_\alpha$ converges to $f$. Indeed, if $V_f$ denotes the collection of all neighbourhoods of $f$ in $\ol{i_\rho(X)}$, then $A=\{ (V,y)\in V_f\times X\colon i_\rho(y)\in V\}$ can be preordered  by setting $(V_1,y_1)\leq (V_2,y_2)$ if $V_2\subseteq V_1$, which forms a directed set. By letting $y_\alpha$ be the second coordinate of $\alpha \in A$, we obtain a net $(y_\alpha)_\alpha$ in $X$ such that $(i_\rho(y_\alpha))_\alpha$ converges to $f$.  

In particular, we get horofunction boundaries of the Funk metric  and reverse-Funk metric:   
 \[
 \mathcal{H}_F = \mathcal{H}_F(C^\circ) = \ol{i_F(C^\circ)}\setminus i_F(C^\circ)\mbox{\qquad and \qquad } 
 \mathcal{H}_{RF} = \mathcal{H}_{RF}(C^\circ) = \ol{i_{RF}(C^\circ)}\setminus i_{RF}(C^\circ),
 \] 
 which are respectively called the {\em Funk horofunction boundary} and {\em reverse-Funk horofunction boundary}. 
 
In general it is complicated to compute the horofunctions explicitly.  In case $C^\circ$ is a finite dimensional symmetric cone, however, there exists the following characteristation, see \cite{LLNW,L}: 
\[\mathcal{H}_{RF}=\{ x\in C^\circ\mapsto \log M(y/x)\colon y\in\partial C \mbox{ with }\|y\|_e=1\}\]
and 
\[\mathcal{H}_{F}=\{ x\in C^\circ\mapsto \log M(z/x^{-1})\colon z\in\partial C \mbox{ with }\|z\|_e=1\}.\]

For our purposes we need to consider the so called Busemann points in the horofunction boundary. Let us recall the definition of these points.  
A net $(y_\alpha)_\alpha$ in a quasi-metric space $(X,\rho)$ is called an \emph{almost geodesic for $s\in X$} if for all $\varepsilon>0$ there is a $\beta$ such that  
\[
\rho(s,y_{\alpha'})\ge \rho(s,y_\alpha)+\rho(y_\alpha,y_{\alpha'})-\varepsilon
\]
for all $\beta\le\alpha\le\alpha'$.  A net $(y_\alpha)_\alpha$  in $(X,\rho)$ is called an \emph{almost geodesic} if it is an almost geodesic for some $s\in X$.

The next lemma was proved in \cite[Proposition 2.3]{Wa2} in the context of metric spaces. It provides a characterisation of almost geodesics and shows that an almost geodesic net in a  quasi-metric space $(X,\rho)$ yields a function in the horofunction compactification. To formulate it we need the following notion. A net of real-valued functions $(f_\alpha)_\alpha$ on $S$ is said to be \emph{almost non-increasing} if for each $\varepsilon>0$ there is a $\beta$ such that $f_\alpha\ge f_{\alpha'}-\varepsilon$ for all $\beta\le\alpha\le\alpha'$. 

\begin{lemma}\label{L:almost geod. <-> almost non-dec.}
A net $(y_\alpha)_\alpha$ in a  quasi-metric space $(X,\rho)$ is an almost geodesic for $s$ if and only if the net of functions $f_\alpha(x) := \rho(x,y_\alpha)-\rho(s,y_\alpha)$ is an almost non-increasing net. Moreover, the net of functions $(i_\rho(y_\alpha))_\alpha$ given by (\ref{eq:intpoint}) converges to a function in the horofunction compactification $\ol{i_\rho(X)}$ with basepoint $b$. 
\end{lemma}
\begin{proof}
Suppose $(y_\alpha)_\alpha$ is an almost geodesic for $s$ and let $\varepsilon>0$. Then there exists a $\beta$ such that
\[
\rho(s,y_{\alpha'})\ge \rho(s,y_\alpha)+\rho(y_\alpha,y_{\alpha'})-\varepsilon
\]
for all $\beta\le\alpha\le\alpha'$. It follows that for each $x\in X$,  
\[
f_\alpha(x) + \varepsilon =\rho(x,y_\alpha)-\rho(s,y_\alpha)+\varepsilon\ge \rho(x,y_\alpha)+\rho(y_\alpha,y_{\alpha'})-\rho(s,y_{\alpha'})\ge \rho(x,y_{\alpha'})-\rho(s,y_{\alpha'})=f_{\alpha'}(x),
\]
hence the net $(f_\alpha)_\alpha$ is an almost non-increasing net. 

Conversely, suppose that $(f_\alpha)_\alpha$ is an almost non-increasing net, $\varepsilon>0$, and $\beta$ is such that $f_\alpha\ge f_\alpha'-\varepsilon$ for all  $\beta\le\alpha\le\alpha'$. Then for $x=y_\alpha$ we have that
\[
-\rho(s,y_\alpha)=\rho(y_\alpha,y_\alpha)-\rho(s,y_\alpha)=f_\alpha(y_\alpha)\ge f_{\alpha'}(y_\alpha)-\varepsilon=
\rho(y_\alpha,y_{\alpha'})-\rho(s,y_{\alpha'})-\varepsilon,
\]
 showing that $(y_\alpha)_\alpha$ is an almost geodesic for $s$.
 
 To show the second assertion, we need to show that if $(y_\alpha)_\alpha$ is an almost geodesic for $s$, then $i_\rho(y_\alpha)$ converges pointwise to a function $h$ on $X$.   We first note that if  $x\in X$ and $\epsilon >0$, then there exists a $\beta$ such that $ f_\alpha(x) \geq f_{\alpha'}(x) -\epsilon$ for all $\alpha'\geq \alpha\geq \beta$. 
 Thus, 
 $f_\alpha(x)\geq \sup_{\alpha'\geq \alpha} f_{\alpha'}(x)-\epsilon$ for all $\alpha\geq \beta$, hence
 \[
 \liminf_\alpha f_\alpha(x)\geq \inf_{\beta\leq \alpha} f_\alpha(x)\geq \limsup_{\alpha} f_\alpha(x)-\epsilon.
 \] 
 We conclude that $ \liminf_\alpha f_\alpha(x) = \limsup_\alpha f_\alpha(x)$, and therefore $f_\alpha(x)$ converges to $h_s(x)$, as $f_\alpha(x)\geq -\rho(s,x)>-\infty$ for all $\alpha$. 
 
 Now using the identity 
 \[
 i_\rho(y_\alpha)(x) = f_\alpha(x) - f_\alpha(b), 
 \]
 we see that $ i_\rho(y_\alpha)(x)$ converges to $h(x)=h_s(x) -h_s(b)$. 
 \end{proof}

A horofunction $h\in \ol{i_\rho(X)}$ is called a {\em Busemann point} if there exists an almost geodesic $(y_\alpha)_\alpha$ such that $i_\rho(y_\alpha)_\alpha$ converges to $h$. The set of all Busemann points is denoted by $\mathcal{B}_\rho(X)$.  

It is known, see e.g., \cite{LW}, that on $\mathcal{B}_\rho(X)$ one can define a metric known as the detour metric, which we will exploit. Let us recall the necessary definitions. 
Given horofunctions $g$ and $h$ in $\overline{i_\rho(X)}$, let $V_g$ denote the collection of all neighbourhoods of $g$ in $\overline{i_\rho(X)}$ and define the \emph{detour cost} of $g$ and $h$ by
\begin{align}\label{E:detour cost}
\delta(g,h) := \sup_{V\in V_g}\left(\inf_{i_\rho(x)\in V}\rho(b,x)+h(x)\right).
\end{align} 
We note that $\delta(g,h) \in [0,\infty]$. Indeed, if $(w_\beta)_\beta$ is a net in $X$ with $i_\rho(w_\beta)\to  h$, then for each $x\in X$ and each $\beta$ we have that 
\[
\rho(b,x) +i_\rho(w_\beta)(x) = \rho(b,x) +\rho(x,w_\beta)-\rho(b,w_\beta) \geq 0,
\]
so $\rho(b,x) +h(x)\geq 0$ for all $x\in X$. 

There is an alternative definition of $\delta(g,h)$ using nets, which is useful in practice. To see this we need the following observation from \cite[Lemma 2.6]{Wa2}. 

\begin{lemma}\label{L:detour cost as limit}
Let $(X,\rho)$ be a quasi-metric spaces. If $g$ and $h$ are horofunctions in $\overline{i_\rho(X)}$, then there is a net $(y_\alpha)_\alpha$ in $X$ such that $(i_\rho(y_\alpha))_\alpha$ converges to $g$ and $\delta(g,h)=\lim_\alpha\rho(b,y_\alpha)+h(y_\alpha)$.
\end{lemma}
\begin{proof}
Let $V_g$ denote the set of neighbourhoods of $g$ in $\overline{i_\rho(X)}$. Let 
$
A := \{(V,y)\in V_g\times X\colon i_\rho(y)\in V\} $
be preordered by the relation $(V_1,y_1)\le (V_2,y_2)$ if $V_2\subseteq V_1$, and note that this preorder defines a directed set. Now given $\alpha=(V,y)\in A$, we set  $y_\alpha := y$, so  $(y_\alpha)_\alpha$ is a net in $X$ such that  $i_\rho(y_\alpha)\to g$. 

Suppose that $W$ is an open neighbourhood of $\delta(g,h)$ in $[0,\infty]$ and $\alpha=(V_0,z) \in A$.  Then there is a neighbourhood $V\in V_g$ such that 
\[
\inf_{i_\rho(y)\in V}\rho(b,y)+h(y)\in W. 
\]
Note that   $V\cap V_0\in V_g$ and  
\[
\inf_{i_\rho(y)\in V}\rho(b,y)+h(y)\leq \inf_{i_\rho(y)\in V\cap V_0}\rho(b,y)+h(y)\leq \delta(g,h),
\]
hence 
\[
\inf_{i_\rho(y)\in V\cap V_0}\rho(b,y)+h(y)\in W.
\]
Thus, we may choose $\beta =(V\cap V_0,w) \in A$ such that $\rho(b,w)+h(w)\in W$. As $\beta\geq \alpha$, we conclude that $\delta(g,h)$ is a cluster point of the net $(\rho(b,y_\alpha)+h(y_\alpha))_\alpha$, hence there exists  a subnet such that $\delta(g,h)=\lim_{\alpha}\rho(b,y_\alpha)+h(y_\alpha)$. 
\end{proof}
The following lemma gives an alternative way to define the detour cost. 
\begin{lemma}\label{L:liminf detour cost}
Let $g$ and $h$ be horofunctions and $N_g$ be the collection of all nets $(y_\alpha)_\alpha$ in  a quasi-metric space $(X,\rho)$ such that $(i_\rho(y_\alpha))_\alpha$ converges to $g$. Then 
\[
\delta(g,h)=\inf_{(y_\alpha)\in N_g}(\liminf_\alpha\rho(b,y_\alpha)+h(y_\alpha)).
\]
\end{lemma}
\begin{proof}
Let $V$ be an open neighbourhood of $g$ and  $(y_\alpha)_\alpha$ be a net in $N_g$. There exists a $\beta$ such that $i_\rho(y_\alpha)\in V$ for all $\alpha\ge\beta$. Hence $\inf_{i_\rho(x)\in V}\rho(b,x)+h(x)\le\inf_{\alpha\ge\beta}\rho(b,y_\alpha)+h(y_\alpha)$ and so $\inf_{i_\rho(x)\in V}\rho(b,x)+h(x)\le\liminf_\alpha\rho(b,y_\alpha)+h(y_\alpha)$. Since $V$ and $(y_\alpha)_\alpha$ were arbitrary, it follows that $\delta(g,h)\le\inf_{N_g}(\liminf_\alpha \rho(u,y_\alpha)+g(y_\alpha))$.

On the other hand, there is a net $(y_\alpha)_\alpha$ in $N_g$ such that $\lim_\alpha \rho(b,y_\alpha)+h(y_\alpha)=\delta(g,h)$ by Lemma~\ref{L:detour cost as limit}, from which we conclude that other inequality $\inf_{N_g}(\liminf_\alpha\rho(b,y_\alpha)+h(y_\alpha))\le\delta(g,h)$ also holds.
\end{proof}
To compute the detour cost between Busemann points the following observations are useful, see \cite[Lemma 3.1]{LW} or \cite[Lemma 2.6]{Wa2}.
\begin{lemma}\label{L:detour limits}
If $(y_\alpha)_\alpha$ is an almost geodesic in a quasi-metric space $(X,\rho)$ such that $(i_\rho(y_\alpha))_\alpha$ converges to a Busemann point $g$, then 
\[
\lim_\alpha \rho(b,y_\alpha)+g(y_\alpha)=0.
\]
Moreover, for any horofunction $h$, we have
\[
\delta(g,h)=\lim_\alpha \rho(b,y_\alpha)+h(y_\alpha).
\]
\end{lemma}
\begin{proof}
Suppose that $(y_\alpha)_\alpha$ is an almost geodesic for $s\in X$ and let $\varepsilon>0$. Then there is a $\beta$ such that 
\[
0\le\rho(s,y_\alpha)+\rho(y_\alpha,y_{\alpha'})-\rho(s,y_{\alpha'})\leq \epsilon
\]
for all $\beta\le\alpha\le\alpha'$. So, 
\[
0\leq \rho(s,y_\alpha) +\rho(y_\alpha,y_{\alpha'}) - \rho(b,y_{\alpha'}) - (\rho(s,y_{\alpha'}) - \rho(b,y_{\alpha'}))\leq \varepsilon
\]
for all $\beta\le\alpha\le\alpha'$.
Now fixing $\alpha'\ge\alpha$ and taking the limit over the subnet indexed by $\gamma\ge\alpha'$, we find that $0\leq \rho(s,y_\alpha)+g(y_\alpha) -g(s)\le\varepsilon$ for all $\alpha\geq \beta$.  Thus, $\lim_\alpha \rho(s,y_\alpha)+g(y_\alpha) = g(s)$. 
The first statement now follows from
\[
\lim_\alpha \rho(b,y_\alpha)+g(y_\alpha)= \lim_\alpha (\rho(b,y_\alpha)-\rho(s,y_\alpha))+\lim_\alpha (\rho(s,y_\alpha)+g(y_\alpha))
= -g(s)+g(s)=0.
\]

For the second part of the statement, we first note that for each $x,z\in X$ we have that 
\begin{equation}\label{h(x)}
h(x) \leq (\rho(x,z)-\rho(b,z)) +(\rho(b,z) +h(z)).
\end{equation}
Indeed, if $(w_\beta)_\beta$ is a net in $X$ such that $i_\rho(w_\beta)\to h$, then  for each $\beta$, 
\[
\rho(x,w_\beta) -\rho(b,w_\beta) \leq \rho(x,z) + \rho(z,w_\beta) -\rho(b,w_\beta)  = (\rho(x,z) -\rho(b,z))+(\rho(b,z) + \rho(z,w_\beta) -\rho(b,w_\beta)),   
\]
hence (\ref{h(x)}) holds. 

It follows that for each $x\in X$ we have that 
\begin{equation}\label{equivalence}
h(x)\leq \inf_{(z_\gamma)\in N_g}\left( \liminf_\gamma (\rho(x,z_\gamma)-\rho(b,z_\gamma)) +(\rho(b,z_\gamma) +h(z_\gamma))\right) = g(x) +\delta(g,h),
\end{equation}
where $N_g$ is the set of  all nets $(z_\gamma)_\gamma$ such that $i_\rho(z_\gamma) \to g$ by Lemma \ref{L:liminf detour cost}. 

This implies that 
\[
\rho(b,y_\alpha) +h(y_\alpha) \leq \rho(b,y_\alpha) +g(y_\alpha) +\delta(g,h)
\]
for all $\alpha$. Now using the first part of the lemma we find that 
\[
\limsup_\alpha \rho(b,y_\alpha) +h(y_\alpha)  \leq \delta(g,h).
\]
On the other hand, 
\[
\delta(g,h) \leq \liminf_\alpha \rho(b,y_\alpha) +h(y_\alpha),
\]
which completes the proof.
\end{proof}
Note that the identities in Lemma~\ref{L:detour limits} do not depend on the almost geodesic converging to the Busemann point.

On $\mathcal{B}_\rho(X)$ the {\em detour metric}, $\Delta \colon \mathcal{B}_\rho(X)\times \mathcal{B}_\rho(X)\to [0,\infty]$, is defined by 
\[
\Delta(g,h) := \delta(g,h) +\delta(h,g).
\]
It is known \cite[Proposition 3.1]{LW} that $\Delta$ satisfies all the usual properties of a metric where the distance can attain the value $\infty$. 
\begin{theorem}\label{T:detour metric}
The function  $\Delta\colon \mathcal{B}_\rho(X)\times \mathcal{B}_\rho(X)\to [0,\infty]$ is a (possibly $\infty$-valued) metric.
\end{theorem}
\begin{proof}
The fact that $\Delta$ is symmetric is clear from the definition, and $\Delta(g,h)\geq 0$, as $\delta(g,h)\geq 0$. 
Suppose $g$ and $h$ are Busemann points with $(y_\alpha)_\alpha$ and $(z_\beta)_\beta$ almost geodesics such that $i_\rho(y_\alpha)\to g$ and $i_\rho(z_\beta)\to h$.  It  follows from Lemma \ref{L:detour limits} that $\Delta(g,g)=0$. Now suppose $\Delta(g,h)=0$. Let $x\in X$, and note that the triangle inequality together with the inequality $\rho(b,y_\alpha)+h(y_\alpha)\ge 0$ yield
\begin{align*}
\rho(x,y_\alpha)-\rho(b,y_\alpha)&\le\rho(x,z_\beta)+\rho(z_\beta,y_\alpha)+h(y_\alpha)\\&=\rho(x,z_\beta)+\left(\rho(z_\beta,y_\alpha)-\rho(b,y_\alpha)\right)+\left(\rho(b,y_\alpha)+h(y_\alpha)\right).
\end{align*}
By taking the limit for $\alpha$, it follows from Lemma~\ref{L:detour limits} that $g(x)\le \rho(x,z_\beta)+g(z_\beta)+\delta(g,h)$. Hence
\[
g(x)\le \left(\rho(x,z_\beta)-\rho(b,z_\beta)\right)+\left(\rho(b,z_\beta)+g(z_\beta)\right)+\delta(g,h).
\]
Subsequently taking the limit for $\beta$ we deduce that  
\[
g(x)\le h(x)+\delta(h,g)+\delta(g,h)=h(x)+\Delta(h,g)=h(x).
\]Interchanging the roles of $g$ and $h$ now shows that $g=h$. 

To show the triangle inequality, let $f$ be a Busemann point with almost geodesic $(x_\gamma)_\gamma$, so $i_\rho(x_\gamma)\to f$.  From the triangle inequality for $\rho$, we have
\[
\rho(b,y_\alpha)+\rho(y_\alpha,z_\beta)-\rho(b,z_\beta)\le\rho(b,y_\alpha)+\rho(y_\alpha,x_\gamma)-\rho(b,x_\gamma)+\rho(b,x_\gamma)+\rho(x_\gamma,z_\beta)-\rho(b,z_\beta).
\]
Taking the limit for $\beta$ first yields
\[
\rho(b,y_\alpha)+h(y_\alpha)\le \rho(b,y_\alpha)+(\rho(y_\alpha,x_\gamma)-\rho(b,x_\gamma))+(\rho(b,x_\gamma)+h(x_\gamma)).
\]
Subsequently taking the limit for $\gamma$ gives
\[
\rho(b,y_\alpha)+h(y_\alpha)\le \rho(b,y_\alpha)+f(y_\alpha)+\delta(f,h)
\]
by Lemma~\ref{L:detour limits}. Finally, taking the limit for $\alpha$ gives $\delta(g,h)\le\delta(g,f)+\delta(f,h)$. By swapping the roles of $g$ and $h$, we see that $\delta(h,g)\le\delta(h,f)+\delta(f,g)$, so 
$\Delta(g,h)\leq \Delta(g,f)+\Delta(f,h)$. 
\end{proof}
The relation $g\sim h$ if $\Delta(g,h)<\infty$ is an equivalence relation on $\mathcal{B}_\rho(X)$ and its equivalence classes are called {\em parts}. So, the set of Busemann points is partitioned into parts, each of which is a metric space. 

\begin{remark}\label{remequiv} It should be noted that on the horofunction boundary 
$\mathcal{H}_\rho(X)$ there is a natural equivalence relation where horofunctions $g$ and $h$ are equivalent if 
\[
\sup_{x\in X} |h(x)-g(x)|<\infty.
\]
On $\mathcal{B}_\rho(X)$  this equivalence relation coincides with the $\sim$ relation. Indeed, it follows from  (\ref{equivalence}) that 
\[
\sup_{x\in X} h(x)-g(x)\leq \delta(g,h).
\] 
On the other hand, if $(y_\alpha)_\alpha$ is almost geodesic  for $g$, then by  Lemma~\ref{L:detour limits}, 
\[
h(y_\alpha) -g(y_\alpha) =  \rho(b,y_\alpha)+h(y_\alpha) - (\rho(b,y_\alpha) +g(y_\alpha))  \to \delta(g,h),  
\]
hence $\sup_{x\in X} h(x) - g(x) \geq \delta(g,h)$. This implies that $\sup_{x\in X} h(x) - g(x)=\delta(g,h)$, which shows that $\sup_{x\in X} |h(x)-g(x)|<\infty$ for $g,h\in \mathcal{B}_\rho(X)$ if and only if $g$ and $h$ are in the same part. 
\end{remark}
In the sequel the parts  of $\mathcal{B}_\rho(X)$ consisting of a single Busemann point will play a key role.  We will refer to these parts as the {\em singleton parts}.
\begin{remark} We  note that the detour metric is independent of the basepoint $b\in X$. Indeed, if we change the basepoint from $b$ to say $s$ in $X$, then the horofunctions with respect to the basepoint $s$ are of the form $h_s(x) = h(x) -h(s)$, where $h$ is a horofunction with respect to basepoint $b$. So, the detour cost $\delta(g_s,h_s)$ between two Busemann points $g_s$ and $h_s$ with respect to  $s$ satisfies
\[
\delta(g_s,h_s) = \inf_{(y_\alpha)}\left(\liminf_{\alpha} \rho(s,y_\alpha) +h_s(y_\alpha)\right) = \delta(g,h) +g(s)-h(s).
\]
Here the infimum is taken over all nets $(y_\alpha)_\alpha$ such that the function $x\in X\mapsto \rho(x,y_\alpha) -\rho(s,y_\alpha)$ converges to $g_s$. This shows that $\Delta$ is independent of the basepoint. 
\end{remark}

It turns out that any surjective isometry between quasi-metric spaces induces a detour metric isometry between the Busemann points. Indeed, if $(X_1,\rho_1)$ and $(X_2,\rho_2)$ are quasi-metric spaces and  
$\varphi\colon X_1\to X_2$ is a surjective isometry, then  we can consider the map  $\Phi\colon \mathcal{H}_{\rho_1}(X_1)\to \mathcal{H}_{\rho_2}(X_2)$ defined by  
\begin{equation}\label{isomboundary}
\Phi(h)(x) := h(\phi^{-1}(x))-h(\phi^{-1}(b_2)),
\end{equation}
for $x\in X_1$, where $b_2$ is the basepoint in $X_2$.  
We note that $\Phi(h)$ is in $\mathcal{H}_{\rho_2}(X_2)$. Indeed, if  $(y_\alpha)_\alpha$ is a net in $X_1$ such that $(i_{\rho_1}(y_\alpha))_\alpha$ converges to $h$, then for each $x\in X_2$ we have that  
\begin{align*}
i_{\rho_2}(\phi(y_\alpha))(x)&=\rho_2(x,\phi(y_\alpha))-\rho_2(b_2,\phi(y_\alpha))
=\rho_1(\varphi^{-1}(x),y_\alpha)-\rho_1(\varphi^{-1}(b_2),y_\alpha)\\
&\qquad=(\rho_1(\varphi^{-1}(x),y_\alpha)-\rho_1(b_1,y_\alpha))-(\rho_1(\varphi^{-1}(b_2),y_\alpha)-\rho_1(b_1,y_\alpha))\\
&\qquad\qquad \to h(\varphi^{-1}(x))-h(\varphi^{-1}(b_2))=\Phi(h)(x).
\end{align*}
It is straightforward to check that $\Phi$ has an inverse given by $\Phi^{-1}(g)(x)=g(\phi(x))-g(\phi(b_1))$, hence $\Phi$ is a bijection.  Furthermore, if $(h_\alpha)_\alpha$ is a net that converges (pointwise) to $h$ in $\mathcal{H}_{\rho_1}(X_1)$, then for any $x\in X_2$, 
\[
\Phi(h_\alpha)(x)=h_\alpha(\phi^{-1}(x))-h_\alpha(\phi^{-1}(b_2))\to h(\phi^{-1}(x))-h(\phi^{-1}(b_2))=\Phi(h)(x),
\]
so $\Phi$ is continuous. Similarly, we find that $\Phi^{-1}$ is continuous, so $\Phi$ is  a homeomorphism. 

Clearly, a net $(y_\alpha)_\alpha$ is an almost geodesic at $s_1\in X_1$ if and only if $(\phi(y_\alpha))_\alpha$ is an almost geodesic at $\phi(s_1)\in X_2$, thus $\Phi$ also induces a bijection from  $\mathcal{B}_{\rho_1}(X_1)$ onto $\mathcal{B}_{\rho_2}(X_2)$. The next theorem shows that $\Phi$ is actually  a detour metric isometry, see \cite[Lemma 3.2]{LW}.  
\begin{theorem}\label{T: d_T isom extends to Busemann}
If  $\phi\colon (X_1,\rho_1)\to (X_2,\rho_2)$ is a surjective isometry, then for all $g,h\in\mathcal{B}_{\rho_1}(X_1)$ we have $\Delta(\Phi(g),\Phi(h))=\Delta(g,h)$.
\end{theorem}
\begin{proof}
By Lemma \ref{L:detour cost as limit} there exists a net $(y_\alpha)_\alpha$ in $X_1$ such that $(i_{\rho_1}(y_\alpha))_\alpha$ converges to $g$ and 
\[
\lim_\alpha\rho_1(b_1,y_\alpha)+h(y_\alpha)=\delta(g,h). 
\]
It follows that 
\[
\rho_2(b_2,\phi(y_\alpha))+\Phi(h)(\phi(y_\alpha)) = (\rho_1(\phi^{-1}(b_2), y_\alpha) - \rho_1(b_1,y_\alpha)) + (\rho_1(b_1,y_\alpha) + h(y_\alpha))-h(\phi^{-1}(b_2)).
\]
Taking the limit over $\alpha$ yields $\lim_\alpha \rho_2(b_2,\phi(y_\alpha))+\Phi(h)(\phi(y_\alpha)) = g(\phi^{-1}(b_2))-h(\phi^{-1}(b_2))+\delta(g,h)$.
Now from Lemma~\ref{L:liminf detour cost} it follows that $\delta(\Phi(g),\Phi(h))\le g(\phi^{-1}(b_2))-h(\phi^{-1}(b_2))+\delta(g,h)$, hence 
\[ \Delta(\Phi(g),\Phi(h))\leq \Delta(g,h).\] 
Applying the same argument to $\Phi^{-1}$ gives $\Delta(\Phi(g),\Phi(h))=\Delta(g,h)$.
\end{proof}

We see that $\Phi$ maps parts of $\mathcal{B}_{\rho_1}(X_1)$ onto parts of $\mathcal{B}_{\rho_2}(X_2)$. In particular singleton parts are mapped to singleton parts. This fact will be exploited in the proof of the main result in this paper. 

\section{Funk and reverse-Funk Busemann points of the cone} \label{sec:funk}
The singletons for the Funk and reverse-Funk horofunction boundary of the cone in an order unit space were analysed by Walsh in \cite{Wa2}. We recall, and slightly refine, the results in that paper here. In particular, we shall characterise the singletons in the reverse-Funk horofunction boundary in case $(V,C,u)$ is reflexive. We will need the representation of an order unit space and its bi-dual in terms of affine functions on the state space, which is discussed in  \cite{AS1} and \cite{AS2}. In what follows, we will choose the base point for the reverse-Funk and the Funk geometry to be the order unit $u$.

The following theorem is attributed to Kadison, see \cite[Theorem 1.20]{AS1}, 
\begin{theorem}\label{T:rep of o.u.s.}
Let $(V,C,u)$ be a complete order unit space with state space $S$. For $x\in V$ let  $\hat{x}$ be the function on $S$ given by $\hat{x}(\phi) := \phi(x)$. The map $x\mapsto \hat{x}$ is an order- and norm-preserving isomorphism from $V$ onto $A_
{w*}(S)$, the space of $w^*$-continuous affine functions on $S$. 
\end{theorem}

The bi-dual of a complete order unit space has the following characterisation, see \cite[Proposition 1.11]{AS1}.
\begin{theorem}\label{T:rep bidual o.u.s.}
Let $(V,C,u)$ be a complete order unit space with state space $S$. The restriction of the functionals in $V^{**}$ to $S$ is an order- and norm-preserving isomorphism from $V^{**}$ onto $A_b(S)$, the space of all bounded affine functions on $S$, equipped with the pointwise ordering and the supremum norm. 
\end{theorem}

In \Cref{T:rep of o.u.s.,T:rep bidual o.u.s.} the ordering on the spaces of affine functions is pointwise, so that $C$ corresponds to the cone $A_{w*}(S)^+$ of nonnegative functions in $A_{w*}(S)$, and similarly, $C^{**}$ corresponds to $A_b(S)^+$. We will also need an intermediate cone between $C$ and $C^{**}$, which we will discuss next.

\paragraph{Affine semi-continuous functions}

We denote by $A_{usc}(S)$ the bounded affine upper semi-continuous functions on $S$, and the cone $C_{sc}$ denotes $(A_{usc}(S)^+ - A_{usc}(S)^+)^+$, i.e., the cone of nonnegative bounded affine functions that can be written as a difference of two nonnegative affine upper semi-continuous functions. Note that $C \subseteq C_{sc} \subseteq C^{**}$, and also note that if $V$ is reflexive, the cones are all equal. The equality $C_{sc} = C$ plays an important role in the sequel. There is, however, an order theoretical characterisation for order unit spaces where $C = C_{sc}$, which we will state below. For this characterisation we need the following two notions. An order unit space $V$ is said to be \emph{monotone complete} if every monotone decreasing net $(x_\alpha)_\alpha$ that is bounded from below has an infimum in $C$. A state $\phi$ is called \emph{normal} if for every monotone decreasing net $(x_\alpha)_\alpha$ with infimum $x$ we have $\phi(x_\alpha) \to \phi(x)$. We will also need the following results, which can be found in \cite{Alfsen}, cf.\,\cite[Proposition 1.1.2]{Alfsen} and the proof of \cite[Corollary 1.1.4]{Alfsen}. Here, we write $f<g$ for  two real  functions $f$ and $g$ if  $f(x)<g(x)$ for all $x$ in the domain. 

\begin{proposition}\label{P:lsc ptw lower limit}
Let $(V,C,u)$ be a complete order unit space with state space $S$.  If $f\colon S\to(-\infty,\infty]$ is a $w^*$-lower semi-continuous affine function on the state space, then for each $\varphi\in S$ we have
\[
f(\varphi)=\sup\bigl\{g(\varphi)\colon f>g|_S,\ g\colon V^*\to \mathbb{R}\ \mbox{affine and $w^*$-continuous}\bigr\}.
\]
\end{proposition}

\begin{corollary}\label{C:f lsc is upper lim net}
Let $(V,C,u)$ be a complete order unit space with state space $S$.  If $f\colon S\to(-\infty,\infty]$ is a $w^*$-lower semi-continuous affine function on the state space, then the set of affine $w^*$-continuous functions $g\colon V^*\to\mathbb{R}$ such that $g|_S <f$ is directed upwards, hence is a net indexed by itself converging pointwise to $f$ when restricted to $S$.
\end{corollary}

\begin{theorem}\label{T:char C=C_sc}
    Let $(V,C,u)$ be a complete order unit space. Then $C = C_{sc}$ if and only if $V$ is monotone complete and all states are normal.
\end{theorem}

\begin{proof}
    We start by identifying $V$ with $A_{w^*}(S)$ where $S$ is the state space of $V$ under the isometric order isomorphism $x \mapsto \hat{x}$, where $\hat{x}(\phi) := \phi(x)$. Suppose that $C = C_{sc}$. Let $(x_\alpha)_\alpha$ be a nonnegative decreasing net in $V$. Then $f(\phi) := \inf_\alpha \hat{x}_\alpha(\phi)$ defines an upper semi-continuous affine function on $S$ and by assumption $f \in C$. Furthermore, if $y \in V$ is such that $y \le x_\alpha$ for all $\alpha$, then $y \le f$, so $f$ is the infimum of $(x_\alpha)_\alpha$. This shows that $V$ is monotone complete. Also note that for any state $\phi$ we have $\phi(x_\alpha) \to \phi(f)$, so $\phi$ is normal.

    Conversely, suppose $V$ is monotone complete and that every state is normal. Let $f$ be a nonnegative upper semi-continuous affine function on $S$. Then there exists a decreasing net of $w$*-continuous affine functions $(\hat{x}_\alpha)_\alpha$ such that $f(\phi) = \inf_\alpha\hat{x}_\alpha(\phi)$ for all states $\phi$ by \Cref{C:f lsc is upper lim net}. Since $V$ is monotone complete, this net has an infimum $x_0$ in $V$, and since all states are assumed to be normal, $x_0$ and $f$ are equal as they are the same pointwise limit. Hence $C = C_{sc}$.
\end{proof}

\begin{remark}
    For a JBW-algebra $M$ the unit ball is $\sigma$-weakly compact by \cite[Corollary~2.56]{AS2}, so the cone $M_+$ equals $(M_+)_{sc}$ if and only if all states are normal, which is equivalent to the unit ball being weakly compact. By Kakutani's theorem \cite[Theorem~V.4.2]{Con} this is equivalent to $M$ being reflexive. Since the reflexive JBW-algebras are precisely the JH-algebras by \Cref{T:char JH-algebras as reflexive JB}, a JBW-algebra $M$ is a JH-algebra if and only if $M_+ = (M_+)_{sc}$.
\end{remark}

Below, we give two examples of non-reflexive order unit spaces where both equality and strict inclusion occur between $C$ and $C_{sc}$. In particular, the fact that $C = C_{sc}$ does not imply that $V$ is reflexive.

\begin{example}\label{E:spin C=C_sc}
    Let $X$ be a Banach space and equip $V := \R \oplus X$ with the cone generated by $B_X$, the unit ball of $X$, at level $1$, i.e., $C := \{\lambda(1, x) \colon \lambda \geq 0, \ x \in B_X\}$. We choose $(1,0)\in V$ as the order unit. Then $V^* = \R \oplus X^*$, $S = \{(1,x^*) \colon x^* \in B_{X^*}\}$, and $V^{**} = \R \oplus X^{**}$. Let $v^{**} := \lambda(1,x^{**}) \in V^{**}$ and denote the corresponding element of $A_b(S)$ from \Cref{T:rep bidual o.u.s.} by $\hat{v}^{**}$. Let $\phi = (1,x^*) \in S$. Then $\hat{v}^{**}(\phi) = \lambda + x^{**}(x^*)$, so when determining whether $\hat{v}^{**}$ is upper semi-continuous, we may assume that $\lambda = 0$ and only consider the upper semi-continuity of $x^{**}$ viewed as a function on $(B_{X^*}, w^*)$.

    Suppose $x^{**}$ is upper semi-continuous. Let $(x^*_\alpha)_\alpha$ be a net in $B_{X^*}$ converging to $x^* \in B_{X^*}$. Then $-x^*_\alpha \to -x^*$ and so $\limsup_\alpha x^{**}(-x^*_\alpha) \leq x^{**}(-x^*)$, hence $- \limsup_\alpha x^{**}(-x^*_\alpha) \geq - x^{**}(-x^*) = x^{**}(x^*)$. It follows that
    \[
    \liminf_\alpha x^{**}(x^*_\alpha) = \liminf_\alpha -x^{**}(-x^*_\alpha) = - \limsup_\alpha x^{**}(-x^*_\alpha) \geq x^{**}(x^*),
    \]
    showing that $x^{**}$ is also lower semi-continuous, thus $x^{**}$ is continuous, so $x^{**} \in X$ and therefore $v^{**} \in V$. Hence $A_{usc}(S) = A_{w^*}(S)$ and we conclude that $C_{sc} = C$.
\end{example}
In the next class of examples $C_{sc}$ is almost always different from $C$.
\begin{example}
    Let $X$ be a compact Hausdorff space and let $V := C(X)$ with the pointwise ordering. We denote by $C_{usc}(X)$ the bounded upper semi-continuous functions on $X$. The pure states of $V$ equipped with the $w$*-topology are homeomorphic to $X$, and under this identification, restricting $f \in A_{usc}(S)$ to $X$ yields a map from $A_{usc}(S)$ to $C_{usc}(X)$. Conversely, if $f \in C_{usc}(X)$, \cite[Theorems~II.7.2~and~II.7.6]{Alfsen} shows that there exists a unique $g \in A_{usc}(S)$ such that $g|_X = f$. Hence $A_{usc}(S)$ can be identified with $C_{usc}(X)$.

    If $X$ is finite, then $C(X) = C_{usc}(X)$, and so $C = C_{sc}$. If, however,  $X$ is not finite, then there exists an $x \in X$ such that $\{x\}$ is closed but not open. So, $\mathbf{1}_{\{x\}} \in C_{usc}(X) \setminus C(X)$ and $C \not= C_{sc}$ in that case.
\end{example}

\paragraph{Reverse-Funk Busemann point}

First, observe that for any $z\in C^\circ$ we have 
\[
i_{RF}(z)(x) = \log M(z/x) - \log M(z/u)\qquad(x\in C^\circ),
\]
from which it follows that we may chose $z$ such that $\|z\|_u = 1$. By \Cref{T:rep of o.u.s.}, we can view $z$ as a strictly positive $g \in A_{w*}(S)^+$ with supremum $1$, so that  

\[
i_{RF}(z)(x) = \log M(z/x) = \log \sup_{\varphi \in S}\frac{\varphi(z)}{\varphi(x)} = \log \sup_{\varphi \in S}\frac{g(\varphi)}{\varphi(x)}\qquad(x\in C^\circ).
\]

We recall the description of the reverse-Funk Busemann points (see \cite[Theorem 6.3]{Wa2}). 
\begin{theorem}\label{T:reverse-Funk Busemann}
Let $(V,C,u)$ be a complete order unit space with state space $S$. The Busemann points of the reverse-Funk metric on $C^\circ$ are  precisely the functions of the form 
\begin{equation}\label{RFhoro}
h(x) :=\log \sup_{\phi\in S}\frac{g(\varphi)}{\varphi(x)}\qquad(x\in C^\circ),
\end{equation}
where $g$ is a $w^*$-upper semi-continuous nonnegative affine function on $S$ with supremum 1, that does not correspond to an element of $C^\circ$.
\end{theorem}

The proof of this result relies on the following facts together with \Cref{P:lsc ptw lower limit} and \Cref{C:f lsc is upper lim net}.

\begin{lemma}\label{L: sup less sup g implies f less g}
Let $(V,C,u)$ be a complete  order unit space and let $g$ and $g'$ be $w^*$-upper semi-continuous nonnegative affine functions on $S$ taking values in $[0,\infty)$. If 
\[
\sup_{\phi\in S} \frac{g(\phi)}{\phi(x)}\le \sup_{\phi\in S} \frac{g'(\phi)}{\phi(x)}\mbox{\qquad for all $x\in C^\circ$,}
\]
 then $g\le g'$. 
\end{lemma}
\begin{proof}
We know from Lemma \ref{L:lsc attains min} that $g'$ attains its maximum on $S$. From Corollary \ref{C:f lsc is upper lim net} we know there exists a decreasing net $(g'_\alpha)_\alpha$ in $A_{w*}(S)$  with $g'_\alpha>g'\geq 0$ for all $\alpha$ and $(g'_\alpha)_\alpha$ converges pointwise to $g'$.  By Theorem  \ref{T:rep of o.u.s.} we know that $g'_\alpha$ corresponds to an $x'_\alpha \in C^\circ$.  So for each $\alpha$ and each $\psi\in S$ we have that 
\[
 \frac{g(\psi)}{g'_\alpha(\psi)} = \frac{g(\psi)}{\psi(x'_\alpha)}\le \sup_{\phi\in S} \frac{g(\phi)}{\phi(x'_\alpha)} \le \sup_{\phi\in S} \frac{g'(\phi)}{\phi(x'_\alpha)}= \sup_{\phi\in S} \frac{g'(\phi)}{g'_\alpha(\phi)}\leq 1,
\]
hence $g(\psi)\leq g'_\alpha(\psi)$ for all $\psi$ and all $\alpha$. Thus, $g(\psi)\leq \lim_\alpha g'_\alpha(\psi) = g'(\psi)$ for all $\psi\in S$.
\end{proof}

The proof of Theorem \ref{T:reverse-Funk Busemann} also uses some facts about upper semi-continuous functions, which are collected in the Appendix.

\begin{proof}[Proof of Theorem \ref{T:reverse-Funk Busemann}]
Suppose that $(y_\alpha)_\alpha$ is an almost geodesic in $C^\circ$ for the reverse-Funk geometry such that $(i_{RF}(y_\alpha))_\alpha$ converges to a Busemann point $h$.  We may scale each $y_\alpha$ such that $M(y_\alpha/u)=1$, as $i_{RF}(\lambda x)=i_{RF}(x)$ for all $\lambda>0$.  Now let $g_\alpha := J(y_\alpha)|_S$ for all $\alpha$, where $J(y_\alpha) \in V^{**}$ is given by $\phi\in V^*\mapsto \phi(y_\alpha)$. Recall that by Lemma~\ref{L:almost geod. <-> almost non-dec.} $(i_{RF}(y_\alpha))_\alpha$ is an almost non-increasing net of functions. 

It follows that for any $\varepsilon>0$ there is an index $\alpha_0$ such that 
\[
e^{\varepsilon}\sup_{\phi\in S}\frac{g_\alpha(\varphi)}{\varphi(x)}\ge\sup_{\phi\in S}\frac{g_\beta(\varphi)}{\varphi(x)}\qquad(\mbox{for all }x\in C^\circ)
\]
for all  $\alpha_0\le \alpha\le \beta$. By evaluating in $y_\alpha$, it follows that $e^\varepsilon g_\alpha(\varphi)\ge g_\beta(\varphi)$ for all $\varphi\in S$ for all  $\alpha_0\le \alpha\le \beta$ and we conclude that $(\log g_\alpha|_S)_\alpha$ is an almost non-increasing net. By Lemma~\ref{L:Dini u.s.c.}(i) the net $(\log g_\alpha)_\alpha$ converges on $S$ to a pointwise limit, and therefore $(g_\alpha)_\alpha$ also converges pointwise on $S$ to a $w^*$-upper semi-continuous limit, say $g$, by Lemma~\ref{L:Dini u.s.c.}(ii). 

Note that  $g$ is affine and nonnegative on $S$, since $g_\alpha$ is positive and affine for all $\alpha$. Using that $(\log g_\alpha|_S)_\alpha$ is almost non-increasing again, for $\varepsilon>0$ there is an index $\alpha'$ such that $\log g_\beta(\varphi)\le\log g_\alpha(\varphi)+\varepsilon$ for all $\alpha'\le \alpha\le \beta$ and all $\varphi\in S$. Note that as $M(y_\alpha/u) =1$, there exists a $\phi_\alpha\in S$ such that $g_\alpha(\phi_\alpha)= \phi_\alpha(y_\alpha) =1$. It follows that $0\le \log g_\alpha(\varphi_\beta)+\varepsilon$ for all $\alpha'\le \alpha\le \beta$. By possibly passing to a convergent subnet, we may assume that $(\varphi_\beta)_\beta$ converges $w^*$ to a $\psi\in S$, hence $0\le \log g_\alpha(\psi)+\varepsilon$ for all $\alpha'\leq \alpha$. Taking the limit for $\alpha$ yields $0\le \log g(\psi)+\varepsilon\le\sup_{\phi\in S}\log g(\phi)+\varepsilon$. As $\varepsilon>0$ was arbitrary, we conclude that $\sup_{\phi\in S} g\ge 1$. But $g$ is the pointwise limit of functions with supremum 1 on $S$, so $\sup_{\phi\in S} g(\phi)=1$.

By Lemma~\ref{L:Dini u.s.c.} we find that $\sup_{\phi\in S}g_\alpha(\varphi)/\varphi(x)$ converges to $\sup_{\phi\in S}g(\varphi)/\varphi(x)$, which means that $i_{RF}(y_\alpha)(x) = RF(x,y_\alpha)-RF(u,y_\alpha)$ converges to $\log\sup_{\phi\in S} g(\varphi)/\varphi(x)$. Thus, 
\[
h(x)=\log\sup_{\phi\in S}\frac{g(\varphi)}{\varphi(x)}\qquad(x\in C^\circ).
\]
As $h$ is a Busemann point, it must be a horofunction, hence $g$ does not correspond to a point in $C^\circ$.

To complete the proof we need to show that each function of the form (\ref{RFhoro}) can be realised as a reverse-Funk Busemann point. For $g$ as in (\ref{RFhoro}) we know by Corollary~\ref{C:f lsc is upper lim net} that there is a decreasing net $(g_\alpha)_\alpha$ of affine $w^*$-continuous functions $g_\alpha\colon V^*\to\mathbb{R}$, with $g<g_\alpha|_S$, such that $(g_\alpha|_S)_\alpha$ converges pointwise to $g$. By Theorem \ref{T:rep of o.u.s.} we know that $g_\alpha|_S=J(y_\alpha)|_S$ for some  $y_\alpha\in C^\circ$ for each $\alpha$, as $g_\alpha(\phi)>g(\phi)\geq 0$ for all $\phi\in S$. 

Let $x\in C^\circ$. Note that the functions $G_\alpha(\varphi) := g_\alpha(\varphi)/\varphi(x)$ are $w^*$-continuous, decreasing in $\alpha$, and converge pointwise to the function $G(\varphi) := g(\varphi)/\varphi(x)$ on $S$. By Lemma~\ref{L:Dini u.s.c.} we have that $\sup_{\phi\in S} G_\alpha(\varphi)$ converges to $\sup_{\phi\in S} G(\varphi)$. Hence $\log \sup_{\phi\in S} G_\alpha(\varphi_\alpha)$ converges to $h(x)$. 
Moreover,  $\log\sup_{\phi\in S}g_\alpha(\varphi)$ converges to $\log\sup_{\phi\in S} g(\varphi)=0$, as $g$ has supremum 1 on $S$. 

We find that 
\[
\log M(y_\alpha/x)-\log M(y_\alpha/u) = \log\sup_{\phi\in S}\frac{\varphi(y_\alpha)}{\varphi(x)}-\log\sup_{\phi\in S}\varphi(y_\alpha) =
\log \sup_{\phi\in S}\frac{g_\alpha(\varphi)}{\varphi(x)}-\log\sup_{\phi\in S} g_\alpha(\varphi),
\]
hence $i_{RF}(y_\alpha)$ converges pointwise to $h$. As $(\log \sup_{\phi\in S} G_\alpha)_\alpha$ is decreasing and $(\log\sup_{\phi\in S}g_\alpha(\varphi))_\alpha$ converges to 0, it follows that $(i_{RF}(y_\alpha))_\alpha$ is an almost non-increasing net of functions, so by Lemma~\ref{L:almost geod. <-> almost non-dec.} the net $(y_\alpha)_\alpha$ is an almost geodesic. 
This implies that $h$ is a Busemann point provided that $h \notin i_{RF}(C^\circ)$. To that end, suppose $h = i_{RF}(z)$ for some $z \in C^\circ$. Then 
\[
 \log\sup_{\phi\in S}\frac{g'(\varphi)}{\varphi(x)} = h(x) =\log\sup_{\phi\in S}\frac{g(\varphi)}{\varphi(x)} \qquad(x\in C^\circ)
\]
for some $g' \in A_{w^*}(S)^+$. Applying \Cref{L: sup less sup g implies f less g} twice now yields $g' = g$, which contradicts the fact that $g$ does not correspond to an element of $C^\circ$. 
\end{proof}

Combining Theorem \ref{T:reverse-Funk Busemann} with Remark \ref{remequiv} gives the following corollary, cf.\,\cite[Corollary 6.5]{Wa2}. 
\begin{corollary}\label{rfparts}
 If $h$ and $h'$ are reverse-Funk Busemann points with $w^*$-upper semi-continuous nonnegative affine functions $g$ and $g'$, respectively, then $h$ and $h'$ are in the same part if and only if there exists $\lambda>0$ such that $g'/\lambda\leq g\leq \lambda g'$.
\end{corollary}
\paragraph{Reverse-Funk singletons}  

The following proposition gives a sufficient condition for a reverse-Funk metric Busemann point to be a singleton. 
\begin{proposition}\label{P:extreme ray implies singleton}
Let $(V,C,u)$ be a complete order unit space. If $h$ in (\ref{RFhoro}) is a Busemann point of the reverse-Funk metric on $C^\circ$ and $g$ is on an extreme ray of $C_{sc}$, then $h$ is a singleton Busemann point.
\end{proposition}
\begin{proof}
Take a reverse-Funk Busemann point $h'$ in the same part as $h$. Then by \Cref{T:reverse-Funk Busemann} it is induced by some function $g' \in C_{sc}$. By Corollary \ref{rfparts}  there exists some $\lambda > 0$ with $g'/\lambda \leq g \leq \lambda g'$. Since $g$ is on an extreme ray, there is some $\mu \geq 0$ with $g' / \lambda = \mu g$ and so $g' = \lambda \mu g$. Since $g$ and $g'$ both have supremum 1 on $S$, we have that $\lambda \mu = 1$, hence $g'=g$. We conclude that $h'=h$. 
\end{proof}

\begin{proposition}\label{P:singleton Busemann is ext}
Let $(V,C,u)$ be a complete order unit space with $C=C_{sc}$. If  $h$ in (\ref{RFhoro}) is a singleton Busemann point of the reverse-Funk metric on $C^\circ$, then $g$ lies on an extreme ray of $C$. 
\end{proposition}
\begin{proof}
As $C=C_{sc}$,  $g$ is a $w^*$-continuous function on $S$.  Now suppose that $g=g_1+g_2$ for some $g_1,g_2\in C$. Let $\mu =\sup_{\phi\in S} (g_1+2g_2) >0$ and consider $g'=\mu^{-1}(g_1+2g_2)$. We note that $g'$ does not correspond to an element of $C^\circ$. Indeed, if $g'$ is positive on $S$, then $g$ is also positive on $S$, since $g'\le 2\mu^{-1}g$. This would imply that $g$ corresponds to an element of $C^\circ$, which is not possible, as $h$ is a Busemann point. Thus, by Theorem~\ref{T:reverse-Funk Busemann} the function $g'$ yields a Busemann point $h'$. Since $\frac{\mu}{2}g'\le g\le \mu g'$, it follows that
\[
\frac{\mu}{2}\sup_{\phi\in S}\frac{g'(\varphi)}{\varphi(x)}\le\sup_{\phi\in S}\frac{g(\varphi)}{\varphi(x)}\le \mu \sup_{\phi\in S}\frac{g'(\varphi)}{\varphi(x)}\qquad(\mbox{for all $x\in C^\circ$}),
\]
so $\log(\mu/2)\le h(x)-h'(x)\le \log \mu$ for all $x\in C^\circ$. This implies that $\sup_{C^\circ}|h'(x)-h(x)|<\infty$, hence $h$ and $h'$ must lie in the same part. As $h$ was a singleton, we conclude that  $h=h'$. But now Lemma~\ref{L: sup less sup g implies f less g} and Corollary \ref{rfparts} imply that $g=g'$, hence $g_1$ and $g_2$ are linearly dependent. Thus, $g$ lies on an extreme ray of $C$.
\end{proof}
To summarise we have the following result, which provides a partial answer to \cite[Question 6.6]{Wa2}. 
\begin{corollary}\label{c:reverse funk singleton char}
If $(V,C,u)$ is a complete order unit space with $C=C_{sc}$, then the singleton reverse-Funk Busemann points on $C^\circ$ are precisely the functions of the form 
\[
h(x)=\log \sup_{\phi\in S}\frac{\varphi(p)}{\varphi(x)}=\log M(p/x)\qquad(x\in C^\circ),
\]
where $p$ is a normalised extreme vector of $C=C_{sc}$. 
\end{corollary}
\begin{proof}
Combine \Cref{T:reverse-Funk Busemann}, \Cref{P:extreme ray implies singleton}, and \Cref{P:singleton Busemann is ext}.
\end{proof}
\begin{remark} If $C\subsetneq C_{sc}$ it is not clear to us whether the singleton reverse-Funk Busemann points correspond to the extreme rays $\{\lambda g\in C_{sc}\colon \lambda\geq 0\}$ of the cone $C_{sc}$, where  $g$ is a $w^*$-upper semi-continuous, nonnegative, affine functions on $S$ with supremum 1. 
\end{remark}

\paragraph{Funk Busemann points} 
First, observe that for any $z\in C^\circ$ we have 
\[
i_{F}(z)(x) = \log M(x/z) - \log M(u/z)\qquad(x\in C^\circ),
\]
and similar to the case where the reverse-Funk geometry is considered, we may view $z$ as a $g \in A_{w*}(S)^+$ with infimum $1$, so that  

\[
i_{F}(z)(x) = \log \sup_{\varphi \in S}\frac{\varphi(x)}{\varphi(z)} = \log \sup_{\varphi \in S}\frac{\varphi(x)}{g(\varphi)}\qquad(x\in C^\circ).
\]
Note that a lower semi-continuous function can be infinite-valued, and we will use the convention that $1/\infty = 0$.
\begin{lemma}\label{L: sup less sup g^-1 implies g less f}
Let $(V,C,u)$ be a complete  order unit space and let $g$ and $g'$ be $w^*$-lower semi-continuous nonnegative affine functions on $S$ taking values in $(0,\infty]$. If 
\[
\sup_{\phi\in S}\frac{\phi(x)}{g(\phi)}\le \sup_{\phi\in S} \frac{\phi(x)}{g'(\phi)}\mbox{\qquad for all $x\in C^\circ$,}
\]
 then $g'\le g$. 
\end{lemma}
\begin{proof}
We know from Lemma \ref{L:lsc attains min} that $g'$ attains its minimum $\varepsilon$ on $S$. From Corollary \ref{C:f lsc is upper lim net} there exists an increasing net $(g'_\alpha)_\alpha$ of functions in $A_{w*}(S)$  with $\frac{1}{2}\varepsilon \le g'_\alpha < g'$ for all $\alpha$ and $(g'_\alpha)_\alpha$ converges pointwise to $g'$.  By Theorem  \ref{T:rep of o.u.s.} we know that $g'_\alpha$ corresponds to an $x'_\alpha \in C^\circ$.  So for each $\alpha$ and each $\psi\in S$ we have that 
\[
 \frac{g'_\alpha(\psi)}{g(\psi)} = \frac{\psi(x'_\alpha)}{g(\psi)}\le \sup_{\phi\in S} \frac{\phi(x'_\alpha)}{g(\phi)} \le \sup_{\phi\in S} \frac{\phi(x'_\alpha)}{g'(\phi)} = \sup_{\phi\in S} \frac{g'_\alpha(\phi)}{g'(\phi)}\leq 1,
\]
hence $g_\alpha'(\psi)\leq g(\psi)$ for all $\psi$ and all $\alpha$. Thus, $g'(\psi) = \lim_\alpha g'_\alpha(\psi) \le g(\psi)$ for all $\psi\in S$.
\end{proof}

Let us now recall  Walsh's characterisation of the Funk Busemann points \cite[Theorem 7.1]{Wa2}.
\begin{theorem}\label{T:Funk Busemann}
Let $(V,C,u)$ be a complete order unit space with state space $S$. The Busemann points of the Funk metric on $C^\circ$ are precisely the functions of the form 
\[
h(x) := \log \sup_{\phi\in S}\frac{\varphi(x)}{g(\varphi)}\qquad(x\in C^\circ),
\]
where $g$ is a $w^*$-lower semi-continuous positive affine function on $S$ with infimum 1, that does not correspond to an element of $C^\circ$.
\end{theorem}

\begin{proof}
Let $h$ be defined as above. By Corollary~\ref{C:f lsc is upper lim net} there is an increasing net $(g_\alpha)_\alpha$ of affine $w^*$-continuous functions $g_\alpha\colon V^*\to\mathbb{R}$ such that $(g_\alpha|_S)_\alpha$ converges pointwise to $g$. As $g>0$, the net $(g_\alpha)_\alpha$  may be chosen such that $g_\alpha|_S>0$ by Lemma \ref{L:lsc attains min}. So, for $\alpha$ we have that $g_\alpha|_S=J(x_\alpha)|_S$ for some  $x_\alpha\in C^\circ$ by Theorem \ref{T:rep of o.u.s.}. Let $x\in C^\circ$. Then the functions $G_\alpha(\varphi) := \varphi(x)/g_\alpha(\varphi)$ are $w^*$-continuous, decreasing in $\alpha$, and converge pointwise to the function $G(\varphi) := \varphi(x)/g(\varphi)$ on $S$. By Lemma~\ref{L:Dini u.s.c.} we have that $\sup_{\phi\in S} G_\alpha(\varphi)$ converges to $\sup_{\phi\in S} G(\varphi)$. Hence $\log \sup_{\phi\in S} G_\alpha(\varphi)$ converges to $h(x)$. Using $x =u$ we also have that  $\log\sup_{\phi\in S}g_\alpha(\varphi)^{-1}$ converges to $\log\sup_{\phi\in S} g(\varphi)^{-1}=0$ as $g>0$ has infimum 1 on $S$. 
Since $M(y/z)=\sup_{\phi\in S}\varphi(y)/\varphi(z)$ for $y,z\in C^\circ$, we find that 
\begin{align*}
i_{F}(x_\alpha)(x)&=\log\sup_{\phi\in S}\frac{\varphi(x)}{\varphi(x_\alpha)}-\log\sup_{\phi\in S}\varphi(x_\alpha)^{-1} = \log \sup_{\phi\in S}\frac{\varphi(x)}{g_\alpha(\varphi)}-\log\sup_{\phi\in S} g_\alpha(\varphi)^{-1},
\end{align*}
so $i_{F}(x_\alpha)$ converges pointwise to $h$. 

As $(\log \sup_{\phi\in S} G_\alpha)_\alpha$ is decreasing and  $(\log\sup_{\phi\in S} g_\alpha(\varphi)^{-1})_\alpha$ converges to 0, it follows that $(i_{F}(x_\alpha))_\alpha$ is an almost non-increasing net of functions, so by Lemma~\ref{L:almost geod. <-> almost non-dec.} the net $(x_\alpha)_\alpha$ is an almost geodesic. It follows that $h$ is a Busemann point provided that $h \notin i_{F}(C^\circ)$. Suppose $h = i_{F}(z)$ for some $z \in C^\circ$. Then 
\[
 \log\sup_{\phi\in S}\frac{\varphi(x)}{g'(\varphi)} = h(x) =\log\sup_{\phi\in S}\frac{\varphi(x)}{g(\varphi)} \qquad(x\in C^\circ)
\]
for some strictly positive $g' \in A_{w^*}(S)^+$ with infimum 1. Applying \Cref{L: sup less sup g^-1 implies g less f} twice now yields $g' = g$, which contradicts the fact that $g$ does not correspond to an element of $C^\circ$. 

Conversely, suppose that $(x_\alpha)_\alpha$ is an almost geodesic in $C^\circ$ for the Funk metric such that $(i_{F}(x_\alpha))_\alpha$ converges to a Busemann point $h$. Set $z_\alpha = M(u/x_\alpha)x_\alpha$, so $M(u/z_\alpha)=1$ for all $\alpha$. Note that $i_{F}(z_\alpha)= i_{F}(x_\alpha)$ for all $\alpha$. 
Let $g_\alpha=J(z_\alpha)|_S$ for all $\alpha$. It follows that for any $\varepsilon>0$ there is an index $\alpha_0$ such that 
\[
e^{\varepsilon}\sup_{\phi\in S}\frac{\varphi(x)}{g_\alpha(\varphi)}\ge\sup_{\phi\in S}\frac{\varphi(x)}{g_\beta(\varphi)}\qquad(\mbox{for all }x\in C^\circ)
\]
for $\alpha_0\le \alpha\le \beta$. By evaluating in $z_\alpha$ we also have that 
\begin{equation}\label{phialpha}
e^\varepsilon g_\beta(\varphi)\ge g_\alpha(\varphi)
\end{equation}
for all $\varphi\in S$ for $\alpha_0\le \alpha\le \beta$. 
We conclude that $(-\log g_\alpha|_S)_\alpha$ is an almost non-increasing net. By Lemma~\ref{L:Dini u.s.c.} the net $(-\log g_\alpha)_\alpha$ converges pointwise on $S$ to a $w^*$-upper semi-continuous limit, and therefore $(g_\alpha)_\alpha$ also converges pointwise on $S$ to a $w^*$-lower semi-continuous limit which we will denote by $g$. Furthermore, $g$ is affine  on $S$, since $g_\alpha$  affine for all $\alpha$. Moreover, from (\ref{phialpha}) and the fact that each $g_\alpha$ is positive on $S$ we find that $g$ is positive on $S$. 

Using that $(-\log g_\alpha|_S)_\alpha$ is almost non-increasing again,  we see that for $\varepsilon>0$ there is an index $\alpha'$ such that $-\log g_\beta(\varphi)\le-\log g_\alpha(\varphi)+\varepsilon$ for all $\alpha'\le \alpha\le \beta$ and all $\varphi\in S$. For each $\beta$ we can choose $\varphi_\beta\in S$ such that $g_\beta(\varphi_\beta)=1$, as $M(u/z_\beta)=1$ for all $\beta$. It follows that $0\le -\log g_\alpha(\varphi_\beta)+\varepsilon$ for $\alpha'\le \alpha\le \beta$. By possibly passing to a convergent subnet, we may assume that $(\varphi_\beta)_\beta$ converges to a $w^*$-limit $\psi$ in $S$, and so $0\le -\log g_\alpha(\psi)+\varepsilon$. Taking the limit for $\alpha$ yields $0\le -\log g(\psi)+\varepsilon\le\sup_{\phi\in S}\log (g(\phi)^{-1})+\varepsilon$. As $\varepsilon>0$ was arbitrary, we conclude that $1\leq \sup_{\phi\in S} g(\phi)^{-1}=(\inf_S g(\phi))^{-1}$. But $g$ is the pointwise limit of the functions $g_\beta$ each of which has  infimum 1, as $M(u/z_\beta)=1$ implies that $g_\beta(\phi) \geq 1$ for all $\phi\in S$. Thus,  $\inf_S g(\phi)=1$.

By Lemma~\ref{L:Dini u.s.c.} we find that $\sup_{\phi\in S} \varphi(x)/g_\alpha(\varphi)$ converges to $\sup_{\phi\in S} \varphi(x)/g(\varphi)$, which means that $i_F(x_\alpha)(x)=i_F(z_\alpha)(x)$ converges to $\log\sup_{\phi\in S}  \varphi(x)/g(\varphi)$. Thus, 
\[
h(x)=\log\sup_{\phi\in S} \frac{\varphi(x)}{g(\varphi)}\qquad(x\in C^\circ).
\]
Finally, since $h$ is a Busemann point, it must be a horofunction, so $g$ does not correspond to an element of $C^\circ$.
\end{proof}

\paragraph{Singleton Funk Busemann points} 
The singleton Funk Busemann points correspond to pure states of $S$. Indeed, Walsh proved the following results \cite[Corollary 7.4]{Wa2}.
\begin{proposition}\label{P:Funk singleton Busemann}
Let $(V,C,u)$ be a complete order unit space with state space $S$. If  $h$ is a Busemann point for the Funk metric on $C^\circ$, then $h$ is a singleton if and only if it can be written as $h(x)=\log \psi(x)$ for some pure state $\psi\in S$.
\end{proposition}
\begin{proof}
Let $\psi\in S$ be a pure state, and suppose that $h(x)=\log\psi(x)$ for all $x\in C^\circ$. Then the function $g$ on $S$ defined by  
\[
g(\varphi):=\begin{cases}
1&\mbox{if $\varphi=\psi$};\\ \infty&\mbox{otherwise}.
\end{cases}
\]
is affine and $w^*$-lower semi-continuous with infimum 1 that does not correspond to an element of $C^\circ$. Furthermore, we have that 
\[
\log \sup_{\phi\in S}\frac{\varphi(x)}{g(\varphi)}=\log \psi(x)=h(x)\qquad(x\in C^\circ) 
\]
by construction of $g$, hence $h$ is a Busemann point by Theorem \ref{T:Funk Busemann}. 

Suppose $h'$ is a Busemann point that lies in the same part as $h$. By Theorem~\ref{T:Funk Busemann} there is a $w^*$-lower semi-continuous positive affine function $g'$ on $S$ which has infimum 1 such that 
\[
h'(x)=\log \sup_{\phi\in S}\frac{\varphi(x)}{g'(\varphi)}\qquad(x\in C^\circ).
\]

Using Remark \ref{remequiv} we have $\sup_{x\in C^\circ}|h(x)-h'(x)|=\Delta(h',h)$, so that 
\[
\log\sup_{\phi\in S}\frac{\varphi(x)}{g'(\varphi)}-\Delta(h,h')\le\log\sup_{\phi\in S}\frac{\varphi(x)}{g(\varphi)}\le\log\sup_{\phi\in S}\frac{\varphi(x)}{g'(\varphi)}+\Delta(h,h'),
\]
hence
\[
e^{-\Delta(h,h')}\sup_{\phi\in S}\frac{\varphi(x)}{g'(\varphi)}\le\sup_{\phi\in S}\frac{\varphi(x)}{g(\varphi)}\le e^{\Delta(h,h')}\sup_{\phi\in S}\frac{\varphi(x)}{g'(\varphi)}
\]
for all $x\in C^\circ$. By Corollary \ref{C:f lsc is upper lim net} there exists an increasing net $(g'_\alpha)_\alpha$ of positive affine $w^*$-continuous functions $g'_\alpha\colon V^*\to\mathbb{R}$, with $g'_\alpha|_S <g'$ for all $\alpha$, such that $(g'_\alpha|_S)_\alpha$ converges pointwise to $g'$. So, for $\alpha$ we have that $g'_\alpha|_S=J(x_\alpha)|_S$ for some  $x_\alpha\in C^\circ$ by Theorem \ref{T:rep of o.u.s.}. We find that 
\[
\frac{g'_\alpha(\varphi)}{g(\varphi)}\leq \sup_{\phi\in S}\frac{g'_\alpha(\varphi)}{g(\varphi)}\le e^{\Delta(h,h')}\sup_{\phi\in S}\frac{g'_\alpha(\varphi)}{g'(\varphi)} \leq e^{\Delta(h,h')},
\]
hence  $g'(\varphi)\le e^{\Delta(h,h')}g(\varphi)$ for all $\varphi\in S$. Using the other inequality, a similar argument shows that $g(\varphi)\le e^{\Delta(h,h')}g'(\varphi)$ for all $\varphi\in S$. It follows that $g'(\varphi)=\infty$ for all $\varphi\neq\psi$ and $g'(\psi)$ is finite. As $g'$ must have infimum 1, we conclude that $g=g'$, so $h=h'$ and $h$ is a singleton Busemann point. 

Conversely, suppose that $h$ is a singleton Busemann point, hence $h$ is of the form
\[
h(x)=\log \sup_{\phi\in S}\frac{\varphi(x)}{g(\varphi)}\qquad(x\in C^\circ)
\]
for some $w^*$-lower semi-continuous positive affine function $g$ on $S$ which has infimum 1 that does not correspond to an element of $C^\circ$ by Theorem~\ref{T:Funk Busemann}. By Lemma~\ref{L:lsc attains min} $g$ attains its infimum at say $\psi$, and suppose that $g$ has a finite value at another state $\psi'$. There is an element $x\in C^\circ$ of norm 1 such that $\psi(x)\neq\psi'(x)$. Define $\lambda := \inf_S(g(\varphi)+\varphi(x))$, and $g' := \lambda^{-1}(g+J(x)|_S)$. Note that $g'$ is a $w^*$-lower semi-continuous positive affine function on $S$ with infimum 1 that does not correspond to an element of $C^\circ$. Thus, by Theorem~\ref{T:Funk Busemann}, the function 
\[
h'(x) := \log\sup_{\phi\in S}\frac{\varphi(x)}{g'(\varphi)}\qquad(x\in C^\circ)
\]
is a Busemann point. Note that 
\[
\lambda^{-1}g(\varphi)\le g'(\varphi)\le\lambda^{-1}g(\varphi)+\lambda^{-1}\le 2\lambda^{-1}g(\varphi)
\]
for all $\varphi\in S$. Thus $h(x)+\log\frac{1}{2}\lambda\le h'(x)\le h(x)+\log\lambda$ for all $x\in C^\circ$, hence $h$ and $h'$ must lie in the same part by  Remark \ref{remequiv}. Thus, we must have that $h=h'$, so 
\[
\sup_{\phi\in S}\frac{\varphi(x)}{g(\varphi)}= \sup_{\phi\in S}\frac{\varphi(x)}{g'(\varphi)}.
\]
As before, by Corollary \ref{C:f lsc is upper lim net} there exists an increasing net $(g'_\alpha)_\alpha$ of positive affine $w^*$-continuous functions $g'_\alpha\colon V^*\to\mathbb{R}$, with $g'_\alpha|_S <g'$ for all $\alpha$, such that $(g'_\alpha|_S)_\alpha$ converges pointwise to $g'$. So, for $\alpha$ we have that $g'_\alpha|_S=J(x_\alpha)|_S$ for some  $x_\alpha\in C^\circ$ by Theorem \ref{T:rep of o.u.s.}. It now follows that 
\[
\frac{g'_\alpha(\varphi)}{g(\varphi)}\leq \sup_{\phi\in S}\frac{g'_\alpha(\varphi)}{g(\varphi)}=\sup_{\phi\in S}\frac{g'_\alpha(\varphi)}{g'(\varphi)} \leq 1
\]
hence $g'(\varphi)\le g(\varphi)$ for all $\varphi\in S$. Interchanging the roles between $g$ and $g'$ shows that $g=g'$. 

This implies that $\lambda g=g+J(x)|_S$, which is absurd, as $J(x)|_S\neq 0$ and $g$ does not correspond to an element of $C^\circ$. We conclude that there is no other state that is mapped to a finite value by $g$, hence
\[
g(\varphi)=\begin{cases}
1&\mbox{if $\varphi=\psi$};\\ \infty&\mbox{otherwise}.
\end{cases}
\]
Since $g$ is affine, this implies that $\psi$ must be a pure state, and we conclude that $h(x)=\log\psi(x)$ for all $x\in C^\circ$.
\end{proof}

\section{Atoms and gauge-reversing maps}
The main goal of this section is to show how gauge-reversing maps $\Psi \colon C^\circ \to K^\circ$ give rise to a correspondence between extreme vectors of the cones $C$ and $K$, and the pure states of the order unit spaces $V$ and $W$. This will further be used to prove that a reflexive order unit space $V$ is spanned by the extreme vectors of the cone in case a gauge-reversing map exists. 

Recall that a gauge-reversing map $\Psi\colon C^\circ\to K^\circ$ maps extreme half-lines to extreme half-lines, see Corollary \ref{C:extreme lines -> extreme lines}. In fact, more can be said.  
\begin{lemma}\label{L:formula for phi on extreme line}
Let $(V,C,u)$ and $(W,K,e)$ be a order unit spaces and $\Psi\colon C^\circ\to K^\circ$ be a gauge-reversing map. Let $x\in C^\circ$ and suppose that $p\in C$ is an extreme vector with $M(p/x)=1$.  Then the following hold:
\begin{enumerate}[(i)] 
\item  For the extreme half-line $\ell_x^p$ we have that $\ell_x^p=\{x+tp\colon t\in(-1,\infty)\}$. 
\item There exists an extreme vector $q$ of $K$ with $M(q/\Psi(x))=1$ such that 
 \[
\Psi(x+tp)=\Psi(x)-\textstyle{\frac{t}{t+1}}q\mbox{\qquad for all $t\in(-1,\infty)$}.
\] 
In particular, $\Psi$ maps the extreme half-line $\ell_x^p$ bijectively onto the extreme half-line $\ell_{\Psi(x)}^q$.
\end{enumerate}
\end{lemma}
\begin{proof}
Since $M(p/x)=1$ it follows that $x-p\in\partial C$, which yields $\ell_x^p=\{x+tp\colon t\in(-1,\infty)\}$. 

We know that $\Psi(x+tp)=\Psi(x)+s(t)q$ for some extreme vector $q\in K$ and scalar function $s$ by  Corollary \ref{C:extreme lines -> extreme lines}.  We can normalise $q$ so that  $M(q/\Psi(x))=1$. As $\Psi$ is antitone, we see that if $t>0$, then $s(t)<0$. Likewise, if $t<0$, then  $s(t)>0$. 

We claim that $d_T(x+tp,x)=|\log(t+1)|$ for all $t\in(-1,\infty)$. Indeed, for $t>0$ we have $M(x/x+tp)\le 1$, so $d_T(x+tp,x)=\log M(x+tp/x)$. Suppose $\lambda>0$ is such that $x+tp\le\lambda x$. Then $p\le \frac{\lambda-1}{t}x$, so $\frac{\lambda-1}{t}\ge M(p/x)=1$. Thus,  $1+t\le \lambda$, which yields $1+t\le M(x+tp/x)$. As $x+tp\le(t+1)x$, we see that $M(x+tp/x)\le t+1$, hence $d_T(x+tp,x)=\log(t+1)$. Now, if $-1<t<0$, then similarly, we have $M(x+tp/x)\le 1$, so $d_T(x+tp,x)=\log M(x/x+tp)$. Furthermore, if $\lambda>0$ is such that $x\le \lambda(x+tp)$, then $p\le \frac{\lambda-1}{-t\lambda}x$, so $\frac{\lambda-1}{-t\lambda}\ge M(p/x)=1$. Thus, $\frac{1}{t+1}\le \lambda$, which yields $\frac{1}{t+1}\le M(x/x+tp)$. Since we also have $x\le\frac{1}{t+1}(x+tp)$, it follows that $M(x/x+tp)=\frac{1}{t+1}$, hence $d_T(x+tp,x)=-\log(t+1)$.

In the same way it can be shown that $d_T(\Psi(x)+s(t)q,\Psi(x))=|\log(s(t)+1)|$, hence 
\[
|\log(s(t)+1)|=d_T(\Psi(x)+s(t)q,\Psi(x))=d_T(\Psi(x+tp),\Psi(x))=d_T(x+tp,x)=|\log(t+1)|.
\] 
This implies that either $s(t)=t$ or $s(t)=-\frac{t}{t+1}$. Since $s(t)$ and $t$ have opposite signs, we conclude that $s(t)=-\frac{t}{t+1}$ and $\Psi(x+tp)=\Psi(x)-\frac{t}{t+1}q$ as required.
\end{proof}

We call an extreme vector $p$ of the cone $C$ in an order unit space $(V,C,u)$ an \emph{atom} if $M(p/u) = 1$. Furthermore, we say that atoms $p_1, \dots, p_n$ of $C$ are \emph{orthogonal} if $p_1 + \dots + p_n \le u$. 

 Let $(V,C,u)$ and $(W,K,e)$ be complete order unit spaces with $C = C_{sc}$. Suppose there exists a gauge-reversing map $\Psi \colon C^\circ \to K^\circ$ with $\Psi(u) = e$.  Then we can consider the extension of $\Psi$ to the reverse-Funk Busemann points given in (\ref{isomboundary}), i.e., 
 \[
 \Psi(h)(x) = h(\Psi^{-1}(x))  \qquad (x\in K^\circ).
 \]
This extension maps parts onto parts by \Cref{T: d_T isom extends to Busemann}. In particular, it maps singleton reverse-Funk Busemann points bijectively onto singleton Funk Busemann points. As $C=C_{sc}$, we know that the singleton reverse-Funk Busemann points are precisely the functions of the form 
\[
h_p(x) :=\log \sup_{\phi\in S}\frac{\phi(p)}{\phi(x)}= \log M(p/x)\qquad (x \in C^\circ)
\]
where $p$ is an atom of $C$, see Corollary \ref{c:reverse funk singleton char}. So, 
\[
\Psi(h_p)(x)  =    h_p(\Psi^{-1}(x)) \qquad (x\in K^\circ)
\]
is a singleton Funk Busemann point. By \Cref{P:Funk singleton Busemann} there exists a pure state $\psi_p\in S(K)$ such that $\Psi(h_p)(x) = \log \psi_p(x)$ for $x\in K^\circ$. We will denote the induced bijection between the atoms of $C$ and the pure states of $(W,K,e)$ by $\Psi^*$, so 
\[
\Psi^*(p) = \psi_p \qquad ( \mbox{$p$ atom of $C$).}
\]

Using these observations and the previous lemma we get the following result.
\begin{theorem}\label{T:bijection atoms pure states}
    Let $(V,C,u)$ be a complete order unit space such that $C = C_{sc}$. If there is a complete order unit space $(W,K,e)$ and a gauge-reversing map $\Psi \colon C^\circ \to K^\circ$ with $\Psi(u) = e$, then 
    \begin{itemize}
        \item[$(i)$] There is a bijection between the atoms of $C$ and the atoms of $K$.
        \item[$(ii)$] The map $\Psi^*$ is a bijection between the atoms of $C$ and the pure states of $(W,K,e)$.
   \end{itemize}
Moreover, if  $p$ is an atom of $C$ and $\psi_p = \Psi^*(p)\in S(K)$, then \[
M(p/x) = \psi_p(\Psi(x))\mbox{\qquad  for all $x \in C^\circ$}\]
and $\psi_p(q)=1$,  where $q$ is the unique atom of $K$ satisfying $\Psi(u+\lambda p) = e - \frac{\lambda}{\lambda+1}q$. 
\end{theorem}

\begin{proof}
    Taking $x=u$ in \Cref{L:formula for phi on extreme line} the identity, $\Psi(u+\lambda p) = \Psi(u) -\frac{\lambda}{\lambda +1}q$ for $\lambda>-1$, yields a bijection between the atoms of $C$ and the atoms of $K$. Statement $(ii)$ follows directly from the discussion in the paragraph preceding the theorem. As $\log \psi_p(w)=\Psi(h_p)(w) =  h_p(\Psi^{-1}(w)) = \log M(p/\Psi^{-1}(w))$ for all $w\in K^\circ$, we find that $M(p/x) = \psi_p(\Psi(x))$   for all $x \in C^\circ$. Finally, to see that $\psi_p(q)=1$ note that for $p_n=\frac{1}{n}u+(1-\frac{1}{n})p$ with $n\geq 1$, we have $\Psi(p_n)= ne+(1-n)q$ by \Cref{L:formula for phi on extreme line}, and $M(p/p_n)=1$ for all $n\ge 1$. It now follows that $
    \psi_p(q)=\lim_{n\to \infty}\psi_p(\Psi(p_n)) = \lim _{n \to \infty}M(p/p_n)=1$. 
\end{proof}

The following theorem precisely describes how linear combinations of orthogonal atoms relate to the interior of the cone, the order unit norm, and how they correspond to orthogonality via a gauge-reversing map.

\begin{theorem}\label{T:equivalent conditions atoms}
    Let $(V,C,u)$ be a complete order unit space such that $C = C_{sc}$, and suppose that there is a complete order unit space $(W,K,e)$ with $K = K_{sc}$ and a gauge-reversing map $\Psi \colon C^\circ \to K^\circ$ with $\Psi(u) = e$. Let $p_1, \dots, p_n$ be orthogonal atoms of $C$ and let $q_1, \dots, q_n$ be the corresponding atoms of $K$ according to \Cref{L:formula for phi on extreme line} applied with $x=u$. Then the following hold:

    \begin{itemize}
        \item[$(i)$] For all $\lambda_1, \dots,\lambda_n \in \R$ we have 
        \[
        u + \sum_{k=1}^n\lambda_k p_k \in C^\circ\ \mbox{ if and only if }\ \lambda_k > -1\ \mbox{ for all } k=1,\dots,n.
        \] 
        \item[$(ii)$]  For all $\lambda_1, \dots,\lambda_n > -1$ we have 
        \[
        \Psi\Bigl(u + \sum_{k=1}^n\lambda_k p_k\Bigr) = e - \sum_{k=1}^n \frac{\lambda_k}{\lambda_k+1}q_k.
        \] 
        \item[$(iii)$] The atoms $q_1, \dots, q_n$ are orthogonal.
        \item[$(iv)$] The atoms $p_1, \dots, p_n$ are linearly independent.
        \item[$(v)$] For all $\lambda_1, \dots,\lambda_n \in \R$ we have 
        \[
        \left\| \sum_{k=1}^n \lambda_k p_k\right\|_u = \max_{1 \le k \le n} |\lambda_k|.
        \]
        \item[$(vi)$] 
        The following are equivalent:
         \begin{itemize}
             \item[$(a)$] $p_1 + \dots + p_n = u$;
             \item[$(b)$] $p_1 + \dots + p_n \in C^\circ$;
             \item[$(c)$] $q_1 + \dots + q_n = e$;
             \item[$(d)$] $q_1 + \dots + q_n \in K^\circ$.
         \end{itemize}
    \end{itemize}
\end{theorem}

\begin{proof}
    $(i)$: It follows from the fact that $M(p_1/u) = 1$ that if $u + \lambda_1 p_1 \in C^\circ$, then $\lambda_1 > -1$. Suppose that for some $n > 1$ we have $u + \sum_{k=1}^{n}\lambda_kp_k \in C^\circ$ and, without loss of generality, assume that $\lambda_n \le -1$. Then there exists $\eps > 0$ such that $u + \sum_{k=1}^{n-1}\lambda_kp_k \ge (1 + \eps)p_{n}$. Define $\lambda := \max\{0, \lambda_1, \dots, \lambda_{n-1}\}$ and $P := p_1 + \dots + p_{n-1}$. It follows that $P \le u - p_{n}$ and 
    \[
    (1 + \lambda)u - \lambda p_{n} \ge u + \lambda P \ge u + \sum_{k=1}^{n-1} \lambda_k p_k \ge (1+ \eps) p_{n}. 
    \]
    Hence $(1 + \lambda) u \ge (1+ \lambda + \eps)p_{n}$, and therefore $1 = M(p_{n}/u) \le \frac{1 + \lambda}{1 + \lambda + \eps} < 1$, which is absurd. We conclude that $\lambda_{k} > -1$ for all $k = 1, \dots ,n$. 
    
    Conversely, suppose that $\lambda_1, \dots, \lambda_n > -1$. Let $\lambda := \min_{1 \le k \le n}\lambda_k > -1$. If $\lambda \ge 0$, then $u + \sum_{k=1}^{m}\lambda_kp_k \ge u$ and if $-1 < \lambda < 0$, then 
    \[
    u + \sum_{k=1}^{m}\lambda_kp_k \ge u + \lambda \sum_{k=1}^n p_k \ge (1 + \lambda)u.
    \]

    $(ii)$: We will prove this statement by induction. The case where $m=1$ follows from \Cref{L:formula for phi on extreme line}. Suppose that 
    \[
    \Psi\Bigl(u + \sum_{k=1}^m\lambda_k p_k\Bigr) = e - \sum_{k=1}^m \frac{\lambda_k}{\lambda_k+1}q_k
    \]
    for some $1\le m < n$. Let $\lambda_1, \dots, \lambda_{m+1} > -1$ and define $y := u + \sum_{k=1}^m \lambda_k p_k$. Then by \Cref{T:bijection atoms pure states} there is a pure state $\phi$ of $V$ such that $M(q_{m+1}/x) = \phi(\Psi^{-1}(x))$ for all $x \in K^\circ$. 
    We will show that $\phi(p_{m+1}) = 1$. First note that $M(q_{m+1}/2e-q_{m+1}) = 1$, so 
    \begin{align*}
    1 &= M(q_{m+1}/2e - q_{m+1}) = {\textstyle\frac{1}{2}}M(q_{m+1}/e-{\textstyle\frac{1}{2}}q_{m+1}) = {\textstyle\frac{1}{2}}M(q_{m+1}/\Psi(u+p_{m+1})) \\& = {\textstyle\frac{1}{2}}\phi(u+p_{m+1}) = \textstyle{\frac{1}{2}} + \textstyle{\frac{1}{2}}\phi(p_{m+1}).
     \end{align*}
    In particular, it follows from $p_k \le u - p_{m+1}$ for all $k \neq m+1$ that $0 \le \phi(p_k) \le \phi(u - p_{m+1}) = 0$, and therefore $\phi(p_{k}) = 0$ for all $k \neq m+1$. The the fact that $\phi(p_{m+1}) = 1$ also implies
    \[
    M(p_{m+1}/y+\lambda_{m+1}p_{m+1}) = \frac{1}{\lambda_{m+1} +1}.
    \]
    
    Next, we will show that $M(p_{m+1}/y) = 1$. Indeed, if $\lambda > 0$ is such that $p_{m+1} \le \lambda y$, then we see that $1 = \phi(p_{m+1}) \le \lambda \phi(y) = \lambda$, showing that $M(p_{m+1}/y) \ge 1$. On the other hand we have
    \[
    u + \sum_{k=1}^m \lambda_k p_k \ge u - \sum_{k=1}^m p_k \ge p_{m+1} + \sum_{k=1}^m p_k - \sum_{k=1}^m p_k = p_{m+1},
    \]
    which yields $M(p_{m+1}/y) \le 1$. 
    
    By \Cref{L:formula for phi on extreme line} it now follows that there is an extreme vector $q$ of $K$ such that for $\lambda_{m+1} > -1$ we have 
    \[
    \Psi(y + \lambda_{m+1}p_{m+1}) = \Psi(y) - \frac{\lambda_{m+1}}{\lambda_{m+1}+1}q.
    \] To finish the induction step, we must show that $q = q_{m+1}$. Again, by \Cref{T:bijection atoms pure states} there is a pure state $\psi$ of $W$ such that $M(p_{m+1}/x) = \psi(\Psi(x))$ for all $x \in C^\circ$. From
    \[
    \frac{1}{\lambda_{m+1}+1} = M(p_{m+1}/y + \lambda_{m+1}p_{m+1}) = 1 - \sum_{k=1}^m\frac{\lambda_k}{\lambda_k + 1}\psi(q_k) - \frac{\lambda_{m+1}}{\lambda_{m+1} + 1}\psi(q)
    \]
    and the fact that $\lambda_1, \dots, \lambda_{m+1} > -1$ were chosen arbitrarily, we conclude that $\psi(q_k) = 0$ for all $k = 1, \dots, m$ and $\psi(q) = 1$. By \Cref{T:bijection atoms pure states} the atom $q$ corresponds to an atom $p$ of $C$. For $\lambda > -1$ we have
    \[
    M(p_{m+1}/u - {\textstyle\frac{\lambda}{\lambda +1}}p) = M(p_{m+1}/\Psi^{-1}(e + \lambda q)) = \psi(e + \lambda q) = 1+ \lambda,
    \]
    which implies that $p_{m+1} \le (1+\lambda)u - \lambda p$ for all $\lambda > -1$. Letting $\lambda \downarrow -1$ shows that $p_{m+1} \le p$ and as both are atoms, it follow that $p_{m+1} = p$. We conclude that $q = q_{m+1}$ by \Cref{T:bijection atoms pure states}.

    $(iii)$: By part $(ii)$ it follows that for all $\lambda > -1$
    \[
    \Psi\Bigl(u + \lambda \sum_{k=1}^n p_k\Bigr) = e - \frac{\lambda}{\lambda +1}\sum_{k=1}^n q_k \in K^\circ,
    \]
    so letting $\lambda \to \infty$ yields $q_1 + \dots + q_n \le e$, so that $q_1, \dots, q_n$ are orthogonal.

    $(iv)$: Suppose $\lambda_1 p_1 + \dots + \lambda_n p_n = 0$. Then for all $\mu \in \R$ we have that 
    \[
    u + \sum_{k=1}^n\mu\lambda_k p_k = u \in C^\circ
    \]
    so that $\mu\lambda_k > -1$ for all $k = 1, \dots, n$ by part $(i)$. This forces $\lambda_k = 0$ for all $k = 1, \dots, n$.

    $(v)$: Let $\lambda_1, \dots, \lambda_n \in \R$ be arbitrary but fixed. If $\lambda > 0$ is such that $-\lambda u \le \sum_{k=1}^n \lambda_k p_k \le \lambda u$, then  for $\eps > 0$ it follows that 
    \[
    (\lambda + \eps)u \pm \sum_{k=1}^n \lambda_k p_k \in C^\circ.
    \]
    By part $(i)$, we must have that $\pm \lambda_k < \lambda + \eps$, so that $\pm \lambda_k \le \lambda$ for all $k = 1, \dots, n$ by letting $\eps \downarrow 0$. Hence $\max_{1 \le k \le n}|\lambda_k| \le \lambda$ and $\max_{1 \le k \le n}|\lambda_k| \le \|\sum_{k=1}^n \lambda_k p_k\|_u$. On the other hand, for $\lambda := \max_{1 \le k \le n}|\lambda_k|$, we have that $\lambda u \le \sum_{k=1}^n \lambda_k p_k \le \lambda u$ as $p_1, \dots, p_n$ are orthogonal, thus $\|\sum_{k=1}^n \lambda_k p_k\|_u \le \max_{1 \le k \le n}|\lambda_k|$.

    $(vi)$: The fact that $(a)$ implies $(b)$ is clear. Define $P := p_1 + \dots + p_n \in C^\circ$ and $Q := q_1 + \dots +q_n$, then there exists $\eps > 0$ such that $\eps u \le P$ and for all $\lambda > 0$ it follows that $u + \eps \lambda u \le u + \lambda P$. Hence 
    \[
    e-\frac{\lambda}{\lambda +1}Q = \Psi(u + \lambda P) \le \Psi((1+\eps\lambda)u) = \frac{1}{\eps\lambda +1}e,
    \]
    so letting $\lambda \to \infty$ yields $e \le Q$. Since $q_1, \dots, q_n$ are orthogonal by part $(iii)$, we also  have $Q\le e$ showing that $(b)$ implies $(c)$. The fact that $(c)$ implies $(d)$ is again clear. If $Q \in K^\circ$, then there exists $\eps > 0$ such that $\eps e \le Q$ and applying the same argument for proving that $(b)$ implies $(c)$ but now for $\Psi^{-1}$ yields $Q = e$. Note that $\Psi(u + P) = e - \frac{1}{2}Q = \frac{1}{2}e = \Psi(2u)$ by part $(ii)$, so $P = u$ as $\Psi$ is injective, showing that $(d)$ implies $(a)$.
 \end{proof}

We say that in an order unit space $(V,C,u)$ two vectors $x, y \in C^\circ$ are \emph{connected by a path of finite piecewise extreme directions} if there are atoms $p_1, \dots, p_n$ of $C$  and $x = x_0, x_1, \dots, x_n = y$ for some $n \ge 1$ such that $x_k \in C^\circ \cap \{x_{k-1} + tp_k \colon t \in \R\}\setminus\{x_{k-1}\}$ for all $k=1, \dots, n$.  

\begin{lemma}\label{L:char span ext(C) = V}
    Let $(V,C,u)$ be an order unit space. Then $V$ is the span of the atoms of $C$ if and only if for every two $x, y \in C^\circ$ there is a path of finite piecewise extreme directions connecting $x$ and $y$.
\end{lemma}

\begin{proof}
    If $V$ is the span of the atoms of $C$, then for any two distinct $x, y \in C^\circ$ we can write $y - x = \sum_{i=1}^n \lambda_i p_i$ for some $n \ge 1$, where $p_1, \dots, p_n$ are linearly independent atoms of $C$ and none of the $\lambda_k$ are zero and in decreasing order. 
    Let $I := \{i \colon \lambda_i > 0\}$ and $J := \{i \colon \lambda_i < 0\}$. Define $x_0 := x$ and $x_k := x + \sum_{i=1}^k\lambda_i p_i$ for $k = 1, \dots, n$. It follows that $x_k \in C^\circ$ for all $k$. Indeed, if $k \in I$, then $x_k \ge x$, and if $k \in J$, then $x_k \ge y$. Hence $x_k \in C^\circ \cap \{x_{k-1} + tp_k \colon t \in \R\}\setminus\{x_{k-1}\}$ for all $k = 1, \dots , n$, so $x$ and $y$ are connected by a path of finite piecewise extreme directions.

    Conversely, let $x \in V\setminus\{0\}$. Then there is an $\eps > 0$ such that $u \neq u + \eps x \in C^\circ$. Since $u$ and $u+\eps x$ are connected by a path of finite extreme directions, there are atoms $p_1, \dots, p_n$ such that $\eps x = \sum_{k=1}^n \lambda_k p_k$, showing that $V$ is spanned by the atoms of $C$. 
\end{proof}

\begin{corollary}\label{C: W is span of atoms}
    Let $(V,C,u)$ be an order unit space such that $V$ is the span of the atoms of $C$. If there exists an order unit space $(W,K,e)$ and a gauge-reversing map $\Psi \colon C^\circ \to K^\circ$, then $W$ is the span of the atoms of $K$. 
\end{corollary}

\begin{proof}
    It follows from \Cref{L:formula for phi on extreme line} and the fact that $\Psi$ is a bijection that distinct $\Psi(x), \Psi(y) \in K^\circ$ are connected by a path of finite extreme directions if and only if $x$ and $y$ are. By \Cref{L:char span ext(C) = V} all distinct $x, y \in C^\circ$ are connected by a path of finite extreme directions, hence $W$ is spanned by the atoms of $K$.
\end{proof}

It turns out that a reflexive order unit space is spanned by the atoms of the cone in the presence of a gauge-reversing map. The existence of a gauge-reversing map is essential, as illustrated by the examples given below.

\begin{theorem}\label{T:reflexive frames}
    Let $(V,C,u)$ and $(W,K,e)$ be order unit spaces such that $V$ is reflexive. If $\Psi \colon C^\circ \to K^\circ$ is a gauge-reversing map, then $C^\circ$ is contained in the positive span of the atoms of $C$, so in particular we have that $V$ is spanned by the atoms of $C$.
\end{theorem}

\begin{proof}
    By \Cref{T:bijection atoms pure states} there is a bijection between the pure states of $W$ and the atoms of $C$ as $C = C_{sc}$, since $V$ is reflexive. Hence, $C$ contains atoms. Next, we will show that any orthogonal set of atoms $\{p_1, \dots, p_n\}$ of $C$ such that $p_1 + \dots + p_n < u$ can be extended to a larger orthogonal set. Let $\{q_1, \dots, q_n\}$ be the corresponding set of atoms in $K$ by \Cref{L:formula for phi on extreme line}. By \Cref{T:equivalent conditions atoms}$(vi)$ it follows that $Q := q_1 + \dots + q_n \in \partial K$. Hence, by the Krein-Milman theorem, there is a pure state $\psi$ of $W$ such that $\psi(Q) = 0$, as $\{\phi \in S(W) \colon \phi(Q) = 0\}$ is a non-empty $w$*-closed face. By \Cref{T:bijection atoms pure states} there is an atom $p$ of $C$ such that $M(p/x) = \psi(\Psi(x))$ for all $x \in C^\circ$. For all $\lambda > -1$ it follows from \Cref{T:equivalent conditions atoms} that 
    \[
    M\Bigl(p/u + \lambda\sum_{k=1}^np_k\Bigr) = \psi\Bigl(e-\frac{\lambda}{\lambda + 1}\sum_{k=1}^n q_k\Bigr) = 1 - \frac{\lambda}{\lambda + 1}\sum_{k=1}^n \psi(q_k) = 1
    \]
    as $\psi(q_k) = 0$ for all $k =1 , \dots, n$. Hence $p \le u + \lambda \sum_{k=1}^n p_k$ and letting $\lambda \downarrow -1$, we see that $p + \sum_{k=1}^n p_k \le u$. This shows that we can extend such an orthogonal set of atoms with one more atom as long as they do not sum to $u$. This procedure must terminate after finitely many steps, since otherwise $E := \mathrm{Span}\{p_1, p_2, \dots\}$ with the order unit norm would be isometrically isomorphic to $c_{00}$, the sequences that are eventually zero, equipped with the maximum norm by \Cref{T:equivalent conditions atoms}. Since $V$ is complete this means that there would be an isometric copy of the closure of $c_{00}$ in $V$, which equals $c_0$, the sequences with limit $0$. This space is not reflexive for the maximum norm, which contradicts the fact that $V$ is reflexive. We find that there are finitely many orthogonal atoms such that $p_1 + \dots + p_N = u$. If we now consider a different order unit $v$ of $V$, then we would find finitely many orthogonal atoms for the norm induced by $v$ that sum to $v$, hence $C^\circ$ is contained in the positive span of the atoms of $C$.      
\end{proof}

The following example shows that there exists a reflexive Banach space that is not spanned by the extreme vectors of the unit ball. This construction is inspired by the \href{https://mathoverflow.net/q/488570}{answer}\footnote{https://mathoverflow.net/q/488570} to a question on MathOverflow provided by Mikael de la Salle. 

\begin{example}\label{E:construction of X}
    For each $n \ge 1$ let $X_n$ be a reflexive Banach space for which there exists a unit vector $w_n$ such that $w_n$ is not a linear combination of $n-1$ extreme vectors of the unit ball. For example $X_n:= (\R^n,\|\cdot\|_1)$, where $\|\cdot\|_1$ denotes the $\ell_1$-norm, and then consider $w_n := (\frac{1}{n}, \dots, \frac{1}{n})$. Consider the $\ell_2$-direct sum $X := \ell_2-\bigoplus_{n=1}^\infty X_n$. Note that $X$ is reflexive by \cite[Ch. IV, Section 5.8]{Schaefer}. If $x = (x_n) \in X$ is an extreme point of the unit ball, then each $x_n$ is a scalar multiple of an extreme point in the unit ball of $X_n$. Indeed, if $\|x_n\|^{-1}x_n$ is not an extreme point of the unit ball in $X_n$, then there are distinct $y_n, z_n$ in the unit ball of $X_n$ such that $\|x_n\|^{-1}x_n = \frac{1}{2}(y_n + z_n)$, so $x_n = \frac{1}{2}\|x_n\|(y_n + z_n)$ and the vectors $y := x - (x_n + \|x_n\|y_n) e_n$ and $z := x - (x_n + \|x_n\|z_n)e_n$ are still both in the unit ball of $X$, for which we have $x = \frac{1}{2}(y+z)$. Consider the vector $w := (\frac{1}{n}w_n)_n$, which defines a vector in $X$. Let $x = \sum_{k=1}^n \lambda_k x_k$ be a linear combination of extreme vectors of the unit ball of $X$. Then for any $m > n$, the component of $x$ in $X_m$ is a linear combination of at most $n$ extreme vectors of the unit ball in $X_n$. Hence $w$ is not contained in the linear span of extreme vectors of the unit ball of $X$.   
 \end{example}

Using this reflexive Banach space $X$, we can construct a reflexive order unit space that is not spanned by the atoms of the cone.

 \begin{example}
    Let $X$ be as in \Cref{E:construction of X} and consider the space $V := \R \oplus X$ as in \Cref{E:spin C=C_sc}. We will show that the order unit norm on $V$ coincides with the $\ell_1$ norm on $\R \oplus X$. Indeed, 
    \begin{align*}
    \norm{(\lambda, x)} &= \inf_\mu \{ (-\mu,0) \leq (\lambda,x) \leq (\mu,0) \} = \inf_\mu \{ (\mu + \lambda,x) \geq 0 \mbox{ and } (\mu - \lambda, -x) \geq 0 \} \\
    &= \inf_\mu \{ \mu + \lambda \geq \norm{x} \mbox{ and } \mu - \lambda \geq \norm{x} \} = \inf_\mu \{ \mu \geq \norm{x} \pm \lambda \} \\
    &= \inf_\mu \{ \mu \geq \norm{x} + |\lambda| \} = \norm{x} + |\lambda|. 
\end{align*}
    This shows that $V = \R\oplus_1 X$ is reflexive. The atoms of the cone $C$ are of the form $\frac{1}{2}(1,x)$ where $x$ is an extreme vector of the unit ball of $X$. It follows that $(0,w)$, where $w$ is as in the previous example, cannot be written as a linear combination of atoms of $C$. Hence $V$ is a reflexive order unit space that is not spanned by the atoms of $C$.
\end{example}

The example below shows that there is an order unit space that is spanned by the atoms of the cone, but is itself not reflexive.

\begin{example}
    Suppose that $X$ is a non-reflexive and separable Banach space. Then $X$ can be renormed with an equivalent strictly convex norm by \cite[Theorem~9]{Clar}. Consider the space $V := \R \oplus X$ as in \Cref{E:spin C=C_sc}. Then $C = C_{sc}$ and since the renormed unit ball of $X$ is strictly convex, the space $V$ is spanned by the atoms of $C$. However, note that $V$ is not reflexive. 
\end{example}

\section{Homogeneity of the cone} \label{sec:hom}
The main goal of this section is to show that if there exists a gauge-reversing map $\Psi\colon C^\circ\to K^\circ$, then $C^\circ$ is a homogeneous cone, i.e., $\mathrm{Aut}(C)=\{T\in \mathrm{GL}(V)\colon T(C) = C\}$ acts transitively on $C^\circ$. Using $\Psi^{-1}$ we also get that $K^\circ$ is homogeneous. \Cref{L:formula for phi on extreme line} allows us to analyse directional derivatives of gauge-reversing maps. 

\begin{proposition}\label{P:derivative antitone map on extremes}
Let $\Psi\colon C^\circ\to K^\circ$ be a gauge-reversing map, $x\in C^\circ$, and $r_1,\ldots,r_n$ be  linearly independent extreme vectors of $C$ with  $M(r_k/x)=1$ for $k=1,\ldots,n$. If $q_1,\ldots,q_n$ are extreme vectors of $K$ such that $\Psi$ maps the extreme half-line $\ell_x^{r_k}$ onto the extreme half-line $\ell_{\Psi(x)}^{q_k}$  and $M(q_k/\Psi(x))=1$  for $k=1,\ldots,n$, then 
\[
\lim_{t\to 0}\frac{1}{t}\left(\Psi\Bigl(x+t\sum_{k=1}^n\lambda_k r_k\Bigr)-\Psi(x)\right)=-\sum_{k=1}^n\lambda_k q_k.
\]
\end{proposition}
\begin{proof}
We argue by induction on $n$. For $\theta\neq 0$ define the function $F_\theta\colon t \mapsto\frac{\theta t}{\theta t +1}$ for $t \in \mathbb{R}$ with  $|t|$ sufficiently small. It follows that $\Psi(x+t\lambda_1 r_1)=\Psi(x)-F_{\lambda_1}(t)q_1$ by Lemma~\ref{L:formula for phi on extreme line} for sufficiently small $|t|$, hence 
\[
\lim_{t\to 0}{\textstyle\frac{1}{t}}\left(\Psi(x+t\lambda_1r_i)-\Psi(x)\right)=\lim_{t\to 0}-{\textstyle\frac{1}{t}}F_{\lambda_1}(t)q_1=-\lambda_1 q_1,
\]
proving the base case. 

Suppose that for some $m\ge 1$ we have
\[
\lim_{t\to 0}\frac{1}{t}\left(\Psi\Bigl(x+t\sum_{k=1}^m\lambda_k r_k\Bigr)-\Psi(x)\right)=-\sum_{k=1}^m\lambda_k q_k.
\]
If $|t|$ is sufficiently small, then by Lemma \ref{L:formula for phi on extreme line}
\begin{eqnarray*}
\Psi\Bigl(x+t\sum_{k=1}^{m+1}\lambda_k r_k\Bigr) &=& \Psi\Bigl(\Bigl(x+t\sum_{k=1}^{m}\lambda_k r_k\Bigr)+t\lambda_{m+1}r_{m+1}\Bigr)\\
 & =& \Psi\Bigl(x+t\sum_{k=1}^{m}\lambda_k r_k\Bigr)-F_{M_{m+1}(t)}(t\lambda_{m+1})q_{m+1}(t),
\end{eqnarray*}
where $M_{m+1}(t)=M(r_{m+1}/x+t\sum_{k=1}^m\lambda_k r_k)$ and $q_{m+1}(t)$ is the extreme vector of $K$ given by
\[
-q_{m+1}(t)=F_{M_{m+1}(t)}(1)^{-1}\left(\Psi\Bigl(x+t\sum_{k=1}^m\lambda_k r_k+r_{m+1}\Bigr)-\Psi\Bigl(x+t\sum_{k=1}^m\lambda_k r_k\Bigr)\right).
\]
By Lemma~\ref{L:M function is continuous}, $M(r_{m+1}/x+t\sum_{k=1}^m\lambda_k r_k)\to M(r_{m+1}/x)=1$ as $t\to 0$, hence $F_{M_{m+1}(t)}(1)^{-1}\to 2$ as $t\to 0$. As the order unit norm topology coincides with the $d_T$-metric topology on $C^\circ$ and $K^\circ$ (see e.g.,\cite{LNBook}), we find  that $-q_{m+1}(t)\to -q_{m+1}$ as $t\to 0$. Again by Lemma~\ref{L:M function is continuous}, we have $\frac{1}{t}F_{M_{m+1}(t)}(t\lambda_{m+1})\to \lambda_{m+1}$ as $t\to 0$, hence by applying the induction hypothesis we conclude that
\begin{align*}
\lim_{t\to 0}\frac{1}{t}&\left(\Psi\Bigl(x+t\sum_{k=1}^{m+1}\lambda_k r_k\Bigr)-\Psi(x)\right)\\&=\lim_{t\to 0}\frac{1}{t}\left(\Psi\Bigl(x+t\sum_{k=1}^{m}\lambda_k r_k\Bigr)-\Psi(x)\right)-\lim_{t\to 0}{\textstyle\frac{1}{t}}F_{M_{m+1}(t)}(t\lambda_{m+1})q_{m+1}(t)\\&=-\sum_{k=1}^m\lambda_k q_k-\lambda_{m+1}q_{m+1},
\end{align*}
which completes the induction step.
\end{proof}

From Proposition~\ref{P:derivative antitone map on extremes} we see that if $(V,C,u)$ is a complete  order unit space that is spanned by the atoms of $C$, then $\Psi\colon C^\circ\to K^\circ$ is Gateaux differentiable at every $x\in C^\circ$. In this case we will denote the Gateaux derivative at $x$ by $D\Psi(x)\colon V\to W$. Using that $W$ is spanned by the atoms of $K$ by \Cref{C: W is span of atoms}, one can show in the same way that $\Psi^{-1}$ is Gateaux differentiable at any $x \in K^\circ$.

\begin{lemma}\label{L:Gateaux derivative is an automorphism}
Suppose that $(V,C,u)$ and $(W,K,e)$ are complete order unit spaces such that $V$ is spanned by the atoms of $C$. If $\Psi\colon C^\circ\to K^\circ$ is a gauge-reversing map, then for $x\in C^\circ$ we have that   $-D\Psi(x)\colon V\to W$ is a homogeneous of degree 1 order-isomorphism with inverse $-D\Psi^{-1}(\Psi(x))$.
\end{lemma}
\begin{proof}
We start by proving that $-D\Psi(x)$ is homogeneous of degree 1 and order-preserving. Let $y,z\in V$ and $\lambda\neq  0$. Then
\begin{align*}
-D\Psi(x)(\lambda y)&=-\lim_{t\to 0}\frac{\Psi(x+\lambda ty)-\Psi(x)}{t}=-\lambda\lim_{t\to 0}\frac{\Psi(x+\lambda ty)-\Psi(x)}{\lambda t}=-\lambda D\Psi(x)(y).
\end{align*}
If $y\le z$, then we have for any $t\neq 0$ sufficiently small  that 
\[
\frac{\Psi(x+tz)-\Psi(x)}{t}\le\frac{\Psi(x+ty)-\Psi(x)}{t}, 
\]  
as $\Psi$ is antitone. Thus,  $D\Psi(x)(z)\le D\Psi(x)(y)$, hence $-D\Psi(x)(y)\le-D\Psi(x)(z)$. 

By interchanging the roles of $\Psi$ and $\Psi^{-1}$, one can show in the same way that $\Psi^{-1}$ is Gateaux differentiable at any $x \in K^\circ$ so that $-D\Psi^{-1}(\Psi(x))$ is also homogeneous of degree 1 and order-preserving.

Next, we prove that $-D\Psi(x)$ is bijective with inverse $-D\Psi^{-1}(\Psi(x))$. Let $y=\sum_{k=1}^n\lambda_k r_k$ where $r_1,\ldots,r_n$ are linearly independent extreme vectors of $C$ such that $M(r_k/x)=1$ for all $k=1,\ldots,n$. By Corollary~\ref{C:extreme lines -> extreme lines} and Lemma~\ref{L:formula for phi on extreme line} there are extreme vectors $q_k$ of $K$ such that $M(q_k/\Psi(x))=1$ and
\[
\Psi(x+tr_k)=\Psi(x)-\frac{t}{t+1}q_k\qquad\mbox{for all $ t\in(-1,\infty)$ and $k=1,\ldots,n$}.
\]
 So, 
\[
D\Psi(x)(r_k)=\lim_{t\to 0}\frac{\Psi(x+tr_k)-\Psi(x)}{t}=\lim_{t\to 0}\frac{-1}{t+1}q_k=-q_k
\]
and 
\begin{align*}
D\Psi^{-1}(\Psi(x))(-D\Psi(x)(r_k))=D\Psi^{-1}(\Psi(x))(q_k)=\lim_{t\to 0}\frac{\Psi^{-1}(\Psi(x)+tq_k)-x}{t}=\lim_{t\to 0}\frac{-1}{t+1}r_k=-r_k.
\end{align*}
Using the identity for the Gateaux derivative for gauge-reversing maps in Proposition~\ref{P:derivative antitone map on extremes} now yields
\[
-D\Psi^{-1}(\Psi(x))(-D\Psi(x)(y))=-D\Psi^{-1}(\Psi(x))\left(\sum_{k=1}^n\lambda_k q_k\right)=\sum_{k=1}^n\lambda_k r_k=y.
\]
By interchanging the roles of $\Psi$ and $\Psi^{-1}$, we also have that $-D\Psi(x)(-D\Psi^{-1}(\Psi(x))(y))=y$,  hence $-D\Psi(x)\colon V\to W$ is a bijection with inverse $-D\Psi^{-1}(\Psi(x))$.  Thus, $-D\Psi(x)\colon V\to W$ is an order-isomorphism.
\end{proof}

Given a gauge-reversing map  $\Psi\colon C^\circ\to K^\circ$, we define the map $S_x\colon C^\circ\to C^\circ$ by 
\begin{equation}\label{sym}
S_x(y)=-D\Psi^{-1}(\Psi(x))(\Psi(y))\mbox{\qquad for $y\in C^\circ$}.
\end{equation}
We shall prove that $S_x$ is a $d_T$-symmetry for each $x\in C^\circ$, for which we need the following lemma.

\begin{lemma}\label{L:DS_x=-Id} Suppose that $(V,C,u)$ and $(W,K,e)$ are complete order unit spaces such that $V$ is spanned by the atoms of $C$. If $\Psi\colon C^\circ\to K^\circ$ is a gauge-reversing map, then for each $x\in C^\circ$ the map $S_x$ satisfies:
\begin{enumerate}[(i)]
\item $S_x(x+\lambda r)=x-\frac{\lambda}{\lambda+1}r$ for any extreme vector $r\in C$ with $M(r/x)=1$ and $\lambda > -1$.
\item $S_x$ is a gauge-reversing map.
\item $DS_x(x)=-\mathrm{Id}_V$.
\end{enumerate}
\end{lemma}
\begin{proof}
Let $r\in C$ be an extreme vector such that $M(r/x)=1$. By Lemma~\ref{L:formula for phi on extreme line} we have that 
\[
\Psi(x+\lambda r)=\Psi(x)-\frac{\lambda}{\lambda+1}q 
\]
for some extreme vector $q\in K$ with $M(q/\Psi(x))=1$ and all $\lambda > -1$. Note this also implies that we have the identity 
\[
\Psi^{-1}(\Psi(x)+\mu q)=\Psi^{-1}(\Psi(x-\textstyle{\frac{\mu}{\mu+1}}r))=x-\textstyle{\frac{\mu}{\mu+1}}r
\]
for all $\mu > -1$. Hence for $\lambda > -1$ and sufficiently small $|t|$, we have
\begin{eqnarray*}
\Psi^{-1}\bigl(\Psi(x)+t\Psi(x+\lambda r)\bigr)-x&= &\Psi^{-1}\bigl(\Psi(x)+t(\Psi(x)-\textstyle{\frac{\lambda}{\lambda+1}}q)\bigr)-x\\
&= & \textstyle{\frac{1}{t+1}}\Psi^{-1}\bigl( \Psi(x)-\textstyle{\frac{t\lambda}{(t+1)(\lambda+1)}}q \bigr)-x\\
&= & \textstyle{\frac{1}{t+1}}\bigl( x+\textstyle{\frac{t\lambda}{(t+1)(\lambda+1)-t\lambda}}r\bigr)-x\\
&= &\textstyle{-\frac{t}{t+1}}x+\textstyle{\frac{t\lambda}{(t+1)^2(\lambda+1)-t(t+1)\lambda}}r,
\end{eqnarray*}
and so, it follows that
\[
S_x(x+\lambda r)=-\lim_{t\to 0}\frac{\Psi^{-1}\bigl(\Psi(x)+t\Psi(x+\lambda r)\bigr)-x}{t}=\lim_{t\to 0}\left(\textstyle{\frac{1}{t+1}}x-\textstyle{\frac{\lambda}{(t+1)^2(\lambda+1)-t(t+1)\lambda}}r\right)=x-\textstyle{\frac{\lambda}{\lambda+1}}r.
\]
This proves statement $(i)$.

For statement $(ii)$, note that $-D\Psi^{-1}(\Psi(x))\colon K^\circ\to C^\circ$ is a homogeneous degree 1 order-isomorphism by Lemma \ref{L:Gateaux derivative is an automorphism}, hence $-D\Psi^{-1}(\Psi(x))$ is gauge-preserving. This implies that $S_x\colon C^\circ\to C^\circ$ is gauge-reversing. 

To prove the final assertion let $y\in V$. We can write $y=\sum_{k=1}^n\lambda_k r_k$ where the $r_k$'s are  extreme vectors of $C$ such that $M(r_k/x)=1$ for all $k$ and $\lambda_k\in\mathbb{R}$. Since $S_x$ is a gauge-reversing map, Proposition~\ref{P:derivative antitone map on extremes} together with statement $(i)$ imply that 
\[
DS_x(x)(y)=\lim_{t\to 0}\frac{S_x(x+ty)-S_x(x)}{t}=-\sum_{k=1}^n\lambda_k r_k=-y,
\]
which proves statement $(iii)$.
\end{proof}
We can now show that $S_x$ is a $d_T$-symmetry. 
\begin{theorem}\label{T:symmetry from antitone}
Suppose that $(V,C,u)$ and $(W,K,e)$ are complete order unit spaces such that $V$ is spanned by the atoms of $C$. If $\Psi\colon C^\circ\to K^\circ$ is a gauge-reversing map, then for each $x\in C^\circ$ the map $S_x\colon C^\circ\to C^\circ$ is a $d_T$-symmetry with unique fixed point $x$. Moreover, $(C^\circ,d_T)$ and $(K^\circ,d_T)$ are Finsler symmetric spaces.
\end{theorem}
\begin{proof} 
Taking $\lambda=0$ in Lemma~\ref{L:DS_x=-Id}$(i)$ yields that $x$ is a fixed point of $S_x$. Moreover, we also know from  Lemma~\ref{L:DS_x=-Id} that $S_x$ is gauge-reversing, and hence a $d_T$-isometry. 

To show that $S^2_x$ is the identity map on $V$, we first note that $S^2_x$ is a gauge-preserving map.  It follows from \cite[Theorem~B]{Schaf} that there exists $T_x\in \mathrm{Aut}(C) = \{T\in \mathrm{GL}(V)\colon T(C) = C\}$ such that $T_x= S^2_x$ on $C^\circ$. Note that $T_x x = x$ and by Lemma~\ref{L:DS_x=-Id}$(i)$ we have $T_x(x+r) = S_x^2(x+r) = x+r$ for $r\in C$ an extreme vector with $M(r/x)=1$, hence, since $T_x$ is affine, it leaves the line through $x$ and $x+r$ invariant. So $T_x(x + \lambda r) = x + \lambda r$ for all $\lambda \in \R$. Let $y\in C^\circ$ and write $y=\sum_{k=1}^n\lambda_k r_k$ where  $r_k$ is an extreme vector of $C$ with $M(r_k/x) =1$  for all $k$. Then  
\[
nx + y = \sum_{k=1}^n(x+\lambda_k r_k) = \sum_{k=1}^nT_x(x + \lambda_k r_k) = \sum_{k=1}^n T_x x + T_x \left( \sum_{k=1}^n \lambda_k r_k \right) = nx+S^2_x(y),
\]
hence $S^2_x(y)=y$. 

To show that $x$ is the unique fixed point of $S_x$ assume by way of contradiction that  $y\neq x$  is a fixed point of $S_x$. Then $y$ must be linearly independent from $x$, as $S_x$ is homogeneous of degree $-1$. Consider the open $d_T$-ball with centre $y$,  
\[
B=\{z\in C^\circ\colon d_T(z,y)<d_T(x,y)\}.
\]
It is known  that $B$ is open for the order unit norm and convex, see \cite[Corollary 2.5.6 and Lemma 2.6.2]{LNBook}. Note that $S_x(B)\subseteq B$, since $y$ is a fixed point of $S_x$, and that $x$ is in the boundary of $B$. The Hahn-Banach separation theorem implies the existence of a continuous linear functional $\psi$ on $V$ and $\alpha\geq 0$ such that $\psi(x)=\alpha$ and $\psi(z)>\alpha$ for all $z\in B$. The continuity of $\psi$ implies that $\psi\circ S_x$ is Gateaux differentiable. Now, if we compute the derivative of $\psi\circ S_x$ at $x$ in the direction of $y-x$, we find that
\[
D(\psi\circ S_x)(x)(y-x)=\lim_{t\to 0}\frac{\psi(S_x(x+t(y-x)))-\psi(S_x(x))}{t}=\lim_{t\downarrow 0}\frac{\psi(S_x(x+t(y-x))) -\alpha}{t}\ge 0.
\]
However, if we use Lemma~\ref{L:DS_x=-Id}$(iii)$ and the continuity of $\psi$ again, we also have 
\[
D(\psi\circ S_x)(x)(y-x)=\psi(DS_x(x)(y-x))=\psi(x-y)<0, 
\]
which is absurd. We conclude that $S_x$ is a $d_T$-symmetry with unique fixed point $x$.

The final assertion is a direct consequence of the first one. 
\end{proof}
\begin{remark}
We like to point out that in \cite{Chu2} Chu studied Finsler symmetric cones that are linearly homogeneous and in which the symmetries are bi-analytic maps on $C^\circ$. He showed that such cones correspond to cones in JB-algebras. The smoothness assumptions on the symmetries and the fact that the cone is linearly homogeneous allow one to use methods from Lie theory that cannot be used here, as no smoothness assumptions are made on the gauge-reversing map. 
\end{remark}

In the remainder of this section we show that the $d_T$-symmetries, $S_x\colon C^\circ\to C^\circ$ for $x\in C^\circ$, can be used to show that the cone $C$ is homogeneous. To begin we note that if $x,y\in C^\circ$, then $S_y\circ S_x\colon C^\circ\to C^\circ$ is a homogeneous degree 1 order-isomorphism, so by \cite[Theorem~B]{Schaf} there exists a linear map $T\in \mathrm{Aut}(C)$ such that $T= S_y\circ S_x$ on $C^\circ$. The idea is to use compositions of these linear automorphisms of $C$ to show that $C^\circ$ is homogeneous. 

\begin{lemma}\label{L:x->x+r by automorphism}
Suppose that $(V,C,u)$ and $(W,K,e)$ are complete order unit spaces such that $V$ is spanned by the atoms of $C$, and let $\Psi\colon C^\circ\to K^\circ$ be a gauge-reversing map. For each $x\in C^\circ$, each extreme vector $r\in C$ with $M(r/x)=1$, and each $\lambda > -1$, there exists a linear map $T\in\mathrm{Aut}(C)$ such that $T x = x+\lambda r$.
\end{lemma}
\begin{proof} 
We will use Lemma \ref{L:DS_x=-Id}(i) and find a $\mu > -1$ such that $S_{x+\mu r}(x) = x+\lambda r$. So, $S_{x+\lambda r}\circ S_{x+\mu r}$ maps $x$ to $x+\lambda r$. It then follows from \cite[Theorem~B]{Schaf} that there exists a linear map $T\in \mathrm{Aut}(C)$ such that $T= S_{x+\lambda r}\circ S_{x+\mu r}$ on $C^\circ$, since the composition is a homogeneous degree 1 order-isomorphism.

Let $\mu = \sqrt{\lambda+1} -1$. It is easy to check that $M(r/x+\mu r) = \frac{1}{1+\mu}$, as $M(r/x) =1$. 
Set $r' = (1+\mu)r$, so $M(r'/x+\mu r) =1$. Now note that by Lemma \ref{L:DS_x=-Id}(i) we have
\[
S_{x+\mu r}(x) = S_{x +\mu r}\Bigl(x+\mu r -\frac{\mu}{1+\mu}r'\Bigr) = x+\mu r +\mu r' = x+(2\mu +\mu^2)r = x+\lambda r, 
\]
which completes the proof.
\end{proof}
The previous lemma has the following consequence.
\begin{corollary}\label{C:x->x+r by automorphism}
Suppose that $(V,C,u)$ and $(W,K,e)$ are complete order unit spaces such that $V$ is spanned by the atoms of $C$, and let $\Psi\colon C^\circ\to K^\circ$ be a gauge-reversing map. If $x,y \in C^\circ$ are such that $y - x = \lambda r$ for some $r \in \mathrm{ext}(C)$, then there exists a linear map $T\in\mathrm{Aut}(C)$ such that $T x = y$.
\end{corollary}
\begin{proof}
 If we write $\lambda r = \lambda M(r/x)M(r/x)^{-1}r := \lambda M(r/x) p$, then $M(p/x) = 1$. Therefore, we must have $\lambda M(r/x) > -1$, as $x-p \in \partial C$ and $x+\lambda M(r/x)p =y \in C^\circ$. Applying Lemma~\ref{L:x->x+r by automorphism} to  $y=x + \lambda M(r/x) p\in C^\circ$ shows that there is a linear map $T \in \mathrm{Aut}(C)$ such that $T x = y$. 
\end{proof}
It now follows immediately from \Cref{C:x->x+r by automorphism} and \Cref{L:char span ext(C) = V} that $C^\circ$ is homogeneous.
\begin{theorem}\label{T:homogeneous cone}
Suppose that $(V,C,u)$ and $(W,K,e)$ are complete order unit spaces such that $V$ is spanned by the atoms of $C$. If there exists a gauge-reversing map $\Psi\colon C^\circ\to K^\circ$, then $C^\circ$ is homogeneous.
\end{theorem}

\section{Self-duality of the cone} \label{sec:sym}
In this section we show that if  there exists a gauge-reversing map $\Psi\colon C^\circ\to K^\circ$, then  $C$ is self-dual with respect to an inner-product on $V$. Here the assumption that $C= C_{sc}$ will play a crucial role. The key idea is to show that the atoms of $C$ correspond exactly to the pure states of $S$ in that case. 
\begin{proposition}\label{P:points of smoothness}
Suppose that $(V,C,u)$ and $(W,K,e)$ are complete order unit spaces such that $V$ is spanned by the atoms of $C$ and $C=C_{sc}$. If $\Psi\colon C^\circ\to K^\circ$ is a gauge-reversing map, then for every atom $p$ of $C$ there is a unique pure state $\psi_p$ such that $\psi_p(p)=1$. Moreover, \begin{equation}\label{E:M and psi}
    \psi_p(x)=M(p/S_u(x))\qquad(x \in \Ci), 
\end{equation}
where $S_u\colon C^\circ\to C^\circ$ is the symmetry at $u$ given by (\ref{sym}). In particular, $u-p$ is a point of smoothness of $C$. Furthermore, for each pure state $\xi$ of $(V,C,u)$ there exists an atom $p$ of $C$ such that $\xi(p)=1$, i.e., $\xi=\psi_p$. 
\end{proposition}
\begin{proof} 
Applying \Cref{T:bijection atoms pure states} taking $S_u$ for $\Psi$ and \Cref{L:DS_x=-Id}, we get a pure state $\psi_p$ of $(V,C,u)$ such that $\psi_p(p)=1$. Moreover, 
$\psi_p(x)=M(p/S_u(x))$ for $x\in C^\circ$. 

To show uniqueness, suppose that $\xi$ is a pure state of $(V,C,u)$ with $\xi(p)=1$. Then $g(x) := \log \xi(x)$ is a singleton Funk Busemann point by Proposition \ref{P:Funk singleton Busemann}. It follows  that  $S_u(g)$ is a singleton reverse-Funk Busemann point by \Cref{L:DS_x=-Id}$(ii)$ and \Cref{T: d_T isom extends to Busemann}, which must be of the form $S_u(g)(x)=\log M(q/x)$ for some atom $q$ of $C$ by Proposition \ref{P:singleton Busemann is ext}. 
It follows from (\ref{isomboundary}) that
\[
\log M(q/x)=S_u(g)(x)=g(S_u(x))=\log \xi(S_u(x)) \mbox{\qquad for all $x\in C^\circ$}.
\]
In particular,  for $p_n=\frac{1}{n}u +(1-\frac{1}{n})p$ we find that $M(q/p_n) =\xi(S_u(p_n)) = \xi (nu +(1-n)p)= 1$ by \Cref{L:DS_x=-Id}$(i)$.  Thus, $q \le p_n$ for all $n$, which implies $q=p$, as $p_n\to p$.  So, $\xi(x) =M(p/S_u(x)) =\psi_p(x)$ for all $x\in C^\circ$, hence $\xi=\psi_p$ and $\psi_p$ is the unique pure state such that $\psi_p(p)=1$. This implies that $u-p$ is a point of smoothness, since $\psi_p$ is the unique supporting state at $u-p$. 

Finally, if $\rho$ is a pure state, then $f(x):=\log \rho(x)$ is a singleton Funk Busemann point by Proposition \ref{P:Funk singleton Busemann}. Reasoning as before, $h(x):=S_u(f)(x)= f(S_u(x))$ is a singleton reverse-Funk Busemann point by \Cref{L:DS_x=-Id}$(ii)$ and \Cref{T: d_T isom extends to Busemann}, which is of the form $h(x)=\log M(p/x)$ for some atom $p$ of $C$ by Proposition \ref{P:singleton Busemann is ext}. As $M(p/x) =\psi_p(S_u(x))$ and $h(x) = \log \rho(S_u(x))$, we see that $\rho$ and $\psi_p$ coincide on $C^\circ$, hence $\rho=\psi_p$, which completes the proof.    
\end{proof} 

 The $d_T$-symmetry $S_u$ at $u$ satisfies the following property for orthogonal atoms of $C$. 

\begin{lemma}\label{L:orthogonal inverse}
Suppose that $(V,C,u)$ and $(W,K,e)$ are complete order unit spaces such that $V$ is spanned by the atoms of $C$ and $C = C_{sc}$. If $\Psi\colon C^\circ\to K^\circ$ is a gauge-reversing map, then for orthogonal atoms $p_1, \dots, p_n$ of $C$ the map $S_u$ given by (\ref{sym}) satisfies
    $$S_u\left(u+\sum_{k=1}^n \lambda_k p_k\right) = u - \sum_{k=1}^n \frac{\lambda_k}{\lambda_k+1} p_k$$ for $u+\sum_{k=1}^n \lambda_k p_k\in C^\circ$.
\end{lemma}

\begin{proof} 
We know from \Cref{T:equivalent conditions atoms}$(i)$ and $(ii)$ that 
$$S_u\left(u+\sum_{k=1}^n \lambda_k p_k\right) = u - \sum_{k=1}^n \frac{\lambda_k}{\lambda_k+1} q_k$$ 
for $u+\sum_{k=1}^n \lambda_k p_k\in C^\circ$, where the atoms $q_k$ of $C$ are the ones from \Cref{L:formula for phi on extreme line} with $\Psi=S_u$ and $x=u$.  It now follows from \Cref{L:DS_x=-Id} that $q_k=p_k$ for all $k$. 
\end{proof}

\begin{proposition}\label{P:pairing of extremes} Suppose that $(V,C,u)$ and $(W,K,e)$ are complete order unit spaces such that $V$ is spanned by the atoms of $C$ and $C=C_{sc}$. If $\Psi\colon C^\circ\to K^\circ$ is a gauge-reversing map, and if $p,q$ are atoms of $C$ with corresponding unique pure states $\psi_p,\psi_q$ such that $\psi_p(p)=1$ and $\psi_q(q)=1$, then $\psi_p(q)=\psi_q(p)$. 
\end{proposition}

\begin{proof}
Let $\psi_p$ and $\psi_q$ be the unique pure states such that $\psi_p(p)=1$ and $\psi_q(q)=1$ from \Cref{P:points of smoothness}. Define $p_n=\frac{1}{n}u+(1-\frac{1}{n})p$ and and $q_m=\frac{1}{m}u+(1-\frac{1}{m})q$ for $m,n\geq 1$. Then $M(p_n/x)\to M(p/x)$ as $n\to\infty$ for all $x\in C^\circ$ by Lemma~\ref{L:M function is continuous}. Analogously, we have $M(q_m/x)\to M(q/x)$ as $m\to\infty$ for all $x\in C^\circ$. Furthermore, using \eqref{E:M and psi}, we obtain
\[
\lim_{m\to\infty}\lim_{n\to\infty}M(p_n/S_u(q_m))=\lim_{m\to\infty}M(p/S_u(q_m))=\lim_{m\to\infty}\psi_p(q_m)=\psi_p(q)
\]
and similarly
\[
\lim_{n\to\infty}\lim_{m\to\infty}M(p_n/S_u(q_m))=\lim_{n\to\infty}\lim_{m\to\infty}M(q_m/S_u(p_n))=\lim_{n\to\infty}\psi_q(p_n)=\psi_q(p).
\]

Let $g_n(m)  = M(p_n/S_u(q_m))$ and $\lim_n g_n =g$, where 
$g(m)= M(p/S_u(q_m))$.  \\
\\
We claim that $\lim_n g_n= g$ is uniform in $m$. To see this, let $\eps>0$. As $p_n\geq p$, we know that $M(p_n/S_u(q_m)) \geq M(p/S_u(q_m)) = g(m)$. Now take $\frac{1}{n} < \eps$ and using that $(\eps-\frac{1}{n})p \le (\eps-\frac{1}{n})u$,
\begin{align*}
{\textstyle\frac{1}{n}u + (1-\frac{1}{n}})p &\le \eps u +(1-\eps)p \\
&\leq    \eps u +(1-\eps)g(m)S_u(q_m) \\
&\leq \eps S_u(q_m) +(1-\eps)g(m)S_u(q_m) \\ 
&=    [\eps(1-g(m)) + g(m)]S_u(q_m) \\
&\leq  [\eps(1-\psi_p(q)) +g(m)]S_u(q_m), 
\end{align*}
as $g(m) = M(p/S_u(q_m))  =\psi_p(q_m)\geq \psi_p(q)$. \\
\\
Thus, $M(p_n/S_u(q_m)) - g(m) \leq \eps(1-\psi_p(q))$ for all $m$. 

It follows that the limits over $n$ and $m$ for $M(p_n/S_u(q_m))$ are interchangeable and converge to the same value by \cite[Theorem~9.16]{Apostol}, hence  $\psi_p(q)=\psi_q(p)$. 
\end{proof}

For an atom $p$ of $C$, we define the linear form on $V$ by 
\begin{equation}\label{B(r,y)} B(p,y):=\psi_p(y) \qquad (y\in V).
\end{equation} 
We obtain a bilinear form on $V$ after further extending $B$ to all of $V$ in the first argument, that is, for $x,y\in V$ and $x=\sum_{k=1}^n\lambda_k p_k$ where $p_1, \dots, p_n$ are atoms of $C$, we put $B(x,y):=\sum_{k=1}^n\lambda_k B(p_k,y)$. Suppose that $x=\sum_{k=1}^n\lambda_k p_k=\sum_{i=1}^m\mu_i q_i$ where $q_1, \dots, q_m$ are atoms of $C$. Now, 
by also writing $y=\sum_{j=1}^l\eta_jr_j$ as a linear combination of atoms of $C$, it follows from Proposition~\ref{P:pairing of extremes} that
\[
B(x,y)=\sum_{k=1}^n\lambda_k B(p_k,y)=\sum_{k=1}^n\lambda_k\sum_{j=1}^l\eta_jB(p_k,r_j)=\sum_{k=1}^n\lambda_k\sum_{j=1}^l\eta_jB(r_j,p_k)=\sum_{j=1}^l\eta_jB(r_j,x),
\]
and similarly, 
\[
\sum_{i=1}^m\mu_iB(q_i,y)=\sum_{j=1}^l\eta_j B(r_j,x).
\]
Thus, $B$ is a well-defined symmetric bilinear form on $V$. It should be clear that for any $x \in C$ the map $B(x,\cdot)$ is positive, and for such bilinear forms we can prove the following more general identities, which hold for positive bilinear maps on partially ordered vector spaces. 

\begin{lemma}\label{L: properties B}
Let $X,Y,Z$ be partially ordered vector spaces, and $B \colon X\times Y \to Z$ be a positive bilinear map. If $-v \le x\le v$ and $-w \le y \le w$, then $-B(v,w) \le B(x,y) \le B(v,w)$. Moreover, if $(X,X_+,v)$ and $(Y,Y_+,w)$ are order unit spaces, and $Z$ is a normed vector lattice (e.g.\ $\R$), then $B$ is bounded with $\norm{B} = \norm{B(v,w)}_Z$.
\end{lemma}

\begin{proof}
The inequalities follow from the fact that $v\pm x, w\pm y\ge 0$, 
\[
0\le B(v-x,w+y)+B(v+x,w-y) = 2B(v,w) - 2B(x,y)
\]
together with 
\[
0\le B(v+x,w+y)+B(v-x,w-y) = 2B(v,w) + 2B(x,y).
\]
Now suppose that $v$ is an order unit of $X$, $w$ is an order unit of $Y$, and $Z$ is a normed vector lattice. If $-v \le x\le v$ and $-w \le y \le w$, then by the above, $|B(x,y)| \leq B(v,w) = |B(v,w)|$ and so $\norm{B(x,y)}_Z \leq \norm{B(v,w)}_Z$. Taking the supremum over all such $x$ and $y$ yields $\norm{B} \leq \norm{B(v,w)}_Z$; combined with $\norm{B} \geq \norm{B(v,w)}_Z$ we obtain the desired conclusion.
\end{proof}

\begin{theorem}\label{T:frames and rank}
    Suppose that $(V,C,u)$ and $(W,K,e)$ are complete order unit spaces such that $V$ is spanned by the atoms of $C$ and $C=C_{sc}$. If $\Psi\colon C^\circ\to K^\circ$ is a gauge-reversing map with $\Psi(u) = e$, and $B$ is the symmetric bilinear form on $V$ as defined above, then for atoms $p_1, \dots, p_n$ of $C$ the following statements are equivalent:
\begin{itemize}
    \item[$(i)$] $B(p_k,p_l)=0$ whenever $k \neq l$;
    \item[$(ii)$] $p_1, \dots, p_n$ are orthogonal.
\end{itemize}
Moreover, a set of orthogonal atoms $p_1, \dots, p_n$ sums to $u$ if and only if it has cardinality $N := B(u,u)$. Furthermore, every collection of orthogonal atoms can be extended to such a maximal set.
\end{theorem}

\begin{proof}

    $(i) \Rightarrow (ii)$: Let $\emptyset \neq I \subsetneq \{1, \dots, n\}$ be such that $\sum_{I} p_i \le u$. Let $j \in \{1, \dots, n\} \setminus I$. Note that $\sum_I p_i \neq u$ as otherwise the fact that $B(p_j, p_i) = 0$ for all $i \in I$ would imply
    \[
    1 = \psi_{p_j}(u) = B(p_j,u) = \sum_IB(p_j, p_i) = 0.
    \]
Define $P := \sum_{I}p_i \neq u$, so $u-P>0$, and further, for $m \ge 1$ define 
\[
P_m:={\textstyle\frac{1}{m}u+(1-\frac{1}{m})(u-P) = u - (1-\frac{1}{m})P}. 
\]
It follows from \eqref{E:M and psi}, \Cref{L:orthogonal inverse} and \Cref{P:pairing of extremes} that
\begin{align*}
M(p_j/P_m) &= \psi_{p_j}(S_u(P_m)) = \psi_{p_j}\left(u + \frac{m-1}{2m-1}P\right) = 1 + \frac{m-1}{2m-1}\psi_{p_j}(P) \\&= 1 + \frac{m-1}{2m-1}\sum_{i \in I}\psi_{p_j}(p_i) = 1 + \frac{m-1}{2m-1}\sum_{i \in I}\psi_{p_i}(p_j) = 1,
\end{align*}
so $p_j\leq P_m$ for all $m$, hence $p_j \le u-P$. It follows that $p_1, \dots, p_n$ are orthogonal.

$(ii) \Rightarrow (i)$: If $p_1 + \dots + p_n \le u$, then for $k \neq l$ it follows from $p_k \le u - \sum_{i \neq k} p_i\le u - p_l$, that $p_k \in \ker\psi_{p_l}$. Hence $B(p_k,p_l) = 0$.

For the second part of the theorem, suppose $p_1, \ldots, p_n$ are orthogonal atoms. If these atoms sum to $u$, then $N = B(u,u) = \sum_{k=1}^n B(u,p_k) = n$. If $p_1+ \dots+ p_n <u$, then let $P' := \sum_{k=1}^n p_k$, and let $q_1, \dots, q_n$ the corresponding atoms of $K$  as in \Cref{T:equivalent conditions atoms}. By \Cref{T:equivalent conditions atoms}$(vi)$ it follows that $Q := q_1 + \dots + q_n \in \partial K$, hence the $w$*-closed face $\{\phi \in S(W) \colon \phi(Q) = 0\}$ of  $S(W)$ is non-empty. So, the Krein-Milman theorem implies that there exists a pure state $\psi$ of $W$ such that $\psi(Q) = 0$. Let $p$ be an atom of $C$ such that $M(p/x) = \psi(\Psi(x))$ for all $x \in C^\circ$ by \Cref{T:bijection atoms pure states}. Then for any $\lambda > -1$ it follows from \Cref{T:equivalent conditions atoms} that 
\[
M\Bigl(p/u + \lambda P'\Bigr) = \psi(\Psi(u + \lambda P')) = \psi\Bigl(e - \frac{\lambda}{\lambda + 1}Q\Bigr) = 1,
\]
so $p \le u + \lambda P'$ and letting $\lambda \downarrow -1$ we find that $p + p_1 + \dots + p_n \le u$ and so the atoms $p, p_1, \dots, p_n$ are orthogonal. Therefore
$$N = B(u,u)\ge \sum_{k=1}^n B(u,p_k) + B(u,p) = n + 1 > n,$$
which shows the second part of the theorem.
\end{proof}

Given an order unit space $(V,C,u)$ as in \Cref{T:frames and rank}, the integer $N = B(u,u)$ will be called the \emph{rank} of $V$, and a set of orthogonal atoms $\{p_1,\dots,p_N\}$ will be called a \emph{frame}, so $u=p_1+\cdots+p_N$.

Note that the linear map $\hat{u}\colon V \to \R$ defined by $\hat{u}(x):=B(x,u)$ is a continuous linear functional by \Cref{L: properties B}. If $x \in C$ is such that $\hat{u}(x) = 0$, then for any pure state $\phi$ of $V$ it follows from \Cref{P:points of smoothness} that $\phi = \psi_p$ for some atom $p$ of $C$. By \Cref{T:frames and rank} we can extend $p$ to a frame $\{p, p_1, \dots, p_{N-1}\}$, and hence 
\[
0 \le \psi_p(x) = B(x,p) \le B(x,p) + \sum_{k=1}^{N-1}B(x,p_k) = B(x,u) = \hat{u}(x) = 0,
\]
so $x = 0$ by the Krein-Milman theorem and the fact that the states determine the ordering on $V$ by \cite[Lemma~1.18]{Alfsen}. Hence $\hat{u}$ is a strictly positive functional. Let 
\[
\Sigma:= \{x \in C\colon \hat{u}(x)=1\}
\]
be the corresponding {\em base} of $C$. 

\begin{lemma}\label{L: inequality in base}
Suppose that $(V,C,u)$ and $(W,K,e)$ are complete order unit spaces such that $V$ is spanned by the atoms of $C$ and $C=C_{sc}$. If $\Psi\colon C^\circ\to K^\circ$ is a gauge-reversing map and $B$ is a bilinear form on $V$ as above and the rank of $V$ is $N$, then for all $x \in \Sigma$ we have that $B(x,x)\ge \frac{1}{N}$. 
\end{lemma}

\begin{proof}
If $x \in \Sigma \cap C^\circ$, then by \Cref{T:homogeneous cone} there is a linear map $T \in \mathrm{Aut}(C)$ such that $Tu=x$. Let $\{q_1,\dots,q_N\}$ be a frame. Then $x=Tu = \sum_{k=1}^N Tq_k =:\sum_{k=1}^N \lambda_k p_k$, where each $\lambda_k>0$ and $p_1, \dots, p_N$ are atoms of $C$. Furthermore, $1=B(x,u) = \sum_{k=1}^N\lambda_k\psi_{p_k}(u) =\sum_{k=1}^n\lambda_k$.   

It follows that
\begin{equation}\label{E:B(x,x)inker u}
B(x,x) = \sum_{i,j=1}^N\lambda_i\lambda_jB(p_i,p_j) \ge \sum_{k=1}^N \lambda_k^2B(p_k,p_k) = \sum_{k=1}^N\lambda_k^2 \ge \frac{1}{N},
\end{equation}
by the Cauchy-Schwarz inequality, as 
\[
1 = \left(\sum_{k=1}^N\lambda_k \right)^2 \le \sum_{k=1}^N\lambda^2_k\sum_{k=1}^N 1 = N\sum_{k=1}^N \lambda_k^2.
\]
As $B$ is continuous by \Cref{L: properties B}, this inequality also holds for general $x \in \Sigma$.
\end{proof}

We will use the strictly positive functional $\hat{u}$ to show that $B$ is positive semi-definite on $V$ by first proving that $B$ is positive semi-definite on $\ker\hat{u}$ and then representing any element of $V$ as a linear combination of an element in $\ker\hat{u}$ and $u$.

\begin{lemma}\label{L: B is pos semi-def on ker u}
Suppose that $(V,C,u)$ and $(W,K,e)$ are complete order unit spaces such that $V$ is spanned by the atoms of $C$ and $C=C_{sc}$. If $\Psi\colon C^\circ\to K^\circ$ is a gauge-reversing map and $B$ is a bilinear form on $V$ as above, then $B$ is positive semi-definite on $\ker\hat{u}$. 
\end{lemma}

\begin{proof}
Let $N$ be the rank of $V$ and let $x \in \ker\hat{u}$. There is an $M>0$ such that $x+Mu \in C^\circ$, and so there also is a $\mu>0$ so that $\mu(x+Mu) \in \Sigma$. Hence, it follows that $\mu(x+Mu)-\frac{1}{N}u \in \ker\hat{u}$ and so $\mu M=\frac{1}{N}$, and we can write 
\begin{equation}\label{E:x in ker u}
x = \mu^{-1}(\mu(x+Mu)-{\textstyle\frac{1}{N}}u).
\end{equation}
That is, every element of $\ker \hat{u}$ is a scalar multiple of the difference of an element in $\Sigma$ and $\frac{1}{N}u$.  Now, it follows from \Cref{L: inequality in base} that
\begin{align*}
\mu^2B(x,x)&=B(\mu(x+Mu)-{\textstyle\frac{1}{N}}u,\mu(x+Mu)-{\textstyle\frac{1}{N}}u) \\&= B(\mu(x+Mu),\mu(x+Mu))-2\frac{1}{N}B(\mu(x+Mu),u)+\frac{1}{N^2}B(u,u) \\&
\ge \frac{1}{N}-2\mu M+\frac{1}{N} =0. \qedhere
\end{align*}
\end{proof}

We can use the fact that $B$ is positive semi-definite on $\ker\hat{u}$ to show that $B$ defines an inner product on $V$.

\begin{proposition}\label{P: B is inner prod}
Suppose that $(V,C,u)$ and $(W,K,e)$ are complete order unit spaces such that $V$ is spanned by the atoms of $C$ and $C=C_{sc}$. If $\Psi\colon C^\circ\to K^\circ$ is a gauge-reversing map and $B$ is the bilinear form on $V$ as above, then $B$ is an inner product on $V$. 
\end{proposition}

\begin{proof}
Let $N$ be the rank of $V$, let $x \in V$, and write $x=y+\lambda u$ where $y \in\ker\hat{u}$. Then it follows from \Cref{L: B is pos semi-def on ker u} that
\begin{align*}
    B(x,x)&=B(y,y)+2\lambda B(y,u)+\lambda^2B(u,u)=B(y,y)+\lambda^2B(u,u)\ge 0,
\end{align*}
and $B$ is positive semi-definite on $V$. Suppose that $x \in V$ is such that $B(x,x)=0$. Then we see that $\lambda=0$ and $x \in \ker\hat{u}$. As in (\ref{E:x in ker u}),  write $x=\mu^{-1}(\mu(x+Mu)-\frac{1}{N}u)$, so  $\mu(x+Mu) \in \Sigma\cap C^\circ$ and $\mu M= 1/N$.  From \Cref{L: inequality in base} we now have that 
\[
0=B(x,x)\geq \mu^{-2}\left(B(\mu(x+Mu),\mu(x+Mu))-\frac{1}{N}\right)\geq 0,
\]
so $B(\mu(x+Mu),\mu(x+Mu))=\frac{1}{N}$. 
Reasoning as in the proof of \Cref{L: inequality in base} there exists a linear map $T\in \mathrm{Aut}(C)$ such that $Tu = \mu(x+Mu)$, so we can write $\mu(x+Mu)=\sum_{k=1}^N\lambda_k p_k$, where $\lambda_k> 0$ with  $\sum_{k=1}^N\lambda_k = 1$ and $p_1, \dots, p_N$ are atoms of $C$. Now each inequality in \eqref{E:B(x,x)inker u} is an equality, and so $\lambda_k=\frac{1}{N}$ for all $1 \le k \le N$, hence $B(p_k,p_l) =0$ for all $k\neq l$. It now follows from  \Cref{T:frames and rank} that $\{p_1, \dots, p_N\}$ is a frame. Therefore $\frac{1}{N}u = \mu(x+Mu) = \mu x+\frac{1}{N}u$, so $x=0$ and $B$ is an inner product on $V$, which completes the proof. 
\end{proof}

Before we show that $C$ is self-dual, we prove that the order unit norm topology and the inner product norm topology on $V$ are the same.

\begin{lemma}\label{L: norm and B top are equal}
Suppose that $(V,C,u)$ and $(W,K,e)$ are complete order unit spaces such that $V$ is spanned by the atoms of $C$ and $C=C_{sc}$. If $\Psi\colon C^\circ\to K^\circ$ is a gauge-reversing map and $B$ is a bilinear form on $V$ as above, then the order unit norm and the inner product norm are equivalent. 
\end{lemma}

\begin{proof}
Let $N$ be the rank of $V$. For $x\in V$ we have $-\|x\|_uu\le x\le \|x\|_uu$, so $-N\|x\|_u^2\le B(x,x)\le N\|x\|_u^2$ by \Cref{L: properties B}. Hence $\sqrt{B(x,x)}\le\sqrt{N}\|x\|_u$. On the other hand, since the pure states are norming on $V$ by the Krein-Milman theorem and \cite[Lemma~1.18]{Alfsen}, and all pure states are of the form $\psi_p = B(\cdot,p)$ for some atom $p$ of $C$ by \Cref{P:points of smoothness}, it follows from the Cauchy-Schwarz inequality that 
\[
|B(x,p)|\le \sqrt{B(x,x)}\sqrt{B(p,p)}=\sqrt{B(x,x)}.
\] 
Hence, $\|x\|_u = \sup\{|B(x,p)| \colon \mbox{$p$ is an atom of $C$}\}\le \sqrt{B(x,x)}$. 
\end{proof}

\begin{theorem}\label{T: cone is self-dual}
Suppose that $(V,C,u)$ and $(W,K,e)$ are complete order unit spaces such that $V$ is spanned by the atoms of $C$ and $C=C_{sc}$. If $\Psi\colon C^\circ\to K^\circ$ is a gauge-reversing map and $B$ is a bilinear form on $V$ as above, then $C$ is self-dual with respect to $B$. 
\end{theorem}

\begin{proof}
If $x \in C$, then $B(\cdot,x)$ is a positive linear functional that is continuous by \Cref{L: properties B}. On the other hand, if $\psi$ is a positive continuous linear functional, then there is an $x \in V$ such that $\psi=B(\cdot,x)$, and since every pure state on $V$ is of the form $\psi_p$ for an atom $p$ of $C$ by \Cref{P:points of smoothness}, it follows that $0\leq \psi(p) = B(p,x) =\psi_p(x)$. By the Krein-Milman theorem and \cite[Lemma~1.18]{Alfsen}, the pure states also determine the ordering on $V$, so that $x \in C$, showing that $C$ is self-dual with respect to $B$.
\end{proof}

\section{Main result}
Collecting the results we can now prove the main theorem. 
\begin{theorem}\label{T:char of JH-algebras}
If $(V,C,u)$ is a complete order unit space such that $V$ is spanned by the atoms of $C$ and $C=C_{sc}$, then the following statements are equivalent:
\begin{itemize}
    \item[$(i)$] There exists a  complete order unit space $(W,K,e)$ and a gauge-reversing map $\Psi\colon \Ci\to K^\circ$.
    \item[$(ii)$] There exists a $d_T$-symmetry $S_x\colon C^\circ\to C^\circ$ at some $x\in C^\circ$.
    \item[$(iii)$] $(C^\circ,d_T)$ is a symmetric Finsler space.
    \item[$(iv)$] There exists an inner product $\langle\cdot,\cdot\rangle$ on $V$ making $V$ a Hilbert space and $C^\circ$ a symmetric cone.
    \item[$(v)$] There exists an inner product $\langle\cdot,\cdot\rangle$ on $V$ and a Jordan product $\bullet$ on $V$ such that $(V,\bullet,\langle\cdot,\cdot\rangle)$ is a JH-algebra with unit $u$ and $C$ equals the set of squares.
\end{itemize}
In particular, the equivalences hold if $(V,C,u)$ is a reflexive order unit space.
\end{theorem}
\begin{proof}
We first show that all statements imply $(i)$ if $(V,C,u)$ is a complete order unit space. We know that if $S_x \colon \Ci\to \Ci$ is a $d_T$-symmetry at $x\in C^\circ$, then it is gauge-reversing by \Cref{T:d_T-symmetry implies order antimorphism}, hence $(ii)$, and also $(iii)$, imply $(i)$.  On the other hand, if there exists an inner product $\langle\cdot,\cdot\rangle$ on $V$ for which is it is a Hilbert space such that $C^\circ$ is a symmetric cone, then it follows from \cite[Theorem~3.1]{chu1} that $V$ is a JH-algebra with cone of squares $C$, unit $u$, and the inner product satisfies $\ip{x \bullet y}{z} = \ip{y}{x \bullet z}$ for all $x,y,z \in V$. By \Cref{T: JH-algebra is JB-algebra}, it follows that $V$ must be a JB-algebra. But if $V$ is a JB-algebra with unit $u$, then the map $x \mapsto x^{-1}$ is clearly homogeneous of degree $-1$ and is antitone by \cite[Lemma~1.31]{AS2} and the functional calculus \cite[Proposition~1.21]{AS2}. Hence, taking the inverse defines a gauge-reversing map on $\Ci$, showing that $(iv)$, and hence also $(v)$, imply $(i)$. 

To show that all statements are equivalent in the case where $V$ is spanned by the atoms of $C$ and $C=C_{sc}$, we note first that it follows from \Cref{T:symmetry from antitone} that $(i)$ implies $(ii)$ and $(iii)$.  Moreover, if $(W,K,e)$ is a complete order unit space and $\Psi\colon \Ci\to K^\circ$ is a gauge-reversing map, then the cone $C$ is homogeneous by \Cref{T:homogeneous cone} and self-dual by \Cref{T: cone is self-dual}. As $V$ is complete for the order unit norm, it is a Hilbert space for the inner product as well by \Cref{L: norm and B top are equal}. Thus, $(i)$ implies $(iv)$.  We already mentioned that $(iv)$ implies $(v)$ by \cite[Theorem~3.1]{chu1}.  

It remains to show that $(i)$ implies each of the other statements when $(V,C,u)$ is a reflexive order unit space. In that case $(V,C,u)$ complete, as it is reflexive. Moreover, if there exists a complete order unit space $(W,K,e)$ and a gauge-reversing map $\Psi \colon C^\circ \to K^\circ$, then $V$ is spanned by the atoms of $C$ by \Cref{T:reflexive frames}, and $C=C_{sc}$, as  $(V,C,u)$ is reflexive. So, the result follows from the previous case.
\end{proof}

We believe that the JB-algebras are precisely the complete order unit spaces $(V,C,u)$ for which there exists an order unit space $(W,K,e)$  and a gauge-reversing map $\Psi\colon \Ci\to K^\circ$. This would provide an elegant characterisation of JB-algebras in purely order theoretic terms. 

\section{Appendix: Semi-continuous functions} 
Recall that for a topological space $X$ a function $f\colon X\to[-\infty,\infty]$ is said to be \emph{lower semi-continuous at $x$} if for every $\lambda<f(x)$ there exists a neighbourhood $U$ of $x$ such that $\lambda<f(y)$ for all $y\in U$. The function $f$ is said to be \emph{lower semi-continuous} if it is lower semi-continuous at every $x\in X$. Note that $f$ is lower semi-continuous if and only if $\{x\in X\colon f(x)>\lambda\}$ is open for all $\lambda\in\mathbb{R}$. Similarly, a function $f\colon X\to[-\infty,\infty]$ is said to be \emph{upper semi-continuous at $x$} if for every $\lambda>f(x)$ there exists a neighbourhood $U$ of $x$ such that $\lambda>f(y)$ for all $y\in U$. The function $f$ is called \emph{upper semi-continuous} if it is upper semi-continuous at all $x\in X$, and $f$ is upper semi-continuous if and only if $\{x\in X\colon f(x)<\lambda\}$ is open for all $\lambda\in\mathbb{R}$. Furthermore, it follows that $f$ is upper semi-continuous if and only if $-f$ is lower semi-continuous. 

Note that the definition of lower semi-continuity for a function $f\colon X\to[-\infty,\infty]$ is equivalent to the following statement. For any $x\in X$ and any $\varepsilon>0$ there is an open neighbourhood $O$ of $x$ such that $f(y)\ge f(x)-\varepsilon$ for all $y\in O$. For an upper semi-continuous function the statement is similar with the inequality reversed.

\begin{lemma}\label{L:lsc liminf}
Let $X$ be a topological space. Then $f\colon X\to[-\infty,\infty]$ is lower semi-continuous if and only if for all $x\in X$ and $x_i\to x$ we have $\liminf_if(x_i)\ge f(x)$.
\end{lemma}

\begin{proof}
Suppose $f$ is lower semi-continuous and let $x\in X$. If $(x_i)_i$ is a net in $X$ such that  $x_i\to x$, then there are two cases to check; $f(x)=-\infty$ and $f(x)\in(-\infty,\infty]$. In the case where $f(x)=-\infty$, we clearly have that $\liminf_if(x_i)\ge f(x)$. If $f(x)>-\infty$, let $\lambda<f(x)$. Since $O:=\{y\in X\colon f(y)>\lambda\}$ is open and contains $x$, there is an index $j$ such that $x_i\in O$ whenever $i\ge j$. Hence $\inf_{i\ge j}f(x_i)\ge\lambda$, so $\liminf_i f(x_i)\ge\lambda$. As $\lambda<f(x)$ was arbitrary, it follows that $\liminf_if(x_i)\ge f(x)$.

Conversely, suppose that $f$ is not lower semi-continuous, then there is a $\lambda\in\mathbb{R}$ such that $O_\lambda:=\{y\in X\colon f(y)>\lambda\}$ is not open. Hence there is $x\in O_\lambda$ and a net $(x_i)_i$ in $X$ satisfying $x_i\to x$ and $x_i\notin O_\lambda$. But then $\liminf_i f(x_i)\le\lambda<f(x)$. 
\end{proof}

\begin{lemma}\label{L: char function lsc}
Let $X$ be a topological space. For $Y\subseteq X$ the characteristic function $\chi_Y$ is lower semi-continuous if and only if $Y$ is open.
\end{lemma}

\begin{proof}
If $\chi_{Y}$ is lower semi-continuous, then $Y=\{x\in X\colon \chi_{Y}(x)>0\}$ is open. Conversely, if $Y$ is open, then 
\[
\{x\in X\colon \chi_{Y}(x)>\lambda\}=
\begin{cases}
X & \mbox{if $\lambda<0$};\\ Y & \mbox{if $0\le\lambda<1$};\\ \emptyset & \mbox{if $\lambda\ge 1$}.
\end{cases}
\]
Hence $\chi_{Y}$ is lower semi-continuous.
\end{proof}

\begin{lemma}\label{L:lsc attains min}
Let $X$ be a compact topological space and let $f\colon X\to[-\infty,\infty]$ be a lower semi-continuous function. Then $f$ attains its minimum.
\end{lemma}

\begin{proof}
If there is a $x\in X$ such that $f(x)=-\infty$, then clearly $f$ attains its minimum. Suppose that $f(x)\in(-\infty,\infty]$ for all $x\in X$. Define for every $n\ge 1$, the open sets 
\[
O_n:=\{y\in X\colon f(y)>-n\}. 
\]
Since $O_n\subseteq O_{n+1}$ for all $n\ge 1$ and $\bigcup_{n\ge 1}O_n$ is an open cover of $X$, the compactness of $X$ implies that $X=O_m$ for some $m\ge 1$. Let $\alpha:=\inf_X f(x)$, and let $(x_i)_i$ be a net such that $f(x_i)\to\alpha$. Let $(x_j)_j$ be a subnet converging to some $x$, and note that Lemma~\ref{L:lsc liminf} implies $f(x)\le\liminf_jf(x_j)=\alpha$. Hence $f(x)=\alpha$ and $f$ attains its minimum.
\end{proof}

\begin{lemma}\label{L:product lsc+pos=lsc}
Let $X$ be a topological space. If $f,g\colon X\to[0,\infty)$ are upper semi-continuous, then so is the product $fg$.
\end{lemma}

\begin{proof}
Let $x\in X$ and $\varepsilon>0$. Define $\delta>0$ such that $\delta(f(x)+g(x))+\delta^2<\varepsilon$. Then there are open neighbourhoods $O$ and $V$ of $x$ such that $f(y)<f(x)+\delta$ whenever $y\in O$ and $g(y)<g(x)+\delta$ whenever $y\in V$. Hence $f(y)g(y)<f(x)g(x)+\varepsilon$ whenever $y\in O\cap V$.
\end{proof}

\begin{lemma}\label{L:Dini u.s.c.}
Let $X$ be topological space and let $(f_\alpha)_\alpha$ be an almost non-increasing net of real-valued functions on $X$. 
\begin{itemize}
    \item[$(i)$] Then $f_\alpha$ converges pointwise to a function $f\colon X\to [-\infty,\infty)$.
    \item[$(ii)$] If $f_\alpha$ is upper semi-continuous for all $\alpha$, then so is the pointwise limit $f$.
    \item[$(iii)$] If $X$ is compact and $f_\alpha$ is upper semi-continuous for all $\alpha$, then $\sup_X f_\alpha$ converges to $\sup_X f$.
\end{itemize}
\end{lemma}
\begin{proof}
Let $x\in X$ and let $\varepsilon>0$. Then there is an index $\beta$ such that $f_\alpha\ge f_{\alpha'}-\varepsilon$ whenever $\beta\le\alpha\le\alpha'$. It follows that for $\alpha\ge \beta$ we have $f_\alpha(x)\ge \sup_{\alpha\le \alpha'}f_{\alpha'}(x)-\varepsilon$. Hence
 \[
 \liminf_\alpha f_\alpha(x)\ge \inf_{\beta\le\alpha}f_\alpha(x)\ge\limsup_\alpha f_\alpha(x)-\varepsilon,
 \]
so $\liminf_\alpha f_\alpha(x)=\limsup_\alpha f_\alpha(x)$ and therefore the net $(f_\alpha)_\alpha$ converges pointwise to a function $f$. 

Suppose $f_\alpha$ is upper semi-continuous for all $\alpha$. Let $(x_\beta)_\beta$ be a net in $X$ that converges to $x$ in $X$. Then by Lemma~\ref{L:lsc liminf} we have that $\limsup_\beta f_\alpha(x_\beta)\le f_\alpha(x)$ for each $\alpha$. Let $\varepsilon>0$. Since $(f_\alpha)_\alpha$ is almost non-increasing and $f_\alpha$ converges pointwise to $f$, there is an index $\alpha_0$ such that $f\le f_{\alpha}+\varepsilon$ and $f_{\alpha}(x)\le f(x)+\varepsilon$ whenever $\alpha\ge\alpha_0$. Now for all $\alpha\geq \alpha_0$,  
\[
\limsup_\beta f(x_\beta)\le \limsup_\beta f_{\alpha}(x_\beta)+\varepsilon\le f_{\alpha'}(x)+\varepsilon\le f(x)+2\varepsilon,
\]
which shows that $f$ is upper semi-continuous, since $\varepsilon>0$ was arbitrary. 

If $X$ is compact, then for each $\alpha$ the function $f_\alpha$ attains its maximum at say $x_\alpha$ by Lemma~\ref{L:lsc attains min}. Let $(f_\beta(x_\beta))_\beta$ be a subnet of $(f_\alpha(x_\alpha))_\alpha$. Since $X$ is compact $(x_\beta)_\beta$ has a convergent subnet that converges to say $x$ in $X$, and since $(f_\alpha)_\alpha$ is almost non-increasing, there is an index $\beta_0$ such that 
\[
f_\beta(x)\le f_\beta(x_\beta)\le f_{\beta_0}(x_{\beta_0})+1
\]
whenever $\beta_0\le\beta$. Hence we may pass to a subnet $(f_\gamma(x_\gamma))_\gamma$ such that $(x_\gamma)_\gamma$ converges to $x$ and $(f_\gamma(x_\gamma))_\gamma$ converges to say $\lambda$. Let $\varepsilon>0$. Then there is an index $\gamma_0$ such that $f_{\gamma'}(x_{\gamma'})\le f_\gamma(x_{\gamma'})+\varepsilon$ whenever $\gamma_0\le\gamma\le\gamma'$. Taking the upper limit for $\gamma\le\gamma'$ yields $\lambda\le f_\gamma(x)+\varepsilon$ by Lemma~\ref{L:lsc liminf} and subsequently taking the limit for $\gamma$ yields $\lambda\le f(x)+\varepsilon\le\sup_X f+\varepsilon$. As $\varepsilon>0$ was arbitrary, we have $\lambda\le\sup_X f$. Conversely, for any $y\in X$, it follows that 
\[
f(y)=\lim_\gamma f_\gamma(y)\le \lim_\gamma f_\gamma(x_\gamma)=\lambda,
\]
so $\sup_X f\le\lambda$. We have shown that every subnet of $(f_\alpha(x_\alpha))_\alpha$ has a subnet that converges to $\sup_X f$, hence $\sup_X f_\alpha$ converges to $\sup_X f$.
\end{proof}

\paragraph{Acknowledgement} We would like to thank Samuel Tiersma for his suggestions and ideas for the proof of \Cref{T:bijection atoms pure states} and \Cref{T:reflexive frames}, and Cormac Walsh for his comments on an earlier version of the paper.


\end{document}